\documentclass{amsart}
\usepackage{amsmath}
\usepackage{amssymb,amscd}
\usepackage{pstricks}
\usepackage{tkz-graph}
\usepackage{amsthm}
\usepackage{multicol}
\usepackage{color}
\usepackage{hyperref}

\numberwithin{equation}{section}

\newtheorem{theorem}{Theorem}[section]
\newtheorem{proposition}[theorem]{Proposition}
\newtheorem{corollary}[theorem]{Corollary}
\newtheorem{lemma}[theorem]{Lemma}

\newtheorem{fact}[theorem]{Fact}
\newtheorem{conjecture}[theorem]{Conjecture}
\newtheorem{question}[theorem]{Question}

\theoremstyle{definition}
\newtheorem{example}[theorem]{Example}
\newtheorem{remark}[theorem]{Remark}
\newtheorem*{acknowledge}{Acknowledgments}

\DeclareMathOperator{\Inv}{Inv}
\DeclareMathOperator{\Sph}{Sph}

\DeclareMathOperator{\End}{End}

\DeclareMathOperator{\Span}{Span}

\DeclareMathOperator{\Isom}{Isom}
\DeclareMathOperator{\inv}{inv}
\DeclareMathOperator{\sph}{sph}

\DeclareMathOperator{\LLC}{LLC}

\definecolor{olive}{rgb}{0.5,0.5,0.0}

\newcommand{\ep}{\varepsilon}

\newcommand{\msf}{\mathsf}
\newcommand{\mbb}{\mathbb}

\newcommand{\mH}{\mathbb{H}_\mathsf{K}}
\newcommand{\mfO}{\mathsf O}

\newif\ifshowflag
\showflagtrue   
\newif\ifshowqstn
\showqstntrue   
\showqstnfalse 
\newif\ifshowinfo
\showinfotrue   
\showinfofalse 

\newcommand{\eqx}{\simeq}   

\renewcommand{\phi}{\varphi}
\renewcommand{\emptyset}{\varnothing}

\DeclareMathOperator{\diam}{Diam}
\DeclareMathOperator{\dist}{dist}

\DeclareMathOperator{\Mob}{M\ddot{o}b}

\newcommand{\mathfont}{\mathsf} 

\newcommand{\mfC}{{\mathfont C}}      
\newcommand{\mfH}{{\mathfont H}}      
\newcommand{\mfN}{{\mathfont N}}      
\newcommand{\mfR}{{\mathfont R}}      
\newcommand{\mfZ}{{\mathfont Z}}      
\newcommand{\mfK}{{\mathfont K}}      

\catcode`\@=11       

\def\rf#1{\@rf{#1}#1:;;}
\def\rfs#1{\@rfs{#1}#1:;;}
\def\rfm#1{\@rfF#1<>;;}

\def\@C{C}\def\@CC{CC}\def\@E{E}\def\@F{F}\def\@L{L}\def\@P{P}\def\@Q{Q}
\def\@R{R}\def\@K{K}\def\@S{S}\def\@T{T}\def\@TT{TT}\def\@X{X}\def\@s{s}\def\@f{f}\def\@A{A}

\def\@rf#1#2:#3;;{\def\@b{#2}
  \ifx\@b\@C Corollary~\ref{#1}\else%
  \ifx\@b\@E (\ref{#1})\else
  \ifx\@b\@F Fact~\ref{#1}\else%
  \ifx\@b\@L Lemma~\ref{#1}\else%
  \ifx\@b\@P Proposition~\ref{#1}\else%
  \ifx\@b\@Q Question~\ref{#1}\else%
  \ifx\@b\@R Remark~\ref{#1}\else%
  \ifx\@b\@K Conjecture~\ref{#1}\else%
  \ifx\@b\@S Section~\ref{#1}\else%
  \ifx\@b\@A Appendix~\ref{#1}\else%
  \ifx\@b\@T Theorem~\ref{#1}\else%
  \ifx\@b\@TT Theorem~\ref{#1}\else%
  \ifx\@b\@X Example~\ref{#1}\else%
  \ifx\@b\@s Subsection~\ref{#1}\else
  \ifx\@b\@f Figure~\ref{#1}\else%
  \ref{#1}\fi\fi\fi\fi\fi\fi\fi\fi\fi\fi\fi\fi\fi\fi\fi}

\def\@rfs#1#2:#3;;{\def\@b{#2}
  \ifx\@b\@C Corollaries~\ref{#1}\else%
  \ifx\@b\@CC Corollaries~\ref{#1}\else%
  \ifx\@b\@F Facts~\ref{#1}\else%
  \ifx\@b\@L Lemmas~\ref{#1}\else%
  \ifx\@b\@P Propositions~\ref{#1}\else%
  \ifx\@b\@Q Questions~\ref{#1}\else%
  \ifx\@b\@R Remarks~\ref{#1}\else%
  \ifx\@b\@S Sections~\ref{#1}\else%
  \ifx\@b\@T Theorems~\ref{#1}\else%
  \ifx\@b\@TT Theorems~\ref{#1}\else%
  \ifx\@b\@X Examples~\ref{#1}\else%
  \ifx\@b\@s \S\ref{#1}\else
  \ifx\@b\@f Figures~\ref{#1}\else%
  \ref{#1}\fi\fi\fi\fi\fi\fi\fi\fi\fi\fi\fi\fi\fi}

\def\@rfF<#1>#2;;{\def\@c{#2}
  \@rfs{#1}#1:;;\ifx\@c\empty\else\@rfL:#2;;\fi}

\def\@rfL:#1<#2>#3;;{\def\@b{#2}\def\@c{#3}
  #1\ifx\@b\empty\else\ref{#2}\ifx\@c\empty\else\@rfL:#3;;\fi\fi}

\catcode`\@=12       

\def\vint_#1{\mathchoice
          {\mathop{\vrule width 6pt height 3 pt depth -2.5pt
                  \kern -8pt \intop}\nolimits_{\kern -4pt#1}}%
          {\mathop{\vrule width 5pt height 3 pt depth -2.6pt
                  \kern -6pt \intop}\nolimits_{#1}}%
          {\mathop{\vrule width 5pt height 3 pt depth -2.6pt
                  \kern -6pt \intop}\nolimits_{#1}}%
          {\mathop{\vrule width 5pt height 3 pt depth -2.6pt
                   \kern -6pt \intop}\nolimits_{#1}}}

\thanks{E.L.D. was partially supported by the Academy of Finland (grant
288501
`\emph{Geometry of subRiemannian groups}')
and by the European Research Council
 (ERC Starting Grant 713998 GeoMeG `\emph{Geometry of Metric Groups}').
}

\title[Invertible Homogeneous Metric Spaces]{
Toward a quasi-M\"obius characterization of \\ Invertible Homogeneous Metric Spaces
}
\date{\today}   

\author{David Freeman}
\author{Enrico Le Donne}
\address{%
\footnotesize
\begin{tabular}{lr}
\textsf{University of Cincinnati Blue Ash College}\\
\textsf{Department of Math, Physics, and Computer Science}\\
\textsf{9555 Plainfield Road, Cincinnati, Ohio 45236, United States}
\end{tabular}
\normalsize}
\email{david.freeman@uc.edu}
\address{%
\footnotesize
\begin{tabular}{lr}
\textsf{University of Jyv\"askyl\"a}\\
\textsf{Department of Mathematics and Statistics}\\
\textsf{P.O. Box (MaD), FI-40014, Finland}
\end{tabular}
\normalsize}
\email{ledonne@msri.org}

\begin{document} 


\begin{abstract}
We study locally compact metric spaces that enjoy various forms of homogeneity with respect to M\"obius self-homeomorphisms. We investigate connections between such homogeneity and the combination of isometric homogeneity with invertibility. In particular, we provide a new characterization of snowflakes of boundaries of rank-one symmetric spaces of non-compact type among locally compact and connected metric spaces. Furthermore, we investigate the metric implications of homogeneity with respect to uniformly strongly quasi-M\"obius self-homeomorphisms, connecting such homogeneity with the combination of uniform bi-Lipschitz homogeneity and quasi-invertibility. In this context we characterize spaces containing a cut point and  provide several metric properties of spaces containing no cut points. These results are motivated by a desire to characterize the snowflakes of boundaries of rank-one symmetric spaces up to bi-Lipschitz equivalence. 
\end{abstract}

\maketitle

\section{Introduction}\label{S:intro}                 

This paper contributes to the metric characterization of boundaries of rank-one symmetric spaces of non-compact type. Such symmetric spaces are Gromov hyperbolic and therefore possess boundaries at infinity, which we view as metric spaces equipped with visual distances. Such distances are non-Riemannian, unless the symmetric space is real hyperbolic. On this subject there have been several contributions: \cite{
MR1114460,
Bourdon95, Bourdon, 
FLS07, 
Foertsch_Lytchak_Schroeder_ERR,
Foertsch-Schroeder,  
Foertsch_Schroeder12,
Platis-ptolemy,
Buyalo_Schroeder14,
Buyalo_Schroeder15,
PS17}.
%

In the present paper, we focus on the fact that boundaries at infinity of rank-one symmetric spaces of non-compact type (and their snowflakes) enjoy an abundance of metric homogeneity. In fact, when equipped with a visual distance with base point at infinity they are isometrically homogeneous and admit an inversion (as defined below). Furthermore, the compact boundary is $2$-point M\"obius homogeneous. One of our main results provides a metric characterization of such spaces in terms of these properties (see \rf{T:sharp}).

Looking beyond characterizations up to isometric equivalence, we also work towards a characterization of snowflakes of boundaries of non-compact rank-one symmetric spaces up to bi-Lipschitz equivalence. Thus we investigate spaces that are uniformly bi-Lipschitz homogeneous and admit quasi-inversions (see \rf{S:terminology:intro} for relevant definitions). We point out that our results in this area fit into a progression of ongoing study. For example, the authors of the present paper have previously studied bi-Lipschitz homogeneity in the context of curves, surfaces, and more general spaces (see \cite{FH-blh,Freeman-blh,LeDonne-blh,LeDonne-geodesic}). 
Quasi-invertibility and bi-Lipschitz homogeneity have been studied in  \cite{BB05,BHX-inversions,DCL17, Freeman-iiblh,Freeman-invertible}. Our main goal in continuing this line of investigation is to uncover metric and geometric implications of such homogeneity and/or invertibility in a rather general metric setting. We articulate a specific version of this goal in \rf{K:coarse1} below. Before stating this conjecture, we explain a bit of terminology.

We  denote by $\mathcal{A}$ the collection of normed division algebras $\{\msf{R}, \msf{C}, \msf{H}, \msf{O}\}$. Here $\msf{R}$ denotes the real numbers, $\msf{C}$ the complex numbers, $\msf{H}$ the quaternions, and $\msf{O}$ the octonions. Using this notation, we write $\mbb{H}_\mfK^n$ (with $ \mfK\in\mathcal{A}$ and $n\in \msf N$) to denote the Lie group obtained as the stereographic projection of the visual boundary of the $\mfK$-hyperbolic space
of dimension $n+1$ over $\mfK$.
 We call $\mathbb{H}_\mfK^n$ the \textit{$n$-th $\mfK$-Heisenberg group}, and denote by $\rho$ its visual distance with base point at infinity; see \eqref{E:Ham} for the formal definition. 
 Classically, it is known that these metric spaces are isometrically homogeneous and invertible. In particular, they are  bi-Lipschitz homogeneous and quasi-invertible in the sense of Section~\ref{sec:coarse:term}.
 With this terminology in hand we state the following conjecture.

\begin{conjecture}\label{K:coarse1}
A metric space $X$ is  bi-Lipschitz equivalent to $(\mbb{H}_\mfK^n,\rho^\alpha)$, for some $\alpha \in (0,1]$,  if and only if  $X$ is locally compact, connected, bi-Lipschitz homogeneous, and quasi-invertible.
\end{conjecture}

A significant issue in the study of ideal boundaries is that, in general, the boundary of a Gromov-hyperbolic space is not connected. Even if the boundary is connected it may not contain any non-degenerate rectifiable curves, thus precluding the possibility of any geometric analysis. In this connection we remind the reader that any snowflake of a visual distance remains a visual distance which allows for no non-degenerate rectifiable curves.  In \rf{T:main:connected}, under the aforementioned homogeneity assumptions, we consider a dichotomy: either the space contains a cut point, or it does not. In the first case, we show that such a space is
bi-Lipschitz homeomorphic to a snowflake of the Euclidean line. Thus we provide a complete characterization of spaces containing a cut point. In the second case, when the space does not contain a cut point, we prove several properties that are useful for developing analysis on these metric spaces and point in the direction of \rf{K:coarse1}.

\subsection{Results}\label{s:results}
Here we summarize the main results of the present paper. 

\subsubsection{M\"obius homogeneity}
We first present results addressing various forms of M\"obius homogeneity in connected and locally compact metric spaces. In the following we denote by  $\hat X:=X\cup\{\infty\} $ the compactification of $X$ equipped with its natural  M\"obius structure, see \rf{S:terminology:intro}.

\begin{theorem}\label{T:sharp}
Suppose $X$ is an unbounded, locally compact, complete, and connected metric space. For $\mathcal{A}=\{\msf{R}, \msf{C}, \msf{H}, \msf{O}\}$, the following statements are equivalent:
\begin{enumerate}
	\item{For some $n\in\msf{N}$, $\mfK\in\mathcal{A}$, and $\alpha\in(0,1]$, the space $X$ is M\"obius homeomorphic to $(\mbb{H}_\mfK^n,\rho^\alpha)$.}
	\item{For some $n\in\msf{N}$,  $\mfK\in\mathcal{A}$,   and $\alpha\in(0,1]$, the space $X$ is isometric to $(\mbb{H}_\mfK^n,  \rho^\alpha)$.}
	\item{The space $X$ is isometrically homogeneous and invertible.}
	\item{The space   $\hat X$ is $2$-point M\"obius homogeneous.}
\end{enumerate}
\end{theorem}

The main content of the above theorem is the implication (3)$\implies$(2). 
The style and conclusion of this result is similar to the main result of \cite{Buyalo_Schroeder14}:
Let $X$ be a compact Ptolemy space possessing at least one Ptolemy circle and for which there exists a unique {\em space inversion} with respect to any two distinct points $\omega$, $\omega'$ of $X$ and any sphere between $\omega$, $\omega'$. Under these assumptions, the space $X$ is M\"obius equivalent to the boundary at infinity of a rank-one symmetric space of non-compact type taken with the canonical M\"obius structure. 
Amidst the apparent similarities between \rf{T:sharp} and the result of \cite{Buyalo_Schroeder14}, we point out some key differences. A space inversion (in the sense of \cite{Buyalo_Schroeder14}) must be an involution that is fixed point free. The inversions of Theorem~\ref{T:sharp} need not possess these properties. Moreover, we  do not assume presence of a Ptolemy circle.

Given a Carnot group $\mbb{G}$ equipped with a Carnot-Carath\'eodory distance $d$ (or any comparable distance), the compactification $\hat{ \mbb{G}}$
  may exhibit unique geometry at the point at $\infty$. In particular, it may be the case that certain classes of mappings on $\hat {\mbb{G}}$ must fix this point; see, for example, \cite[page 249]{Freeman-invertible}. Or, for another example along this line, note the \textit{Pointed Sphere Conjecture} of Yves de Cornulier recorded in \cite[Conjecture 19.104]{Cornulier18}. In light of these observations, the homogeneity assumptions on the sphericalization of $X$ in \rf{T:sharp} may be seen as a natural way to restrict the geometric (and resulting algebraic) structure of $X$ itself.
  
Conjecture~\ref{K:coarse1} is open even in the case that the space $X$ is a Carnot group. In particular, it is an open question whether the complexified   Heisenberg group $(\mbb{H}_\mfC^1)_\mfC$ and the direct product   $\mbb{H}_\mfC^1\times \mbb{H}_\mfC^1$ admit a quasi-inversion. However, there is some relevant work by Xie \cite{Xie-non-rigid}.

In order to provide additional context for \rf{T:sharp}, we record the following immediate corollary. This result should be seen as a rephrasing of the  result in \cite{KL16}.

\begin{corollary}\label{C:three_point}
Suppose $X$ is an unbounded, locally compact, and connected metric space. There exists $n\in\msf N$    and $\alpha\in(0,1]$ such that $X$ is isometric to $(\msf{R}^n, |\cdot|^\alpha)$ if and only if 
$\hat X$ is $3$-point M\"obius homogeneous.
\end{corollary}

Indeed, one can verify (via \rf{T:sharp} and results from Section~\ref{s:ihms}) that the $3$-point M\"obius homogeneity of $\hat X$ implies   $2$-point isometric homogeneity of $X$. In other words, we can conclude that, given pairs $\{x,y\}$ and $\{a,b\}$ of distinct points from $X$ such that $d(x,y)=d(a,b)$, there exists $f\in \Isom(X)$ with $f(x)=a$ and $f(y)=b$. The only space $\mathbb{H}_\mfK^n$ whose snowflakes enjoy this property is $\mathbb{H}_\msf{R}^n=\msf{R}^n$. 

In light of \rf{T:sharp} and \rf{C:three_point}, we observe that $2$-point and $3$-point M\"obius homogeneity in locally compact and connected metric spaces provide metric analogues to algebraic results about $2$-point and $3$-point topological homogeneity in such spaces (see, for example,  \cite[Theorem~3.3 and Corollary~3.4]{Kramer-transitive}). 

We also explore the metric consequences of $1$-point M\"obius homogeneity. In order to compensate for this weaker homogeneity assumption we work under stronger connectivity assumptions. Under such assumptions, one can prove that the geometry of the metric space under scrutiny is sub-Riemannian in nature (at least, up to bi-Lipschitz distortion). 

\begin{theorem}\label{QC:Mobius:thm}
Let $X$ be a compact and quasi-convex metric space of finite topological dimension. If $X$ is M\"obius homogeneous, then $X$ is bi-Lipschitz homeomorphic to a sub-Riemannian manifold.
\end{theorem}

We stress that  the boundary of a {\rm CAT}$(-1)$-space is naturally equipped with special visual  distances: either Bourdon distances or Hamenst\"{a}dt distances; see Section~\ref{S:BH}. With such metrics, the boundaries become Ptolemy spaces by \cite{Foertsch-Schroeder}. With this in mind, a consequence of the above theorem is the following.

\begin{corollary}\label{MCircles:Mobius:thm}
Let $X$ be the boundary of a {\rm CAT}$(-1)$-space. If $X$ is M\"obius homogeneous, of finite topological dimension, and connected by M\"obius circles, then $X$ is bi-Lipschitz homeomorphic to a sub-Riemannian manifold.
\end{corollary}

\subsubsection{Uniformly strongly quasi-M\"obius homogeneity}
Having discussed various forms of homogeneity with respect to M\"obius maps, we now present a few results concerning homogeneity with respect to uniformly strongly quasi-M\"obius maps. We refer the reader to Section~\ref{sec:coarse:term} for relevant terminology.
We start with the coarse analogue of the equivalence (3)$\iff$(4) of Theorem~\ref{T:sharp}:

\begin{proposition}\label{P:quasi-version}
A proper and unbounded metric space $X$ is uniformly bi-Lipschitz homogeneous and quasi-invertible if and only if $\hat X$ is 2-point uniformly strongly quasi-M\"obius homogeneous.
\end{proposition}

The reader may consult \rf{P:invertible_char} for additional characterizations of spaces that are both uniformly bi-Lipschitz homogeneous and quasi-invertible.   

For the coarse analogue of the equivalence (1)$\iff$(2) of Theorem~\ref{T:sharp} we show the following. 

\begin{proposition}\label{general_prop}
A homeomorphism $f:X\to Y$ between proper and unbounded metric spaces is strongly quasi-M\"obius if and only if it is bi-Lipschitz. Furthermore, $f$ is M\"obius if and only if $f$ is a similarity.
\end{proposition}

It is straightforward to verify that if $X$ is bi-Lipschitz equivalent to some $(\mbb{H}_\mfK^n,\rho^\alpha)$, then $X$ is uniformly bi-Lipschitz homogeneous and quasi-invertible. Conjecture~\ref{K:coarse1} claims the converse for connected spaces: the coarse analogue of (3)$\implies$(2) of Theorem~\ref{T:sharp}. Our main contribution towards this conjecture  is as follows.

\begin{theorem}\label{T:main:connected}
Suppose $X$ is an unbounded  locally compact  metric space that is uniformly bi-Lipschitz homogeneous, quasi-invertible, and contains an non-degenerate curve. 
\begin{enumerate}
\item The metric space $X$ is path connected, locally path connected, proper, and Ahlfors regular.
	\item{If $X$ has a  cut point, then, for some $\alpha\in(0,1]$, $X$ is bi-Lipschitz homeomorphic to $(\msf{R}, |\cdot|^\alpha)$.}
		\item{If $X$ has no cut points, then $X$ is linearly locally connected. Furthermore,
			\begin{enumerate} 
				\item{If $X$ contains a non-degenerate rectifiable curve, then $X$ is annularly quasi-convex.}
				\item{If $X$  does not contain a non-degenerate rectifiable curve, then, for some $\alpha\in(0,1)$, $X$ is bi-Lipschitz homeomorphic to an $\alpha$-snowflake.}
			\end{enumerate}}
\end{enumerate}
\end{theorem}

We point out that \rf{T:main:connected} affirms \rf{K:coarse1} in the case that $X$ is path connected and contains a cut point.

At the present time we are unable to provide a coarse analogue of Corollary~\ref{C:three_point}. In particular, we do not have answers to the following questions.

\begin{question}
Is the compactified Heisenberg group $\Sph_e({\mbb{H}_\mfC^1})$ 3-point uniformly strongly quasi-M\"obius homogeneous?
\end{question}
\begin{question}
If a metric space is 3-point uniformly strongly quasi-M\"obius homogeneous and homeomorphic to $\msf{S}^n$, is it  bi-Lipschitz homeomorphic to a snowflake of the round $n$-sphere? 
\end{question}
This last question, even in the cases $n=2$ or $3$, seems challenging. It is related to other open problems (such as \cite[Question 5]{HS-questions}) about structures on spheres that are 3-point quasi-symmetrically homogeneous.  

\subsubsection{Disconnected spaces}
Finally, we present two results pertaining to unbounded, proper, and disconnected metric spaces.  

The standard examples of disconnected, isometrically homogeneous, and invertible metric spaces are given by the boundaries at infinity of non-rooted regular trees. Indeed, let $T_N$ denote the $(N+1)$-regular tree equipped with the path distance for which each edge has length 1. The metric space $T_N$ is $\textrm{CAT}(-1)$. Furthermore, there is a notion of geodesic inversion on it, and in this sense $ T_{N} $ is a sort of non-Riemannian symmetric space. Therefore, the parabolic visual boundary $C_{N}:=\partial_\infty T_{N}$, equipped with the parabolic visual distance $\rho_s$ of parameter $s$ (see Section~\ref{Sec:ex}) is disconnected, (2-point) isometrically homogeneous, and invertible. Moreover, it is an ultrametric space. We refer the reader to \cite{BS17-trees}, for example, for more information about the boundaries of trees.

In parallel with \rf{T:sharp}, one might expect that the spaces $(   C_{N}, \rho_s )$, with $N\geq 2$ and $s>1$, are in some sense the only unbounded and disconnected metric spaces possessing the above homogeneity and invertibility properties. The following theorem tells us that this is indeed the case, up to bi-Lipschitz homeomorphisms. We refer the reader to \rf{S:disconnected} for relevant definitions.

\begin{theorem}\label{T:sharp:nonconnected}
Suppose $X$ is a disconnected, unbounded, locally compact, and isometrically homogeneous metric space. If $X$ is invertible, then there exists a positive integer $N\geq 2$ and $s>1$ such that $X$ is bi-Lipschitz homeomorphic to $(   C_{N}, \rho_s )$.
\end{theorem}

\rf{T:sharp:nonconnected} is sharp in the sense that, in general, a space $X$ satisfying the assumptions of \rf{T:sharp:nonconnected} might not be isometric to any $(C_{N}, \rho_s)$. This is demonstrated in \rf{X:counter}.  Furthermore, in parallel with \rf{C:three_point}, we have the following characterization.

\begin{theorem}\label{T:sharp:nonconnected:2pt}
Suppose $X$ is a disconnected, unbounded, and locally compact metric space. There exists a positive integer $N\geq 2$ and $s>1$ such that $X$ is isometric to $(C_N,\rho_s)$ if and only if $\hat{X}$ is three-point M\"obius homogeneous.
\end{theorem}
 
In the next (and last) theorem, we demonstrate that the structure of the boundary of a tree can still be recovered under the weaker assumptions of uniform bi-Lipschitz homogeneity and quasi-invertibility, at least up to quasi-M\"obius homeomorphisms.

\begin{theorem}\label{T:main:connected:NOT}
Suppose $X$ is a disconnected, unbounded, locally compact, and uniformly bi-Lipschitz homogeneous metric space. If $X$ is quasi-invertible, then $X$ is quasi-M\"obius homeomorphic to $(C_{2}, \rho_2 )$. 
\end{theorem}

\subsubsection{Structure of the paper}
The remainder of the Introduction provides terminology and notation for use throughout the paper. In \rf{S:isometric}, we investigate the metric geometry of space that are both isometrically homogeneous and invertible. We also study certain metric Lie groups, and provide a proof of \rf{T:sharp}. In \rf{S:strong}, we use results of Montgomery-Zippin pertaining to the structure of locally compact groups to prove Theorem~\ref{QC:Mobius:thm} and Corollary~\ref{MCircles:Mobius:thm}. In \rf{S:BLH}, we study spaces that are uniformly bi-Lipschitz homogeneous and quasi-invertible. In particular, we investigate the relationship between quasi-invertibility and quasi-dilation invariance. We also prove \rfs{P:quasi-version} and \ref{general_prop}. Before proving \rf{T:main:connected}, we provide additional characterizations of quasi-invertibility under the assumption of uniform bi-Lipschitz homogeneity in \rf{P:invertible_char}. In \rf{S:disconnected}, we prove a dichotomy between connectedness and uniform disconnectedness for uniformly bi-Lipschitz homogeneous and quasi-invertible spaces (see \rf{L:uniformly_disconnected}). Then we illustrate examples of disconnected homogeneous invertible spaces. Finally, we prove \rfs{T:sharp:nonconnected}, \ref{T:sharp:nonconnected:2pt}, and \ref{T:main:connected:NOT}.

\begin{acknowledge}
This research was initiated and (mostly) completed during visits by the first author to the University of Jyv\"askyl\"a; he thanks the University for its hospitality. We would also like to acknowledge Mario Bonk and Gareth Speight for providing helpful discussions and feedback on previous drafts of this manuscript. Furthermore, we thank Guy C. David for suggesting the use of Laakso's line-fitting property in our proof of \rf{T:main:connected}, and Viktor Schroeder for suggesting Proposition~\ref{Moebius QC}. 
\end{acknowledge}

\tableofcontents

\subsection{Terminology}\label{S:terminology:intro}

We now explain the terminology used in this introduction. In this paper, given a point $p$ in a set $X$, we define 
\[X_p:=X\setminus\{p\} \qquad\text{ and } \qquad \hat{X}:=X\cup\{\infty\}.\] 
Thus $\hat{X}_p=(X \cup\{\infty\})\setminus\{p\}$.

We also make use of the following standard metric-space notation. Given $x\in X$ and $r>0$, we write $B(x;r)$ to denote the open ball $\{z\in X\,|\,d(x,z)<r\}$ centered at $x$ of radius $r$. Given $R\geq r$, We write $A(x;r,R)$ to denote the open annulus $\{z\in X\,|\,r<d(x,z)<R\}$ centered at $x$. Given a subset $U$ of a topological space, we write $\overline{U}$ to denote its topological closure.

\subsubsection{Inversions, sphericalizations, and M\"obius maps}

We say that a metric space $X$ is {\em invertible} provided it is unbounded and it admits an {\em inversion} at some point $p\in X$. That is, there exists a homeomorphism $\tau_p:X_p\to X_p$ such that, for $x,y\in X_p$, we have
\[d(\tau_p(x),\tau_p(y))=\frac{d(x,y)}{d(x,p)d(y,p)}.\]

Inversions are closely related to the concept of the \textit{inverted space} of $X$ at $p$, denoted as $\Inv_p(X)$. This inverted space $\Inv_p(X)$ is given by $(\hat{X}_p,i_p)$, where $i_p$ is the quasi-distance defined by
\begin{equation}\label{def:inv}
i_p(x,y):=\frac{d(x,y)}{d(x,p)d(y,p)}  \text{ \quad and \quad} i_p(x,\infty):=\dfrac{1}{d(x,p)},  \quad\text{ for }  x,y\in {X}_p.\end{equation}
Here we use the term \textit{quasi-distance} to describe a positive definite and symmetric function $\delta$ on a product $Z\times Z$ such that, for any $x,y,z\in Z$, we have $\delta(x,y)\leq C(\delta(x,z)+\delta(z,y))$; see \cite[Section 3.8]{BHX-inversions} for further discussion of the quasi-distance $i_p$. Quasi-distances are sometimes referred to as \textit{quasi-metrics} in the literature; thus we refer to $\Inv_p(X)$ as a quasi-metric space. 

Following {\cite{Foertsch-Schroeder}}, we say that a metric space $X$ is a {\em Ptolemy space} if \textit{Ptolemy's inequality} holds. That is, for all $w\in X$, the function $i_p$ from \eqref {def:inv} is a distance; i.e., it satisfies the triangle inequality (cf.\,\cite[Remark 2.6]{Foertsch-Schroeder}). Observe that an inversion $\tau_p$ is an isometry from $\Inv_p(X)$ onto $X$, where we take $\tau_p(\infty)=p$. In particular, the existence of an inversion on $X$ implies that $X$ is a Ptolemy space. 

Given a point $p\in X$, and $x,y\in X$, we define
\begin{equation}
s_p(x,y):=\frac{d(x,y)}{(1+d(x,p))(1+d(y,p))} 
\text{ \quad and \quad}
s_p(x,\infty) := \frac{1}{1+d(x,p)},\end{equation} 
and we call $\Sph_p(X)=(\hat X, s_p)$ the {\em sphericalized space} of $X$ at $p$. In general, the function $s_p$ is a quasi-distance. The topology induced by $s_p$ on $\hat{X}$ agrees with the one-point compactification of $X$ when $X$ is non-compact and proper. As in the case of $\Inv_p(X)$, it is straightforward to verify that when $X$ is a Ptolemy space the function $s_p$ satisfies the triangle inequality and so $\Sph_p(X)$ is a metric space. 

A homeomorphism $f:X\to X$ between (quasi-)metric spaces is said to be {\em M\"obius} provided that, for all quadruples $(a,b,c,d)$ of distinct points in $X$, we have 
\[\frac{d (f(a),f(c))\, d (f(b),f(d))}{d (f(a),f(d))\, d (f(b),f(c))}
=\frac{d (a,c) \,d (b,d)}{d (a,d) \,d (b,c)}.\]
We denote the group of all M\"obius self-homeomorphisms of a space $X$ with the notation $\Mob(X)$. We remark that, for any $p\in X$, we have
 $$\Mob(  X\cup\{\infty\}) = \Mob(\Sph_p(X))=\Mob(\Inv_p(X)\cup\{p\}).$$
 A metric space $X$ is   \textit{$2$-point  M\"obius homogeneous} if for every   two pairs $\{x,y\}$ and $\{a,b\}$ of distinct points in $X$ there exists a  M\"obius self-homeomorphism of $X$ for which $f(x)=a$ and $f(y)=b$.

\subsubsection{Quasi-inversions, quasi-sphericalizations, quasi-dilations, and quasi-M\"obius maps.}\label{sec:coarse:term}

In the sequel we shall use the symbol $a\simeq_L b$ to mean $ {L}^{-1} b \leq a \leq L  b$ for some constant $L\geq1$. 

A homeomorphism $f:X\to Y$ is \textit{$L$-bi-Lipschitz}, for some $L\geq1$, if, for any $x,y\in X$, we have $d(f(x),f(y))\eqx_Ld(x,y)$. We say that a metric space $X$ is {\em uniformly bi-Lipschitz homogeneous} if there exists $L\geq 1$ such that for any $x,y\in X$, there exists an $L$-bi-Lipschitz homeomorphism $f:X\to X$ for which $f(x)=y$. 

We say that a metric space $X$ is {\em quasi-invertible} provided it admits a {\em quasi-inversion} at some point $p\in X$. That is, there exists $L\geq1$ and a homeomorphism $\sigma_p:X_p\to X_p$ such that, for $x,y\in X_p$, we have
\[d(\sigma_p(x),\sigma_p(y))\simeq_L\frac{d(x,y)}{d(x,p)d(y,p)}.\]

Quasi-inversions $\sigma_p$ are closely related to the notion of the \textit{quasi-inverted space of $X$ at $p$}, denoted by $\inv_p(X)$, which is the metric space $(\hat{X}_p,d_p)$, where $d_p$ is a distance satisfying
\begin{equation}\label{E:inverted}
\frac{1}{4}i_p(x,y)\leq d_p(x,y)\leq i_p(x,y).
\end{equation}
See \cite{BHX-inversions} or \cite{LS15} for the construction of such a distance (this notion is referred to as \textit{flattening} in \cite{LS15}). This distance can be continuously extended to $\hat{X}_p$, and one can use the triangle inequality to verify that, for any point $x\in X$, one has $d_p(x,\infty)=1/d(x,p)$. 

The \textit{quasi-sphericalized space of $X$ at $p$} is denoted by $\sph_p(X)$. This is the metric space $(\hat{X},\hat{d}_p)$, where $\hat{d}_p$ is a distance satisfying
\begin{equation}\label{E:sphere_dist}
\frac{1}{4}s_p(x,y)\leq \hat{d}_p(x,y)\leq s_p(x,y). 
\end{equation}
We again refer the reader to \cite{BHX-inversions} or \cite{LS15} for the construction of such a distance. This distance can be continuously extended to $\hat{X}$ such that, for $x\in X$, we have $\hat{d}_p(x,\infty)=1/(1+d(x,p))$.
 
Given $q\in X$, $\lambda>0$, and $L\geq1$, a homeomorphism  $f:(X,d)\to (X,d)$ is said to be  a  {\em  $(\lambda,L)$-quasi-dilation} at 
$q$ provided that $f (q)=q$ and, for all $x,y\in X$,
\[d(f(x),f(y))\eqx_L \lambda\,d(x,y).\]
In particular, a $(\lambda,1)$-quasi-dilation is a {\em dilation  of factor} $\lambda$. A metric space   $ X$ is  {\em uniformly quasi-dilation invariant} at $q$ provided that there exists $L\geq 1$ such that for all $\lambda >0$ the space $X$ admits a  $(\lambda,L)$-quasi-dilation at $q$.

Given a homeomorphism $\theta:[0,+\infty)\to[0,+\infty)$, a homeomorphism $f:X\to Y$ between metric spaces is said to be a {\em $\theta$-quasi-M\"obius map} provided that, for all quadruples of distinct points $a,b,c,d\in X$, we have
\[\frac{d (f(a),f(c))\, d (f(b),f(d))}{d(f(a),f(d))\, d (f(b),f(c))}
\leq \theta\left(\frac{d (a,c) \,d(b,d)}{d (a,d) \,d (b,c)}\right).\]
When there exists $C\geq1$ such that  $\theta(t)=Ct$, then we say that $f$ is a \textit{$C$-strongly quasi-M\"obius} map. A metric space $X$ is said to be \textit{$2$-point uniformly strongly quasi-M\"obius homogeneous} provided there exists $C\geq1$ such that, given any two pairs $\{x,y\}$ and $\{a,b\}$ of distinct points in $X$, there exists a $C$-strongly quasi-M\"obius self-homeomorphism of $X$ for which $f(x)=a$ and $f(y)=b$.

\subsubsection{Additional terminology}\label{S:additional}

Given any metric space $(X,d)$ and $\alpha\in(0,1]$, the \textit{$\alpha$-snowflake of $(X,d)$} is the metric space $(X,d^{\alpha})$. A metric space $(X,d)$ is called an \textit{$\alpha$-snowflake} if it is isometric to an $\alpha$-snowflake of a metric space, or, equivalently, if $d^{1/\alpha}$ satisfies the triangle inequality. When we want to emphasize that $\alpha<1$, we say that $X$ is a \textit{non-trivial} snowflake. 

Given $L\geq1$, a space $X$ is said to be \textit{$L$-quasi-convex} if, for any two points $x,y\in X$, there exists a rectifiable curve $\gamma$ joining $x$ to $y$ satisfying $\textrm{Length}(\gamma)\leq L\,d(x,y)$. Such a curve $\gamma$ is said to be an \textit{$L$-quasiconvex curve}. Given $z\in X$, if every pair of points in the annulus $A(z;r,2r)$ can be joined by an $L$-quasi-convex curve contained in $A(z;r/L,2Lr)$, then we say that $X$ is \textit{$L$-annularly quasi-convex at $z$}. If $X$ is $L$-annularly quasi-convex at every point, we say that $X$ itself is $L$-annularly quasi-convex. 
We remark that if a space is annularly quasi-convex, then it is linearly locally connected (in the sense of \cite{BK-spheres}) and it is quasi-convex. Moreover, every proper quasi-convex space is bi-Lipschitz equivalent to a geodesic space.

A metric space $X$ is said to be \textit{linearly locally connected} if there exists a constant $C\geq1$ such that the following two conditions are satisfied:
\begin{itemize}
	\item[$(\textrm{LLC}_1)$]{For any $x\in X$, $r>0$, and points $u,v\in B(x;r)$, there exists a continuum $E\subset B(x;Cr)$ containing $u$ and $v$.}
	\item[$(\textrm{LLC}_2)$]{For any $x\in X$, $r>0$, and points $u,v\in X\setminus B(x;r)$, there exists a continuum $E\subset X\setminus B(x;r/C)$ containing $u$ and $v$.}
\end{itemize}

We say that $(X,d)$ is \textit{$C$-uniformly perfect}, for some $C\geq1$, provided that, for every $x\in X$ and $r>0$ such that $B(x;r)\subsetneq X$, we have $B(x;r)\setminus B(x;r/C)\not=\emptyset$. 

\section{Isometric Homogeneity and Invertibility}\label{S:isometric}

In this section we prove \rf{T:sharp}. We begin with a discussion of relevant definitions in connection with certain isometrically homogeneous metric spaces and metric Lie groups, respectively. 

\subsection{Isometrically homogeneous metric spaces}\label{s:ihms}
In this subsection, we prove a few useful results about metric spaces that are both isometrically homogeneous and invertible. 

\begin{proposition}\label{P:equivalence1a}
If $X$ is isometrically homogeneous and invertible, then
\begin{enumerate}
\item{for any $p\in X$, the space $\Sph_p(X)$ is 2-point M\"obius homogeneous, and}
\item{for any $x,y\in X$, the space $X$ admits a dilation of factor $d(x,y)^2$ at $p$.}
\end{enumerate} 
\end{proposition}

\begin{proof}
To prove (1), let $\tau_p$ be an inversion at some $p\in X$. We show that every point $(a,b)\in( \Sph_p(X)\times\Sph_p(X))\setminus\Delta$ can be mapped to $(\infty, p)$. Here $\Delta\subset\Sph_p(X)\times\Sph_p(X)$ denotes the diagonal. If $a=\infty $, then one uses an element in $\Isom(X)\subset\Mob(X)\subset\Mob(\Sph_p(X))$ mapping $b$ to $p$. If $a\neq\infty $, then one first uses an element in $\Isom(X)$ mapping $a$ to $p$, then the map $\tau_p\in\Mob(\Sph_p(X))$, to be back in the case $a=\infty$.
 
To prove (2), let $\tau_p$ denote an inversion at $p\in X$. Choose $x\in X_p$, and define the map $g:X\to X$ as $g:=f_3\circ\tau_p\circ f_2\circ \tau_p\circ f_1\circ \tau_p$. Here $f_1:X\to X$ is an isometry such that $f_1(p)=x$, $f_2: X\to X$ is an isometry such that $f_2(\tau_p(x))=p$, and $f_3: X\to X$ is an isometry such that $f_3(\tau_p(f_2(p)))=p$. We then observe that $g(p)=p$, and that, for any $a,b\in X$, we have
\begin{align*}
	d(g(a),g(b))&=\frac{d(\tau_p(f_1(\tau_p(a))),\tau_p(f_1(\tau_p(b))))}{d(f_2(\tau_p(f_1(\tau_p(a)))),p)\,d(f_2(\tau_p(f_1(\tau_p(b)))),p)}\\
	&=\frac{d(\tau_p(f_1(\tau_p(a))),\tau_p(f_1(\tau_p(b))))}{d(f_2(\tau_p(f_1(\tau_p(a)))),f_2(\tau_p(x)))\,d(f_2(\tau_p(f_1(\tau_p(b)))),f_2(\tau_p(x)))}\\
	&=\frac{d(\tau_p(a),\tau_p(b))}{d(f_1(\tau_p(a)),p)\,d(f_1(\tau_p(b)),p)}\cdot\frac{d(f_1(\tau_p(a)),p)\,d(x,p)}{d(f_1(\tau_p(a)),x)}\cdot\frac{d(f_1(\tau_p(b)),p)\,d(x,p)}{d(f_1(\tau_p(b)),x)}\\
	&=\frac{d(\tau_p(a),\tau_p(b))\,d(x,p)^2}{d(f_1(\tau_p(a)),f_1(p))\,d(f_1(\tau_p(b)),f_1(p))}\\
	&=\frac{d(a,b)\,d(x,p)^2}{d(a,p)\,d(b,p)}\cdot d(a,p)\,d(b,p)\\
	&=d(x,p)^2\,d(a,b)
\end{align*}
Thus $g:(X,d)\to(X,d)$ is a dilation of factor $d(x,p)^2$ at $p$. 
\end{proof}

In the following lemma, we say that a bijection $h:X\to X$ is a \textit{similarity}  provided that there exists $\lambda>0$ such that, for any $x,y\in X$, we have $d(f(x),f(y))=\lambda\,d(x,y)$.
Hence, within this paper, the difference between a similarity and a dilation is that the latter requires the presence of a fixed point while the former does not. Although the following result is contained in the proof of Proposition~\ref{general_prop}, it is included to provide a convenient reference. For a similar result, the reader is pointed to \cite[Proposition 2.4]{PS17}. 

\begin{lemma}\label{L:similarity}
Suppose $X$ and $Y$ are unbounded. If $h:X\to Y$ is a M\"obius homeomorphism such that both $h$ and $h^{-1}$ send bounded sets to bounded sets, then $h:X\to Y$ is a similarity. 
\end{lemma}
 
\begin{remark}\label{R:similarity}
As consequence of \rf{L:similarity} we note that, for unbounded spaces $X$ and $Y$, any M\"obius homeomorphism $h:\Sph_p(X)\to\Sph_{q}(Y)$ fixing $\infty$ is a similarity map from $X$ to $Y$. Indeed, $\Sph_p(X)\setminus\{\infty\}$ and $\Sph_q(Y)\setminus\{\infty\}$ are M\"obius equivalent to $X$ and $Y$, respectively. Therefore, $h:X\to Y$ is M\"obius and both $h$ and $h^{-1}$ send bounded sets to bounded sets.
\end{remark}

\begin{proposition} \label{P:equivalence1b}
Suppose $(X,d)$ is an unbounded and complete metric space. If, for some $p\in X$, the space $\Sph_p(X)$ is $2$-point M\"obius homogeneous, then the space $(X,d)$ is isometrically homogeneous and, for some $c>0$, the space $(X,c\,d)$ is  invertible. 
\end{proposition}
 
\begin{proof}
We first prove that $X$ is isometrically homogeneous. By \rf{R:similarity}, every M\"obius map $h:\Sph_p(X)\to\Sph_p(X)$ fixing $\infty$ yields a map $h: X \to X $ that is a $\lambda$-similarity for some $\lambda=\lambda(h)>0$. Therefore, given any $x,y\in X$, our assumptions on $\Sph_p(X)$ ensure the existence of a $\lambda$-similarity $h:X\to X$ such that $h(x)=y$. If $\lambda=1$ we have an isometry of $X$ sending $x$ to $y$. If $\lambda\neq 1$, then, since $X$ is complete, by the Banach Fixed Point Theorem, there exists $o\in X$ such that $h(o)=o$. Again invoking our assumptions on $\Sph_p(X)$ and \rf{R:similarity}, there exists a $\mu$-similarity $g:X \to X$ such that $g(o)=y$. Here $\mu=\mu(g)>0$. We claim that the composition $g\circ  h^{-1} \circ g^{-1}\circ  h$ is an isometry of $X$ that sends $x$ to $y$. Indeed, such a map is a similarity of factor $ \mu\lambda^{-1}\mu^{-1}\lambda=1$, and 
\[g(  h^{-1} ( g^{-1}(  h(x))))= g(  h^{-1} ( g^{-1}(  y)))= g(  h^{-1} ( o))= g(  o)=y.\]
Since $x,y\in X$ were arbitrary, it follows that $X$ is isometrically homogeneous. 

Next, we prove that $X$ admits an inversion, up to a rescaling of its distance. To this end, let $f:\Sph_p(X)\to\Sph_p(X)$ denote a M\"obius map such that $f(p)=\infty$ and $f(\infty)=p$. Then, for any $a,b\in X_p$ such that $a\not=b$, we have
\begin{align*}
d(p,f(a))&=s_p(p,f(a))(1+d(p,f(a)))\\
&=s_p(p,f(a))s_p(\infty,f(b))\frac{(1+d(p,f(a)))}{s_p(\infty,f(b))}\\
&=\frac{s_p(\infty,a)s_p(p,b)}{s_p(\infty,b)s_p(a,p)}\frac{1+d(p,f(a))}{s_p(\infty,f(b))}s_p(p,f(b))s_p(f(a),\infty)\\
&=\frac{1}{d(a,p)}\frac{(1+d(p,f(a)))(1+d(a,p))(1+d(b,p))(1+d(f(b),p))}{(1+d(p,f(a)))(1+d(a,p))(1+d(b,p))(1+d(f(b),p))}d(p,b)d(p,f(b))\\
&=\frac{d(p,b)d(p,f(b))}{d(a,p)}.
\end{align*}
Since the above equalities hold for any $b\not=a$ in $X$, we conclude that there exists a constant $r>0$ such that, for any $b\in X$, we have 
\begin{equation}\label{E:iso_scale}
d(f(b),p)=r\cdot d(b,p)^{-1}.
\end{equation}
Now let $a,b\in X_p$ be such that $a\not=b$. Using the same function $f$ as above, we observe that
\begin{align*}
d(f(a),f(b))&=s_p(f(a),f(b))(1+d(f(a),p))(1+d(f(b),p))\\
&=s_p(f(a),f(b))s_p(f(p),f(\infty))(1+d(f(a),p))(1+d(f(b),p))\\
&=\frac{s_p(a,b)s_p(p,\infty)}{s_p(a,\infty)s_p(b,p)}s_p(f(a),p)s_p(f(b),\infty)(1+d(f(a),p))(1+d(f(b),p))\\
&=\frac{d(a,b)d(f(a),p)}{d(b,p)}\frac{(1+d(f(a),p))(1+d(f(b),p))(1+d(a,p))(1+d(b,p))}{(1+d(f(a),p))(1+d(f(b),p))(1+d(a,p))(1+d(b,p))}\\
&=\frac{d(a,b)d(f(a),p)}{d(b,p)}=\frac{r\cdot d(a,b)}{d(a,p)d(b,p)}.
\end{align*}
Here we note that the final equality follows from \rf{E:iso_scale}. 

Set $c:= 1/\sqrt{r}$. Therefore the last formula becomes 
\begin{align*}
c\,d(f(a),f(b))&=\tfrac{1}{\sqrt{r}}\,d(f(a),f(b))=\sqrt{r}  \frac{d(a,b)}{d(a,p)d(b,p)}\\
&=\frac{\tfrac{1}{\sqrt{r}}\,d(a,b)}{\tfrac{1}{\sqrt{r}}\,d(a,p)\tfrac{1}{\sqrt{r}}\,d(b,p)}=\frac{c\, d(a,b)}{c\, d(a,p)\,c\, d(b,p)}.
\end{align*}
Therefore, $f $ satisfies the definition of an inversion for $(X,c\, d)$.
\end{proof}

\subsection{Metric lie groups}\label{s:metric_Lie}
For the purposes of this paper, we refer to Lie groups equipped with left-invariant distance functions that induce the manifold topology as \textit{metric Lie groups}. Thus our terminology aligns with that of \cite{CKLDGO17}. Important examples of metric Lie groups are provided by groups referred to as \textit{generalized Heisenberg groups} (as in \cite{Freeman-invertible}) or \textit{$\mfK$-Heisenberg groups} (as in \cite{PS17}). Here $\mfK$ denotes a real normed division algebra: either the real numbers $\mfR$, the complex numbers $\mfC$, the quaternions $\mfH$, or the octonions $\mfO$. These groups can be defined as follows. 

\begin{itemize}
  \item{Given $n\in\msf{N}$, the $n$-th $\mfR$-Heisenberg group, or a \textit{real Heisenberg group} $\mathbb H_\mfR^n$, is $\mfR^n$.}
  \item{Given $n\in\msf{N}$, the $n$-th $\mfC$-Heisenberg group, or a \textit{complex Heisenberg group} $\mathbb H_\mfC^{n}$, is the Carnot group with step two real Lie algebra $\mathfrak{n}=\mathfrak{v}\oplus \mathfrak{z}$, where $\mathfrak{v}:=\Span\{X_i,Y_i:1\leq i\leq n\}$ and $\mathfrak{z}:=\Span\{Z\}$. Equip $\mathfrak{n}$ with an inner product such that $\{X_i,Y_i,Z:1\leq i\leq n\}$ is an orthonormal basis. The only non-trivial bracket relations are $[X_i,Y_i]=Z$, for $1\leq i\leq n$.} 
  \item{Given $n\in\msf{N}$, the $n$-th $\mfH$-Heisenberg group, or a \textit{quaternionic Heisenberg group} $\mathbb H_\mfH^{n}$, is the Carnot group with step two real Lie algebra $\mathfrak{n}=\mathfrak{v}\oplus\mathfrak{z}$, where $\mathfrak{v}=\Span\{X_i,Y_i,V_i,W_i:1\leq i \leq n\}$ and $\mathfrak{z}=\Span\{Z_k:1\leq k \leq 3\}$. Equip $\mathfrak{n}$ with an inner product such that $\{X_i,Y_i,V_i,W_i,Z_k:1\leq i\leq n, 1\leq k \leq 3\}$ is an orthonormal basis. For $1\leq i \leq n$, the only nontrivial bracket relations are $[X_i,Y_i]=Z_1=[V_i,W_i]$, $[X_i,V_i]=Z_2=[W_i,Y_i]$, and $[X_i,W_i]=Z_3=[Y_i,V_i]$.}
  \item{The $\mfO$-Heisenberg group, or the \textit{octonionic Heisenberg group} $\msf H_{\msf O}^{1}$, is the Carnot group with step two real Lie algebra $\mathfrak{n}=\mathfrak{v}\oplus\mathfrak{z}$, where $\mathfrak{v}=\Span\{X_i:0\leq i\leq 7\}$ and $\mathfrak{z}=\Span\{Z_k:1\leq k \leq 7\}$. Equip $\mathfrak{n}$ with an inner product such that $\{X_i,Z_k:0\leq i\leq 7, 1\leq k \leq 7\}$ is an orthonormal basis. The only nontrivial bracket relations are $[X_0,X_k]=Z_k$ for $1\leq k\leq 7$ and $[X_i,X_j]=\varepsilon_{ijk}Z_k$, for $1\leq i,j,k\leq 7$. Here $\varepsilon$ is a completely antisymmetric tensor whose value is $+1$ when $ijk=124, 137, 156, 235, 267, 346, 457$.}
\end{itemize}

For our purposes, it is sufficient to define $\mfK$-hyperbolic space via the results of \cite{CDKR98} and \cite{CDKR-Heisenberg}. In particular, we may view the $\mfK$-hyperbolic spaces as the rank-one symmetric spaces of non-compact type, and thus the $\mfK$-Heisenberg groups described above can be viewed as the boundaries at infinity of the $\mfK$-hyperbolic spaces. For more detailed information about $\mfK$-Heisenberg groups in relation to $\mfK$-hyperbolic space the reader may consult \cite{Platis-ptolemy}.

Given $(x,z)=\exp(X+Z)\in\mbb{H}_\mfK^n$, where $X\in\frak{v}$ and $Z\in\frak{z}$, we define  
\[\|(x,z)\|:=\left(\frac{|X|^4}{16}+|Z|^2\right)^{1/4}.\]
Here $|\cdot|$ denotes the norm obtained from the inner product on $\mathfrak{n}$ described above. We then define the {\em parabolic visual distance} $\rho$ on $\mH^n$ as 
\begin{equation}\label{E:Ham}
\rho((x,z),(x',z')):=\|(x',z')^{-1}(x,z)\|.
\end{equation}
This distance (or a rescaling thereof) is sometimes referred to as the \textit{Kor\'anyi-Cygan} distance, or simply the \textit{Kor\'anyi distance} (cf.\,\cite[page 18]{CDPT07}). Via the exponential map, an inversion of the metric Lie group $(\mH^n,\rho)$ is given by
\begin{equation}\label{E:sigma}
\sigma(X,Z):=-\left(\frac{|X|^2}{4}+J_Z\right)^{-1}X-\left(\frac{|X|^4}{16}+|Z|^2\right)^{-1}Z.
\end{equation}
Here $J:\mathfrak{z}\to\End(\mathfrak{v})$ is defined via the formula $\langle J_ZX,Y\rangle=\langle Z,[X,Y]\rangle$. See \cite{CDKR-Heisenberg} for a detailed treatment of the map $\sigma$. 

One of the primary theoretical tools we shall employ in the proof of \rf{T:sharp} is provided by the following version of results from \cite{Kramer-transitive}.

\begin{fact}\label{F:Kramer}
Suppose $G$ is a locally compact and $\sigma$-compact topological group acting continuously, effectively, and $2$-transitively on the sphere $\msf{S}^m$. In this case, $G$ can be given the structure of a Lie group and the identity component $G^\circ$ is simple, non-compact, and of real rank $1$. Furthermore, the action of $G$ on $\msf{S}^m$ is isomorphic to the action of $G_\mfK$ or $G_\mfK^\circ$ on the (compact) boundary at infinity of $\mfK$-hyperbolic space. Here $G_\mfK$ denotes the isometry group of $\mfK$-hyperbolic space: If $\mfK=\mfR$, then $G_\mfK=PO(n,1)$ for $m=n-1\geq1$. If $\mfK=\mfC$, then $G_\mfK=PU(n,1)\rtimes \mfZ_2$ for $m=2n-1$. If $\mfK=\mfH$, then $G_\mfK=PSp(n,1)$ for $m=4n-1$. If $\mfK=\mfO$, then $G_\mfK=F_{4(-20)}$ for $m=15$. 
\end{fact}

The above fact follows immediately from \cite[Theorem 3.3(a)]{Kramer-transitive} and \cite[Proposition 7.1]{Kramer-transitive}. Indeed, by \cite[Theorem 3.3(a)]{Kramer-transitive}, we conclude that $G$ is a Lie group with simple and non-compact connected component. Furthermore, $G$ is isomorphic to either $G_\mfK$ or $G_\mfK^\circ$ for some $\mfK\in\{\msf{R},\msf{C},\msf{H},\msf{O}\}$, as described above. We also point out \cite[Proposition 7.1]{Kramer-transitive}, which affirms that the action of $G$ on $\msf{S}^m$ is the standard action of $G_\mfK$ on the boundary of its corresponding symmetric space, namely $H_\mfK/B$. Here $H_\mfK$ denotes $G_\mfK^\circ$ and $B=NAM$, where $H_\mfK=NAK$ is the Iwasawa decomposition of $H_\mfK$, $M$ is the centralizer of $A$ in $K$, and $N$ is isomorphic to $\mbb{H}^n_\mfK$.

Another tool we employ in the proof of \rf{T:sharp} is the following version of results from \cite{PS17}. Here $\rho$ denotes the parabolic visual  distance on $\mbb{H}_\mfK^n$ defined in \rf{E:Ham}. 

\begin{theorem}\label{T:rigidity}
Suppose $d$ is a metric on $\mbb{H}_\msf{K}^n$ such that both $d$ and $\rho$ induce the same topology. If $\Mob(\Sph_e(\mbb{H}_\msf{K}^n,\rho))^\circ\subset\Mob(\Sph_e(\mbb{H}_\msf{K}^n,d))$, then there exists $c>0$ and $0<\alpha\leq1$ such that $d=c\cdot \rho^\alpha$. 
\end{theorem}

\begin{proof}
We note that orientation-preserving similarity mappings of $(\mbb{H}_\msf{K}^n,\rho)$ are contained in the identity component $\Mob(\Sph_e(\mbb{H}_\msf{K}^n,\rho))^\circ$. Therefore, when $\msf{K}=\msf{R}$, we reach the desired conclusion via \cite[Theorem 1.1(a)]{PS17}. In the cases that $\msf{K}\not=\msf{R}$, we note that the inversion  $\sigma$ defined in \rf{E:sigma} is contained in $\Mob(\Sph_e(\mbb{H}_\msf{K}^n,\rho))^\circ$. Via \cite[Theorem 1.2]{PS17}, we are done.
\end{proof}

In connection with \rf{T:rigidity}, we point out that the norm utilized in the present paper to define the visual distance on $\mfK$-Heisenberg groups  differs slightly from the norm defined in \cite[page 358]{PS17}. This is due to a different choice of coordinates for $\mH^n$. Nevertheless, up to corresponding alterations in the definition of the canonical inversion map (compare \cite[page 363]{PS17} with \rf{E:sigma}), the proofs of \cite{PS17} yield \rf{T:rigidity}. 

\begin{remark}
Suppose $X$ is a proper and connected metric space. As a consequence of \cite[Theorem 1.4]{CKLDGO17}, we find that if the action of $\Mob(X)$ on $X$ is transitive and not proper, then $X$ has the structure of self-similar metric Lie group in the sense of \cite{CKLDGO17}. Indeed, by \rf{L:similarity}, the group  $\Mob(X)$ is precisely the group of similarities. Therefore, if its action is not proper, then $\Mob(X)$ must contain a similarity that is not an isometry.
\end{remark}

\subsection{Proof of Theorem~\ref{T:sharp}}

Before beginning the proof of \rf{T:sharp}, we first prove a lemma regarding compactness properties of $\Mob(\Sph_p(X))$, where we remind the reader that, in general, $\Sph_p(X)$ is a quasi-metric space when equipped with the quasi-distance $s_p$. 

If $X$ is an unbounded, proper (i.e., boundedly compact), and connected metric space, then the topology induced on $\hat{X}$ by the distance $\hat{d}_p$ from \eqref{E:inverted} coincides with that of the one-point compactification of $X$. Since the distance $\hat{d}_p$ is bi-Lipschitz equivalent to the quasi-distance $s_p$ on $\hat{X}$, this topology coincides with the topology on $\hat{X}$ generated by open balls with respect to $s_p$. Thus we may speak of continuous self-mappings of $\Sph_p(X)$ with respect to this topology. We then define a quasi-distance 
\[s_p^*(f,g):=\sup\{s_p(f(x),g(x))\,|\,x\in \hat{X}\}\]
on the set of continuous mappings of the quasi-metric space $\Sph_p(X)$. We refer to the topology induced on $\Mob(\Sph_p(X))$ (a group of continuous mappings of $\Sph_p(X)$) by the quasi-distance $s_p^*$ as the \textit{topology of uniform convergence}. It is straightforward to check that the action of $\Mob(\Sph_p(X))$ on $\Sph_p(X)$ is a continuous action with respect to these topologies. 

Recall from Section~\ref{sec:coarse:term} that  the quasi-metric space $\Sph_p(X)$ is bi-Lipschitz equivalent to the metric space $\sph_p(X)$ via the identity. Hence, the group $G:=\Mob(\Sph_p(X))$ acts on the metric space $\sph_p(X)$ by uniformly strongly quasi-M\"obius mappings. Furthermore, the topology of uniform convergence induced on $G$ by $s_p$ coincides with the topology of uniform convergence induced on $G$ by the distance $\hat{d}_p$. Via \cite[Lemma 4.4]{Freeman-invertible}, this last observation yields the following lemma.

\begin{lemma}\label{L:quasimetric}
Given an unbounded, proper, and connected metric space $X$, the group $\Mob(\Sph_p(X))$ is locally compact and $\sigma$-compact in the topology of uniform convergence. 
\end{lemma}

\begin{proof}[Proof  of Theorem~\ref{T:sharp}]
We prove $(4)\Leftrightarrow(3)\Leftrightarrow(2)\Leftrightarrow(1)$. 

We begin with $(4)\Leftrightarrow(3)$. By \rf{P:equivalence1b}, assuming $(4)$ for $(X,d)$ implies that $(X,d)$ is isometrically homogeneous and $(X,c\,d)$ admits an inversion $\tau$, where $c$ is some positive constant. Furthermore, by \rf{P:equivalence1a}, there exists a dilation $f$ of $(X,d)$ at $p\in X$ with dilation factor $c^{-2}$. We then observe that $\tau\circ f$ is an inversion of $(X_p,d)$, where $\tau$ denotes an inversion of $(X_p,c\,d)$ at $p$. Therefore, $(4)\Rightarrow(3)$. To see that $(3)\Rightarrow(4)$, we first note that the combination of isometric homogeneity and local compactness implies that $X$ is complete. To confirm that $\Mob(\hat{X})$ acts $2$-transitively on $\hat{X}$, we refer to \rf{P:equivalence1a}. 

\medskip
We now prove the main implication   $(3)\implies(2)$. We claim that $X$ admits dilations of all factors at $p$. Indeed, fixing $y\in X$, the distance function $x\in X\mapsto d(x,y)\in [0,\infty)$ is continuous and unbounded, since $X$ is assumed unbounded. Thus $\Lambda:=\{ d(x,y) : x\in X\}$ is a closed and unbounded set that contains $0$. Since $X$ is connected, $\Lambda=[0,\infty)$. By \rf{P:equivalence1a}, our claim is verified. 

Since $X$ is assumed to be connected and locally compact, by \cite[Theorem 1.4]{CKLDGO17}, we conclude that $X$ may be given the structure of a metric Lie group for which every dilation fixing the identity element is an automorphism.
It then follows from results in \cite{Siebert86} that $X$ is nilpotent and simply connected. In particular, the space $X$ is homeomorphic to $\msf{R}^m$, for some $m\in\msf{N}$. In addition, since $X$ is locally compact and admits dilations, it is proper. Consequently, the one-point compactification of $X$ coincides with the topology of $\Sph_p(X)$ induced by the quasi-distance $s_p$. Therefore, the space $\Sph_p(X)$ is homeomorphic to the topological sphere $\msf{S}^m$. Via \rf{L:quasimetric}, we know that $G=\Mob(\Sph_p(X))$ is locally compact and $\sigma$-compact. Since $G$ acts continuously, effectively, and $2$-transitively on the topological sphere $\Sph_p(X)$, by \rf{F:Kramer} we conclude that the action of $G$ on $\Sph_p(X)$ is isomorphic to the standard action of either $G_\mfK$ or $G_\mfK^\circ$ on the (compact) boundary at infinity of some $\mfK$-hyperbolic space. We recognize this boundary as $\hat{\mbb{H}}_\mfK^n$, for some $n\in \msf N$ via  the approach of \cite{CDKR98}, and we emphasize that $G_\mfK$ acts by M\"obius mappings on $\hat{\mbb{H}}_\mfK^n$. Thus we identify $X$ with $\mbb{H}_\mfK^n$, 
identifying $p$ with the identity element $e$ of $\mbb{H}_\mfK^n$ and $\infty$ with $\infty$.
Also, we identify the action of $G$ on $\Sph_e(\mbb{H}_\mfK^n,d)$ with the action of either $G_\mfK$ or $G_\mfK^\circ$ on $\Sph_e(\mbb{H}_\mfK^n,\rho)$. In particular, we have
\[\Mob(\Sph_e(\mbb{H}_\mfK^n,\rho))^\circ\subset\Mob(\Sph_e(\mbb{H}_\mfK^n,d)).\]
By \rf{T:rigidity}, we have $d=c\,\rho^\alpha$ for some $c>0$ and $\alpha\in(0,1]$, 
We conclude by noticing that any dilation (with respect to $\rho$) of factor $c^{1/\alpha}$
provides an isometry from $ (\mbb{H}_\mfK^n,d) $ to $(\mbb{H}_\mfK^n,\rho^\alpha)$, 
and thus $(3)\Rightarrow(2)$. 

We next prove $(2)\Rightarrow(3)$. Clearly, $(\mathbb{H}_\mfK^n,\rho)$ is isometrically homogeneous (since the distance $\rho$ is left-invariant) and invertible (since it is the boundary of a symmetric space); see \cite{CDKR98} for these classical facts.    It is clear that the same is true of its snowflakes.   

\medskip

Finally, we prove $(2)\Leftrightarrow(1)$. The implication $(2)\Rightarrow(1)$ is trivial, since the identity map from $(X,d)$ to $(X,c\,d)$ is a M\"obius homeomorphism. Conversely, suppose that $f:(\mbb{H}_\mfK^n,\rho^\alpha)\to X$ is a M\"obius homeomorphism. For any $s>0$, write $\delta_s:\mbb{H}_\mfK^n\to\mbb{H}_\mfK^n$ to denote the standard automorphic dilation of factor $s$ with respect to the distance $\rho$. Given a point $z_0\in \mbb{H}_\mfK^n$ such that $\rho(e,z_0)=1$, write $\gamma(z_0)$ to denote the closure of the curve 
\[\gamma'(z_0):=\{\delta_s(z_0)\,|\,s>0\}.\]
Thus $\gamma(z_0)=\gamma'(z_0)\cup\{e\}$, where $e$ denotes the identity element of $\mbb{H}_\mfK^n$. Since $\mbb{H}_\mfK^n=\cup_{d(e,z)=1}\gamma(z)$, $\mbb{H}_\mfK^n$ is proper, $X$ is unbounded, and $f$ is a homeomorphism, we may assume that $f(\gamma(x_0))$ is unbounded.

Write $p:=f(e)\in X$. We claim that $d(p,f(\delta_t(z_0)))\to+\infty$ as $t\to+\infty$. Indeed, since $\mbb{H}_\mfK^n$ is proper and $f(\gamma(z_0))$ is unbounded, there exists a sequence of real numbers $(t_n)_{n\in\msf{N}}$ such that $t_n\to+\infty$ and $d(p,f(\delta_{t_n}(z_0)))\to+\infty$. Suppose, by way of contradiction, that there also exists a sequence of real numbers $(s_n)_{n\in\msf{N}}$ such that $s_n\to+\infty$ and $\{f(\delta_{s_n}(z_0))\}$ is bounded. Up to a subsequence, we may assume that 
\begin{equation}\label{E:sequence_relation}
t_n\geq s_n.
\end{equation}
Since $f$ is M\"obius and $\rho(e,z_0)=1$, we have
\begin{align*}
\frac{\rho(\delta_{s_n/t_n}(z_0),z_0)}{\rho(z_0,\delta_{s_n}(z_0))}&=\frac{\rho(\delta_{s_n}(z_0),\delta_{t_n}(z_0))}{t_n\rho(z_0,\delta_{s_n}(z_0))}\\
&=\frac{\rho(e,z_0)\,\rho(\delta_{s_n}(z_0),\delta_{t_n}(z_0))}{\rho(e,\delta_{t_n}(z_0))\,\rho(z_0,\delta_{s_n}(z_0))}\\
&=\frac{d(p,f(z_0))\,d(f(\delta_{s_n}(z_0)),f(\delta_{t_n}(z_0)))}{d(p,f(\delta_{t_n}(z_0)))\,d(f(z_0),f(\delta_{s_n}(z_0)))}
\end{align*}
Since $\{f(\delta_{s_n}(z_0))\}$ is bounded, it is straightforward to verify via the triangle inequality that there exists $N\in\msf{N}$ such that, for any $n\geq N$, we have
\begin{equation}\label{E:part2}
\frac{d(p,f(z_0))\,d(f(\delta_{s_n}(z_0)),f(\delta_{t_n}(z_0)))}{d(p,f(\delta_{t_n}(z_0)))\,d(f(z_0),f(\delta_{s_n}(z_0)))}\eqx_2\frac{d(p,f(z_0))}{d(f(z_0),f(\delta_{s_n}(z_0)))}.
\end{equation}
Since $f^{-1}$ is continuous, there exists $c>0$ such that, for all $n\geq N$, we have $d(f(z_0)),f(\delta_{s_n}(z_0))\geq c$. Therefore, for all $n\geq N$, we have
\begin{equation}\label{E:part3}
\frac{d(p,f(z_0))}{d(f(z_0),f(\delta_{s_n}(z_0)))}\in [C^{-1},C],
\end{equation}
for some $C\in[1,+\infty)$. By combining \rf{E:sequence_relation}, \rf{E:part2}, and \rf{E:part3}, for any $n\geq N$, we have
\[\rho(z_0,\delta_{s_n}(z_0))\leq 2C\rho(z_0,\delta_{s_n/t_n}(z_0))\leq C'\]
for some $C'<+\infty$. The constant $C'$ arises from the fact that $\{\delta_s(z_0)\,|\,s\in(0,1]\}\cup\{e\}$ is compact. This inequality contradicts the fact that $s_n\to+\infty$. From this contradiction it follows that $d(p,f(\delta_t(z_0)))\to+\infty$ as $t\to+\infty$.

Choose $\lambda>1$. We claim that the homeomorphism $g_\lambda:=f\circ \delta_\lambda \circ f^{-1}:X\to X$ is a dilation of $X$ at $p$. To verify this claim, write $z_n:=f(\delta_\lambda^n(z_0))$, for $n\in\msf{N}$. We then note that $g_\lambda(z_n)=f(\delta_\lambda^{n+1}(z_0))=z_{n+1}$. By the previous paragraph, both $d(p,z_n)\to+\infty$ and $d(p,g_\lambda(z_n))\to+\infty$ as $n\to+\infty$. Since $g_\lambda$ is a M\"obius map, for any $x,y\in X$, we have
\[\frac{d(g_\lambda(x),g_\lambda(y))\,d(g_\lambda(z_n),g_\lambda(p))}{d(g_\lambda(x),g_\lambda(p))\,d(g_\lambda(z_n),g_\lambda(y))}=\frac{d(x,y)\,d(z_n,p)}{d(x,p)\,d(z_n,y)}.\]
Taking a limit as $n\to+\infty$, we obtain
\[\frac{d(g_\lambda(x),g_\lambda(y))}{d(g_\lambda(x),g_\lambda(p))}=\frac{d(x,y)}{d(x,p)}.\]
Here we note that $g_\lambda(p)=p$, and thus we have
\[\frac{d(g_\lambda(x),g_\lambda(y))}{d(x,y)}=\frac{d(g_\lambda(x),p)}{d(x,p)}.\]
The left side of this equality is symmetric in the variables $x$ and $y$. The right side is independent of $y$. Therefore, we conclude that there exists some number $\beta>0$ such that $d(g_\lambda(x),g_\lambda(y))=\beta\,d(x,y)$. Thus our claim is verified. 

Next, we claim that $\beta>1$. Indeed, for every $n\in \msf{N}$, we have
\[d(p,z_n)=d(p,g_\lambda^n(z_{0}))=\beta^n\,d(p,z_0).\] 
Since $d(p,z_n)\to+\infty$, we conclude that $\beta>1$.

Since $X$ is locally compact and admits a dilation of factor $\beta>1$, it is straightforward to verify that $X$ is proper. Therefore, any homeomorphism between $X$ and $\mbb{H}_\mfK^n$ preserves bounded sets. The implication $(1)\Rightarrow(2)$ then follows from \rf{L:similarity}. Indeed, by \rf{L:similarity}, there exists a constant $c>0$ such that the M\"obius homeomorphism $f^{-1}: X\to (\mbb{H}_\mfK^n,c\cdot \rho^\alpha)$ is an isometry. Since $(\mbb{H}_\mfK^n, \rho^\alpha)$ is dilation invariant, we conclude that $(1)\Rightarrow(2)$.  \end{proof}

\section{M\"obius Homogeneity and Strong Connectivity}\label{S:strong}    

This section is devoted to the proof of Theorem~\ref{QC:Mobius:thm} and Corollary~\ref{MCircles:Mobius:thm}. The arguments are heavily based on Montgomery-Zippin results about the structure of locally compact groups.

\subsection{Proof of Theorem~\ref{QC:Mobius:thm}} 

We first use theory pertaining to Hilbert's Fifth Problem to show that, in the setting of Theorem~\ref{QC:Mobius:thm}, the space of M\"obius transformations has the structure of a Lie group. As usual, it is topologized via uniform convergence.

\begin{proposition}[After Montgomery-Zippin]\label{prop:Montgomery:Zippin}
Let $Z$ be a compact, connected, locally connected metric space of finite topological dimension.
If $\Mob(Z)$ acts transitively on $Z$, then ${\rm M\ddot{o}b}(Z)$ is a Lie group.
\end{proposition}

\begin{proof}
Since $Z$ is compact, $G:={\rm M\ddot{o}b}(Z)$ is a separable, locally compact, and metrizable group. Moreover, the standard action $G\times Z\to Z$ is continuous and effective. Following, \cite[page 238]{mz}, the  locally compact group $G$ has an open subgroup $G'<G$ that is the inverse limit of Lie groups. In the language of \cite{mz}, $G'$ has property $A$.

First, we claim that, for any $q\in Z$, the orbit of $q$ under $G'$, denoted by $G'\cdot q$, is open. This is because the projection $G\to G/H$ is open and the orbit action $G/H\to Z$ is a homeomorphism (see \cite[page 121, Theorem 3.2]{Helgason}). Here $H$ denotes the isotropy subgroup of $G$ at $q$. 

Now we show that the $G'$-action is transitive. Indeed, fix a point $p\in Z$, and suppose (by way of contradiction) that $G'\cdot p\neq Z$. Hence,
$$ Z=\left(G'\cdot p\right)\bigsqcup\left( \bigcup_{q\notin G'\cdot p}G'\cdot q\right)$$
is a disjoint union of two non-empty open
sets of $Z$. This contradicts the fact that $Z$ is connected.

Thus $G'$ satisfies the hypotheses of Montgomery-Zippin's Theorem \cite[page 243]{mz}, so  $G'$ is a Lie group. Since  $G'$ does not contain small subgroups, neither does $G$. By work of Gleason-Yamabe (cf. \cite[Chapter III]{mz}), $G$ is a Lie group.
\end{proof}

\begin{proof}[Proof of Theorem~\ref{QC:Mobius:thm}] Since $Z$ is quasi-convex, it is connected and locally connected. Therefore, by Proposition~\ref{prop:Montgomery:Zippin}, we conclude that $G:={\rm M\ddot{o}b}(Z)$ is a Lie group. Since the action of $G$ on $Z$ is transitive, the space $Z$ is a manifold homeomorphic to $G/H$ for some closed subgroup $H\subset G$.

Since $Z$ is  quasi-convex and compact, up to a bi-Lipschitz change of distance  we can assume that the distance $d_Z$ of $Z$ is   geodesic. Also, since $Z$ is compact, every M\"obius homeomorphism is bi-Lipschitz (see \cite[Remark 3.2]{Kinneberg-fractal}). Thus $G$ acts on $Z$ by bi-Lipschitz maps. By \cite[Theorem 1.1]{LeDonne-geodesic} there exists a completely non-holonomic $G$-invariant distribution on $Z$ such that any  Carnot-Carath\'eodory metric coming from it gives a metric that is locally bi-Lipschitz equivalent to $d_Z$.
Since $Z$ is compact, the bi-Lipschitz equivalence is global.
\end{proof}

\subsection{M\"obius circles  and M\"obius-homogeneity}

A metric space $X$ is said to be {\em connected by M\"obius circles} if, for any $p,q\in X$, there exists a M\"obius embedding $\gamma:\msf S^1\to X$ such that $p,q\in \gamma(\msf S^1)$. Here $\msf S^1\subset\msf R^2$ denotes the unit circle. The following lemma confirms that our definition is consistent with the definition of M\"obius circles used in \cite[Section 2.4]{Buyalo_Schroeder14}.

\begin{lemma}\label{equivalenza}
Let $S$ be a subset   of a metric space $X$. The following are equivalent
\begin{itemize}
\item{$S$  is   the image of a M\"obius embedding of $\msf S^1$.}
\item{$S$ is  the closure of the image of a M\"obius embedding of $\msf R$.}
\item{$S$ is homeomorphic to $\msf S^1$ and, for every $x,y,z,u$ in order along $S$ we have
\[d(x,z)d(y,u)=d(x,y)d(z,u)+d(x,u)d(y,z).\]}
\end{itemize}
\end{lemma}

\begin{proof}
The first two characterization are a consequence of the fact that $\msf R$ and $\msf S^1$ are M\"obius equivalent (up to compactification). Observe that the equation of the lemma is equivalent to 
\begin{equation}\label{E:mob}
1=\dfrac{d(x,y)d(z,u)}{d(x,z)d(y,u)}+\dfrac{d(x,u)d(y,z)}{d(x,z)d(y,u)},
\end{equation}
and the right-hand side is the sum of two cross ratios. Hence it is a M\"obius invariant.

Let $\gamma:\msf S^1\to X$ be a M\"obius embedding with $S=\gamma(\msf S^1)$. Fix consecutive points $x,y,z,u$ along $S$ (here the order is inherited from $\msf{S}^1$). Let $x'$, $y'$, $z'$, $u'$ be the respective points in $\msf S^1$. Up to a M\"obius transformation, we may assume that $x'=0$, $y'=1$, $z'=c>1$, and $u'=\infty$. Under this transformation equation \rf{E:mob} becomes $c^{-1} + (c-1)c^{-1}=1$, which is true.

Conversely, assume points of $S=\gamma(\msf S^1)$ satisfy \rf{E:mob}, where $\gamma$ is some embedding. Fix $u\in S$ and consider the quasi-metric space $\Inv_u(X)$. In $\Inv_u(X)$ equation \rf{E:mob} becomes   $i_u(x,z) =i_u(x,y) + i_u(y,z)$. Hence the curve $\gamma\setminus\{u\}$ is an infinite geodesic in $\Inv_u(X)$, and thus isometric to $\msf R$. Since $\Inv_u(X)\setminus\{\infty\}$ is M\"obius equivalent to $X_u$, we confirm that $S$ is the closure of the image of a M\"obius embedding of $\msf{R}$.
\end{proof}

Before proceeding to the proof of Corollary~\ref{MCircles:Mobius:thm}, we first demonstrate that Ptolemy spaces connected by M\"obius circles are quasi-convex. This fact (and its proof) was suggested to the authors by V. Schroeder.

\begin{proposition}\label{Moebius QC}
If $X$ is a Ptolemy space that is connected by M\"obius circles, then $X$ is $K$-quasi-convex, for some universal constant $K\leq 144$.
\end{proposition}

\begin{proof}
Fix $p,q\in X$. Let $C$ be a  M\"obius circle through $p$ and $q$. We consider two cases.
Either $(1)$ $D:={\rm diam}_d(C)\leq 6 d(p,q)$, or $(2)$ $D> 6 d(p,q)$.

\medskip
Case 1: $D \leq 6 d(p,q)$. By continuity, choose a point $w\in C$ for which $d(p,w)=d(w,q)$. The triangle inequality yields $2d(w,q)\geq d(p,q)$. Let $\gamma$ denote the sub-arc of $C\setminus\{p,q\}$ that does not contain $w$ and joins $p$ to $q$.

We claim that Length$_d(\gamma)\leq K d(p,q)$, for $K=144$. Indeed, since $X$ is assumed to be Ptolemy, we consider the metric space $\Inv_w(X)$. In this space, the set $C\setminus\{w\}$ is an infinite geodesic M\"obius equivalent to $\msf R$ (see Lemma~\ref{equivalenza}). Therefore,
\begin{align*}
{\rm Length}_{i_w}(\gamma)&=i_w(p,q)=\dfrac{d(p,q)}{d(p,w) d(q,w)}\\
&=\dfrac{d(p,q)}{d(w,q)^2}\leq \dfrac{4d(p,q)}{d(p,q)^2}=\dfrac{4}{d(p,q)}.
\end{align*}
For $x,y\in \gamma \subset C$, we have
$$d(x,y)=d(x,w) d(y,w) i_w (x,y)\leq D^2 i_w (x,y).$$
Therefore, since we are in the case that $D \leq 6 d$, we conclude that
\[{\rm Length}_d(\gamma) \leq 36 d(p,q)^2 {\rm Length}_{i_w}(\gamma)\leq   \dfrac{4}{d(p,q)} 36 d(p,q)^2  =K d(p,q).\]

\medskip
Case 2: $D \geq 6 d(p,q)$.
We claim that by continuity there is a point $w\in C$ such that $d(p,w)=D/3.$ If not, we would have $C\subset B(p;D/3)$, and thus arrive at the contradiction $D\leq 2D/3<D$. Thus we fix $w\in C$ such that $d(p,q)=D/3$. Via the assumption that $D \geq 6 d(p,q)$, we have
\begin{align*}
d(q,w) &\geq d(p,w)-d(p,q)\\
&= \dfrac{D}{3} -d(p,q)\geq    \dfrac{D}{3} - \dfrac{D}{6}= \dfrac{D}{3}.
\end{align*}
As in Case $1$, let $\gamma$ denote the sub-arc of $C$ not containing $w$ and joining $p$ to $q$. Then
\[{\rm Length}_{i_w}(\gamma)=\frac{d(p,q)}{d(p,w) d(q,w)}\leq \frac{d(p,q)}{(D/3) (D/3)}=\frac{9 d(p,q)}{D^2}.\]
As before, for $x,y\in\gamma$, we have $d(x,y) \leq D^2 i_w(x,y).$ Therefore, we conclude that
\[{\rm Length}_d(\gamma)  \leq  D^2 {\rm Length}_{i_w}(\gamma)   \leq   9 d(p,q) .\]
\end{proof}

\subsection{Proof of Corollary~\ref{MCircles:Mobius:thm}}\label{S:BH}
Given a pointed metric space $(X,d,o)$ one  considers the {\it visual function}
 \begin{equation}\label{def:Bourdon:dist}
 \rho^{d}_o ( x,y) = \exp (- \langle x, y \rangle_o) ,
 \end{equation}
 where $\langle x, y \rangle_o$ denotes the Gromov product in $(X,d)$. Bourdon proved in \cite{Bourdon95}  that, on every CAT($-1$) space $X$, the function  $ \rho^{d}_o $   satisfies the triangle inequality and the visual boundaries $(\partial_{\infty} X, \rho^d_o)$ corresponding to different base points  $o,o'\in X$ are M\"obius equivalent. Thus we refer to $ \rho^{d}_o $ as the {\em Bourdon distance}, based at $o$.

In \cite{MR1016663} Hamenst\"adt studied similar distances where the point $o$ is replaced with a point in the boundary. We refer to such distances as {\em Hamenst\"adt  distances}. In \cite{Foertsch-Schroeder,MR2327160}, simple arguments are presented which demonstrate that these distances are M\"obius equivalent. 

\begin{proof}[{Proof of Corollary~\ref{MCircles:Mobius:thm}}] 
Suppose the boundary $X$ of a CAT$(-1)$-space endowed with a Bourdon distance $d$ is bi-Lipschitz homeomorphic to a sub-Riemannian manifold. If $X$ is equipped with a Hamenst\"adt distance $d'$, then $(X,d')$ is locally uniformly bi-Lipschitz homeomorphic to a sub-Riemannian manifold. Up to changing the sub-Riemannian metric, we conclude that this bi-Lipschitz equivalence is global. 

By \cite[Theorem 1]{Foertsch-Schroeder}, if $X$ is the boundary of a CAT$(-1)$-space endowed with a Bourdon distance, then $X$ is a Ptolemy space. By Proposition~\ref{Moebius QC}, the space $X$ is quasi-convex. We then obtain the desired conclusion via Theorem~\ref{QC:Mobius:thm}.
\end{proof}

\section{Bi-Lipschitz Homogeneity and Quasi-Invertibility}\label{S:BLH}

\subsection{Quasi-inversions and quasi-dilation invariance}
 
In this subsection we prepare for the proof of \rf{P:quasi-version} by investigating the relationship between quasi-inversions and quasi-dilations in a uniformly bi-Lipschitz homogeneous metric space. 
The reader can find the definitions of these terms along with the definition of quasi-dilation invariance in Section ~\ref{sec:coarse:term}. The definition of uniform perfectness  is provided in Section ~\ref{S:additional}.

\begin{lemma}\label{L:quasi_dilation}
Suppose $X$ is a uniformly $L$-bi-Lipschitz homogeneous metric space. If there exists a point $p\in X$ at which $X$ is $M$-quasi-invertible, then, for any $x\in X_p$, the space $X$ admits a $(C,r)$-quasi-dilation at $p$, where $r=d(x,p)^2$ and $C=C(L,M)$. Furthermore:
\begin{enumerate}
  \item{If $X$ is $N$-uniformly perfect, then $X$ is $K$-quasi-dilation invariant, with $K=K(L,M,N)$.}
  \item{If $X$ is connected, then $X$ is $K$-quasi-dilation invariant, with $K=K(L,M)$.} 
\end{enumerate}
\end{lemma}

\begin{proof}
Let $\sigma_p$ denote an $M$-quasi-inversion of $X$ at $p$. Choose $x\in X_p$, and define the map $g:X\to X$ as $g:=f_3\circ\sigma_p\circ f_2\circ \sigma_p\circ f_1\circ \sigma_p$. Here $f_1:X\to X$ is an $L$-bi-Lipschitz map such that $f_1(p)=x$, $f_2:X\to X$ is an $L$-bi-Lipschitz map such that $f_2(\sigma_p(x))=p$, and $f_3:X\to X$ is an $L$-bi-Lipschitz map such that $f_3(\sigma_p(f_2(p)))=p$. We then observe that $g(p)=p$, and that, for any $a,b\in X$, we have
\begin{align*}
	d(g(a),g(b))&\eqx\frac{d(\sigma_p(f_1(\sigma_p(a))),\sigma_p(f_1(\sigma_p(b))))}{d(f_2(\sigma_p(f_1(\sigma_p(a)))),p)\,d(f_2(\sigma_p(f_1(\sigma_p(b)))),p)}\\
	&=\frac{d(\sigma_p(f_1(\sigma_p(a))),\sigma_p(f_1(\sigma_p(b))))}{d(f_2(\sigma_p(f_1(\sigma_p(a)))),f_2(\sigma_p(x)))\,d(f_2(\sigma_p(f_1(\sigma_p(b)))),f_2(\sigma_p(x)))}\\
	&\eqx\frac{d(\sigma_p(a),\sigma_p(b))}{d(f_1(\sigma_p(a)),p)\,d(f_1(\sigma_p(b)),p)}\cdot\frac{d(f_1(\sigma_p(a)),p)\,d(x,p)}{d(f_1(\sigma_p(a)),x)}\cdot\frac{d(f_1(\sigma_p(b)),p)\,d(x,p)}{d(f_1(\sigma_p(b)),x)}\\
	&=\frac{d(\sigma_p(a),\sigma_p(b))\,d(x,p)^2}{d(f_1(\sigma_p(a)),f_1(p))\,d(f_1(\sigma_p(b)),f_1(p))}\\
	&\eqx\frac{d(a,b)\,d(x,p)^2}{d(a,p)\,d(b,p)}\cdot d(a,p)\,d(b,p)\\
	&=d(x,p)^2\,d(a,b)
\end{align*}
Thus $g:X\to X$ is a $(K,d(x,p)^2)$-quasi-dilation at $p$, where $K=K(L,M)$. 

To verify $(2)$, assume that $X$ is connected. It follows that, for all $r>0$, there exists $x\in X$ such that $d(x,p)^2=r$ and a $(K,r)$-quasi-dilation as constructed above. 

To verify $(1)$, assume that $X$ is $N$-uniformly perfect. By definition, for all $r>0$, there exists a point $x\in X$ such that $\sqrt{2r}/N\leq d(x,p)< \sqrt{2r}$. Therefore, there exists a $(K,s)$-quasi-dilation $f:X\to X$ as constructed above, where $2r/N^2\leq s<2r$. Thus, for every $a,b\in X$, we have
\[\frac{2r}{KN^2}d(a,b)\leq \frac{s}{K}d(a,b)\leq d(f(a),f(b))\leq Ksd(a,b)\leq 2Krd(a,b).\]
Therefore, $f:X\to X$ is a $(2KN^2,r)$-quasi-dilation. 
\end{proof}

We distinguish between the connected and disconnected cases in \rf{L:quasi_dilation} in order to clarify quantitative dependence of the conclusions on the parameters pertaining to the assumptions. In a qualitative sense, a space $X$ satisfying the assumptions of \rf{L:quasi_dilation} is always uniformly perfect. This is the content of the following lemma.

\begin{lemma}\label{L:no_isolated}\label{L:uniformly_perfect}
Suppose that $X$ is an unbounded metric space. If $X$ is $L$-uniformly bi-Lipschitz homogeneous and $M$-quasi-invertible, then $X$ is uniformly perfect and, in particular, it has no isolated points.
\end{lemma}

\begin{proof}
First we prove that $X$ does not contain any isolated points. Let $\sigma_p:X_p\to X_p$ denote a quasi-inversion at $p\in X$, and let $(x_i)_{i=0}^{+\infty}$ denote a sequence of points in $X$ such that $d(p,x_i)\to+\infty$. It follows from the definition of a quasi-inversion that $d(\sigma_p(x_i),p)\to0$. Therefore, we conclude that $p$ is not an isolated point of $X$. By uniform bi-Lipschitz homogeneity, no point of $X$ is isolated.

Suppose $X$ is not uniformly perfect. Via uniform bi-Lipschitz homogeneity, we can assume that there exist positive numbers $r_k>0$ and $C_k\to+\infty$ such that, for each $k\in\mfN$, we have
\begin{equation}\label{E:empty}
A(p;r_k,C_kr_k)=\emptyset.
\end{equation}
If there exists $1\leq C<+\infty$ such that, for all $k\in\msf{N}$, we have $r_k\in[C^{-1},C]$, then $X$ is bounded, in contradiction to the assumption that $X$ is unbounded. Therefore, we may assume that there exists a subsequence of $(r_k)_{k\in\msf{N}}$ either converging to $0$ or diverging to $+\infty$. If there exists a subsequence $r_{n_k}\to+\infty$, then we may use the $M$-quasi-inversion at $p$ to ensure that, for each $k\in\msf{N}$, we have $A(p;M(C_{n_k}r_{n_k})^{-1},(Mr_{n_k})^{-1})=\emptyset$. 

By the above paragraph, we may assume that there exist sequences $r_k\to0$ and $C_k\to+\infty$ such that, for each $k\in \msf{N}$, we have \rf{E:empty}. Since $X$ is unbounded and contains no isolated points, we may assume that these empty annuli are maximal in the sense that there exist $x_k,y_k\in X$ such that $d(p,x_k)=r_k$ and $d(p,y_k)=C_kr_k$. Therefore, up to a subsequence, for each $k\in\msf{N}$ we have $C_{k+1}r_{k+1}\leq r_k$, and so $C_kr_k\to0$. Fix $x\in X$ such that $r:=d(p,x)$ satisfies $r^2>L$. Note that this is possible because $X$ is unbounded. By \rf{L:quasi_dilation} there exists an $(L,r^2)$-quasi-dilation $f:X\to X$ at $p$. Therefore, for every $k\in \mfN$, we have
\[A(p;Lr^2r_k,L^{-1}C_kr^2r_k)=\emptyset.\]
Since $L^{-1}r^2>1$ and $d(p,y_k)=C_kr_k$, it follows that, for every $k\in\mfN$, we have $Lr^2>C_k$. Since $C_k\to+\infty$, this is a contradiction. This contradiction reveals that $X$ must be uniformly perfect. 
\end{proof}

\subsection{Proof of Propositions~\ref{P:quasi-version} and~\ref{general_prop}}

\begin{remark}\label{R:strong}
The inverse of a $\theta$-quasi-M\"obius map is $\theta'$-quasi-M\"obius, where $\theta'(t)=\theta^{-1}(t^{-1})^{-1}$ (see \cite[pg.\,219]{Vaisala-qm}). Therefore, for use below, we remark that the inverse of a  $C$-strongly quasi-M\"obius map   is  $C$-strongly quasi-M\"obius.
\end{remark}

\begin{proof}[Proof of \rf{P:quasi-version}]
We first prove sufficiency. To verify that $X$ is uniformly bi-Lipschitz homogeneous, we proceed as in Proposition~\ref{P:equivalence1b}. Let $h:\sph_p(X)\to \sph_p(X)$ denote a $C$-strongly quasi-M\"obius map such that $h(\infty)=\infty$. We claim that $h$ is a quasi-similarity of $X$. In other words, there exists $L=L(C)$ and $\lambda>0$ such that, for any $a,b\in X$, we have
\[d(h(a),h(b))\eqx_L \lambda\, d(a,b).\]
To verify this claim, let $a,b,c\in X$ be a triple of distinct points. Then we have
\begin{align*}
d(h(a),h(b))&\eqx_4 \hat{d}_p(h(a),h(b))(1+d(h(a),p))(1+d(h(b),p))\\
&=\frac{\hat{d}_p(h(a),h(b))\hat{d}_p(h(c),\infty)}{\hat{d}_p(h(c),\infty)}(1+d(h(a),p))(1+d(h(b),p))\\
&\eqx_C\frac{\hat{d}_p(a,b)\hat{d}_p(c,\infty)}{\hat{d}_p(a,\infty)\hat{d}_p(b,c)}\frac{(1+d(h(a),p))(1+d(h(b),p))}{\hat{d}_p(h(c),\infty)}\hat{d}_p(h(a),\infty)\hat{d}_p(h(b),h(c))\\
&\eqx_{4^7}d(a,b)\frac{d(h(b),h(c))}{d(b,c)}.
\end{align*}
Here we have used \rf{R:strong} and omitted some of the straightforward calculations. Since the above comparability statements hold for any triple of distinct points $a,b,c\in X$, we conclude that $d(h(b),h(c))\eqx_L \lambda\, d(b,c)$ for some $L=L(C)$ and $\lambda>0$. Therefore, any $C$-strongly quasi-M\"obius map of $\sph_p(X)$ fixing $\infty$ is quasi-similarity mapping of $X$. 

Given any $a\in X$, let $h:\sph_p(X)\to\sph_p(X)$ denote a $C$-strongly quasi-M\"obius map fixing $\infty$ such that $h(a)=p$. Let $\lambda>0$ and $L=L(C)$ denote the corresponding constants such that, for any $x,y\in X$, we have $d(h(x),h(y))\eqx_L \lambda\, d(x,y)$. If $L^{-1}\leq \lambda\leq L$, then we conclude that $h$ is $L^2$-bi-Lipschitz. If $\lambda<L^{-1}$ (or $\lambda>L$) then $h$ (or $h^{-1}$) is a strict contraction mapping $X$ to itself. Since $X$ is proper, it is complete. Therefore, by the Banach Fixed Point theorem, there exists a point $o\in X$ such that $h(o)=o$. Now let $g:\sph_p(X)\to \sph_p(X)$ denote a $C$-strongly quasi-M\"obius map fixing $\infty$ and sending $o$ to $p$. Write $\mu>0$ and $M=M(C)$ to denote constants such that, for any $x,y\in X$, we have $d(g(x),g(y))\eqx_M\mu\, d(x,y)$. We consider the map $g\circ h^{-1} \circ g^{-1}\circ h$. First, we note that this map sends $a$ to $p$. Then, we note that this map is $(ML)^2$-bi-Lipschitz. It follows that $X$ is uniformly bi-Lipschitz homogeneous. 

Next, we demonstrate that $X$ admits a quasi-inversion. To this end, let $f:\sph_p(X)\to\sph_p(X)$ denote a $C$-strongly quasi-M\"obius map such that $f(p)=\infty$ and $f(\infty)=p$. Then, for any $a,b\in X_p$ such that $a\not=b$, we have
\begin{align*}
d(p,f(a))&\eqx_4\hat{d}_p(p,f(a))(1+d(p,f(a)))\\
&=\hat{d}_p(p,f(a))\hat{d}_p(\infty,f(b))\frac{(1+d(p,f(a)))}{\hat{d}_p(\infty,f(b))}\\
&\eqx_C\frac{\hat{d}_p(\infty,a)\hat{d}_p(p,b)}{\hat{d}_p(\infty,b)\hat{d}_p(a,p)}\frac{1+d(p,f(a))}{\hat{d}_p(\infty,f(b)}\hat{d}_p(p,f(b))\hat{d}_p(f(a),\infty)\\
&\eqx_{4^7}\frac{1}{d(a,p)}\frac{(1+d(p,f(a)))(1+d(a,p))(1+d(b,p))(1+d(f(b),p))}{(1+d(p,f(a)))(1+d(a,p))(1+d(b,p))(1+d(f(b),p))}d(p,b)d(p,f(b))\\
&=\frac{d(p,b)d(p,f(b))}{d(a,p)}.
\end{align*}
The above statement again utilizes \rf{R:strong}. Since the above comparabilities hold for any $b\not=a$ in $X$, we conclude that there exist constants $L=L(C)$ and $r>0$ such that, for any $b\in X$, we have 
\begin{equation}\label{E:scale}
d(f(b),p)\eqx_L r\cdot d(p,b)^{-1}.
\end{equation}
Now let $a,b\in X_p$ be such that $a\not=b$. Using the same function $f$ as above, we observe that
\begin{align*}
d(f(a),f(b))&\eqx_4\hat{d}_p(f(a),f(b))(1+d(f(a),p))(1+d(f(b),p))\\
&=\hat{d}_p(f(a),f(b))\hat{d}_p(f(p),f(\infty))(1+d(f(a),p))(1+d(f(b),p))\\
&\eqx_C\frac{\hat{d}_p(a,b)\hat{d}_p(p,\infty)}{\hat{d}_p(a,\infty)\hat{d}_p(b,p)}\hat{d}_p(f(a),p)\hat{d}_p(f(b),\infty)(1+d(f(a),p))(1+d(f(b),p))\\
&\eqx_{4^6}\frac{d(a,b)d(f(a),p)}{d(b,p)}\frac{(1+d(f(a),p))(1+d(f(b),p))(1+d(a,p))(1+d(b,p))}{(1+d(f(a),p))(1+d(f(b),p))(1+d(a,p))(1+d(b,p))}\\
&=\frac{d(a,b)d(f(a),p)}{d(b,p)}\eqx_{Lr}\frac{d(a,b)}{d(a,p)d(b,p)},
\end{align*}
where  the final comparison follows from \rf{E:scale}. Therefore, $f$ is a quasi-inversion of $X$.  

\medskip
To prove necessity, we assume that $X$ is uniformly $L$-bi-Lipschitz homogeneous and admits a $K$-quasi-inversion at some point $p\in X$. To confirm that $\sph_p(X)$ is 2-point uniformly strongly quasi-M\"obius homogeneous, we mimic the proof of \rf{P:equivalence1a}. Given $p\in X$, let $\sigma_p$ denote a $K$-quasi-inversion of $X$ at $p$. We show that every point $(a,b)\in(\sph_p(X)\times\sph_p(X))\setminus \Delta$ can be mapped to $(\infty,p)$ via a uniformly strongly quasi-M\"obius map of $\sph_p(X)$. If $a=\infty$, then simply map $b$ to $p$ via an $L$-bi-Lipschitz map of $X$. Here we note that any $L$-bi-Lipschitz map of $X$ is an $L^4$-strongly quasi-M\"obius map of $\sph_p(X)$. If $a\not=\infty$, then we map $a$ to $p$ via an $L$-bi-Lipschitz map of $X$ before applying $\sigma_p$. This composition is an $(LK)^4$-strongly quasi-M\"obius map of $\sph_p(X)$. Thus we return to the case that $a=\infty$. 
\end{proof}

\begin{proof}[{Proof of Proposition~\ref{general_prop}}]
Assume $f:X\to Y$ is $L$-bi-Lipschitz, then $f$ is $L^4$-strongly quasi-M\"obius. Furthermore, we note that if $f$ is a similarity mapping, then $f$ is M\"obius.

Conversely, assume that $h:X\to Y$ is $C$-strongly quasi-M\"obius. We first claim that $h$ extends homeomorphically to $h:\hat{X}\to\hat{Y}$ such that $h(\infty)=\infty$. Indeed, because $X$ and $Y$ are proper, both $h$ and $h^{-1}$ must send bounded sets to bounded sets. The claim follows. Therefore, we may view $h$ as a $C$-strongly quasi-M\"obius map $h:\Sph_p(X)\to\Sph_q(Y)$ for some points $p\in X$ and $q\in Y$. 

Let $a,b,c\in X$ be a triple of distinct points. We observe that
\begin{align*}
d(h(a),h(b))&=s_p(h(a),h(b))(1+d(h(a),p))(1+d(h(b),p))\\
&=\frac{s_p(h(a),h(b))s_p(h(c),\infty)}{s_p(h(c),\infty)}(1+d(h(a),p))(1+d(h(b),p))\\
&\eqx_C\frac{s_p(a,b)s_p(c,\infty)}{s_p(a,\infty)s_p(b,c)}\frac{(1+d(h(a),p))(1+d(h(b),p))}{s_p(h(c),\infty)}s_p(h(a),\infty)s_p(h(b),h(c))\\
&=\frac{d(a,b)d(h(b),h(c))(1+d(a,p))(1+d(b,p))(1+d(c,p))}{d(b,c)(1+d(a,p))(1+d(b,p))(1+d(c,p))}\,\cdot\\
& \hspace{64pt}\cdot\,\frac{(1+d(h(a),p))(1+d(h(b),p))(1+d(h(c),p))}{(1+d(h(a),p))(1+d(h(b),p))(1+d(h(c),p))}\\
&=d(a,b)\frac{d(h(b),h(c))}{d(b,c)}.
\end{align*}
Since the above equalities hold for any triple of distinct points $a,b,c\in X$, we conclude that there exists $\lambda>0$ such that, for any $a,b\in X$, we have $d(h(a),h(b))\eqx_C\lambda\cdot d(a,b)$. Therefore, $h$ is $(C\lambda)$-bi-Lipschitz. When $C=1$, the map $h$ is a $\lambda$-similarity. 
\end{proof}

\subsection{Characterizing quasi-invertibility}\label{S:char}

This subsection records a few useful technical results and culminates in the statement and proof of of \rf{P:invertible_char}. We begin with the following lemma, which extends \cite[Lemma 3.2]{BHX-inversions} in the case of quasi-sphericalization.

\begin{lemma}\label{L:BL_on_sphere}
Let $f:X\to X$ be a homeomorphism of a metric space, and let $p\in X$. If $f$ is $L$-bi-Lipschitz, then $f:\sph_p(X)\to\sph_p(X)$ is $C$-bi-Lipschitz, where $C=C(L,d(f(p),p))$. 
\end{lemma}

\begin{proof}
Given $a\in X$, we first note that, for   $C_0:=L(1+d(f(p),p))$,
\begin{align*}
1+d(f(a),p)&\leq1+d(f(a),f(p))+d(f(p),p)\\
&\leq 1+Ld(a,p)+d(f(p),p)\leq C_0(1+d(a,p)).
\end{align*}
 Therefore, given two points $a,b\in X$, we have
\begin{equation}\label{E:lower_dilate}
\hat{d}_p(f(a),f(b))\geq\frac{d(a,b)}{4LC_0^2(1+d(a,p))(1+d(b,p))}\geq\frac{\hat{d}_p(a,b)}{4LC_0^2}.
\end{equation}

To obtain a relevant upper bound on $\hat{d}_p(f(a),f(b))$, we consider two cases. 

\textit{Case 1:} $d(f(a),p)\leq 1$. In this case, we note that, for $C_1:=(1+L+Ld(f(p),p))$,
\begin{align*}
1+d(a,p)&\leq 1+Ld(f(a),f(p))\leq 1+L(d(f(a),p)+d(f(p),p))\\
&\leq 1+L+Ld(f(p),p)\leq C_1(1+d(f(a),p)).
\end{align*}

\textit{Case 2:} $d(f(a),p)>1$. We consider two subcases. First, suppose that $d(f(a),p)\geq (2L)^{-1}d(a,p)$. Then we note that, for $C_2:=2L$,
\begin{align*}
1+d(a,p)\leq 1+2Ld(f(a),p)\leq C_2(1+d(f(a),p)).
\end{align*}
 Next, suppose that $d(f(a),p)<(2L)^{-1}d(a,p)$. Then we note that
\begin{align*}
d(f(p),p)&\geq d(f(p),f(a))-d(f(a),p)\geq L^{-1}d(a,p)-d(f(a),p)\geq (2L)^{-1}d(a,p).
\end{align*}
Therefore, $d(a,p)\leq 2Ld(f(p),p)\leq 2Ld(f(p),p)d(f(a),p)$, and so, for $C_3:=2Ld(f(p),p)>1$,
\[1+d(a,p)\leq C_3(1+d(f(a),p)),\]

Considering Case 1 and Case 2 together, we conclude that, for any two points $a,b\in X$, we have  
\begin{equation}\label{E:upper_dilate}
\hat{d}_p(f(a),f(b))\leq \frac{LC_4^2\,d(a,b)}{(1+d(a,p))(1+d(b,p))}\leq 4LC_4^2\,\hat{d}_p(a,b),
\end{equation}
where $C_4=\max\{C_1,C_2,C_3\}$. Combining \rf{E:lower_dilate} and \rf{E:upper_dilate}, we reach the desired conclusion.
\end{proof}

\rf{L:BL_on_sphere} can be used to prove the following lemma  regarding the behavior of quasi-inversions with respect to the quasi-sphericalized distance. 

\begin{lemma}\label{L:BL_inversion_2pt}
Suppose $X$ is an $L$-bi-Lipschitz homogeneous metric space. For $p,x\in X$, any $M$-quasi-inversion $\sigma_x:X_x\to X_x$ is a $C$-bi-Lipschitz self-homeomorphism of $\sph_p(X)$, with $C=C(L,M,d(p,x))$. If $p=x$, then we reach the same conclusion with $C=4M^3$. 
\end{lemma}

\begin{proof}
Let $x,y\in X$ be given, and fix a point $p\in X$. Then 
\begin{align*}
\hat{d}_p(\sigma_p(x),\sigma_p(y))&\leq\frac{d(\sigma_p(x),\sigma_p(y))}{(1+d(\sigma_p(x),p))(1+d(\sigma_p(y),p))}\\
&\leq\frac{M\,d(x,y)}{d(p,x)d(p,y)}\frac{1}{(1+(M\,d(p,x))^{-1})(1+(M\,d(p,y))^{-1})}\\
&=\frac{M\,d(x,y)}{(d(p,x)+M^{-1})(d(p,y)+M^{-1})}\leq 4M^3\hat{d}_p(x,y).
\end{align*}
On the other hand, we have
\begin{align*}
\hat{d}_p(\sigma_p(x),\sigma_p(y))&\geq \frac{d(x,y)}{4M\,d(p,x)d(p,y)}\frac{1}{(1+Md(p,x)^{-1})(1+Md(p,y)^{-1})}\\
&=\frac{d(x,y)}{4M(d(p,x)+M)(d(p,y)+M)}\geq\frac{\hat{d}_p(x,y)}{4M^3}.
\end{align*}
Similar calculations produce the same conclusion when $y=\infty$ or $x=\infty$. Thus we reach the desired conclusion when $p=x$. 

Now let $f$ denote an $L$-bi-Lipschitz self-homeomorphism $f:X\to X$ such that $f(x)=p$. Then we note that $f\circ \sigma_x\circ f^{-1}: X_p\to X_p$ is an $L^4M$-quasi-inversion at $p$. It follows from the above estimates and \rf{L:BL_on_sphere} that $\sigma_x:\sph_p(X)\to\sph_p(X)$ is $C$-bi-Lipschitz, where $C=C(L,M,d(f(p),p))$. Since $d(f(p),p)\eqx_Ld(p,x)$, we reach the desired conclusion.
\end{proof}

Before stating and proving \rf{P:invertible_char} we record the following observations describing the metric implications of iterated quasi-sphericalizations and/or quasi-inversions. These observations are analogous to \cite[Propositions 3.3 and 3.4]{BHX-inversions}.

\begin{lemma}\label{L:iterations}
Suppose $X$ is an unbounded metric space and $p\in X$. 
\begin{enumerate}
	\item{The space $\inv_p(\sph_p(X))$ is bi-Lipschitz equivalent to $\inv_p(X)$ via the identity map.}
	\item{The space $\sph_\infty(\inv_p(X))$ is bi-Lipschitz equivalent to $\sph_p(X)$ via the identity map.}
\end{enumerate}
\end{lemma}

\begin{proof}
The lemma follows from \rf{E:inverted} and \rf{E:sphere_dist}. We first prove $(1)$. Let $d'$ denote the quasi-inverted distance $(\hat{d}_p)_p$ on $\hat{X}_p$. For any $x,y\in X_p$, we have
\[d'(x,y)\eqx\frac{\hat{d}_p(x,y)}{\hat{d}_p(x,p)\hat{d}_p(y,p)}\eqx\frac{d(x,y)}{d(x,p)d(y,p)}\frac{(1+d(x,p))(1+d(y,p))}{(1+d(x,p))(1+d(y,p))}\eqx d_p(x,y).\]
If $y=\infty$, then we note that 
\[d'(x,\infty)\eqx\frac{\hat{d}_p(x,\infty)}{\hat{d}_p(x,p)\hat{d}_p(\infty,p)}\eqx\frac{1+d(x,p)}{d(x,p)}\frac{1}{1+d(x,p)}\eqx d_p(x,\infty).\]

To prove $(2)$, let $d''$ denote the quasi-sphericalized distance $\widehat{(d_p)}_\infty$ on $\hat{X}$. For any $x,y\in X$, we have
\begin{align*}
d''(x,y)&\eqx\frac{d_p(x,y)}{(1+d_p(x,\infty))(1+d_p(y,\infty))}\\
&=\frac{d(x,y)}{d(x,p)\,d(y,p)}\frac{1}{(1+d(x,p)^{-1})(1+d(y,p)^{-1})}\\
&=\frac{d(x,y)}{(1+d(x,p))(1+d(y,p))}\eqx \hat{d}_p(x,y).
\end{align*}
If $y=\infty$, similar calculations reveal that $d''(x,\infty)\eqx\hat{d}_p(x,\infty)$. 
\end{proof}

At this point we are ready to state and prove \rf{P:invertible_char}. As stated above, the purpose of this result is to provide equivalent characterizations of quasi-invertibility under the assumption that $X$ is uniformly bi-Lipschitz homogeneous. 

\begin{proposition}\label{P:invertible_char}
Suppose $X$ is an unbounded and uniformly bi-Lipschitz homogeneous metric space. Given any point $p\in X$, the following statements are equivalent:
\begin{enumerate}
  \item{$X$ admits a quasi-inversion at $p$.}
  \item{$X$ is bi-Lipschitz equivalent to $\inv_p(X)$.}
  \item{$\inv_p(X)$ is uniformly bi-Lipschitz homogeneous.}
  \item{$\sph_p(X)$ is uniformly bi-Lipschitz homogeneous.}
\end{enumerate}
\end{proposition}

\begin{proof}
We prove $(3)\Rightarrow(1)\Rightarrow(4)\Rightarrow(2)\Rightarrow(3)$. 

Suppose first that $\inv_p(X)$ is uniformly bi-Lipschitz homogeneous, and fix some $q\in X_p$. Let $f:X\to X$ denote a bi-Lipschitz map such that $f(p)=q$. Let $g:\inv_p(X)\to\inv_p(X)$ denote a bi-Lipschitz map such that $g(q)=\infty$. Lastly, let $h:X\to X$ denote a bi-Lipchitz map such that $h(g(\infty))=p$. We claim that the composition $h\circ g\circ f:X_p\to X_p$ is a quasi-inversion. Indeed, we first note that $h(g(f(p)))=\infty$ and $h(g(f(\infty)))=p$. Furthermore, for any $x,y\in X$, we have
\begin{align*}
d(h(g(f(x))),h(g(f(y))))&\eqx d(g(f(x)),g(f(y)))\\
&\eqx\frac{d(f(x),f(y))\,d(g(f(x)),p)\,d(g(f(y)),p)}{d(f(x),p)\,d(f(y),p)}\\
&\eqx\frac{d(x,y)\,d(g(f(x)),p)\,d((f(y)),p)}{d(f(x),p)\,d(f(y),p)}
\end{align*}
Here we use \rf{E:inverted}. We then note that 
\begin{align*}
\frac{1}{d(g(f(x)),p)}&= d_p(g(f(x)),\infty)= d_p(g(f(x)),g(q))\eqx d_p(f(x),q)\\
&=d_p(f(x),f(p))\eqx\frac{d(x,p)}{d(f(x),p)\,d(q,p)}.
\end{align*}
It follows that 
\begin{align*}
d(h(g(f(x))),h(g(f(y))))&\eqx\frac{d(x,y)\,d(f(x),p)\,d(f(y),p)d\,(q,p)^2}{d(x,p)\,d(y,p)\,d(f(x),p)\,d(f(y),p)}\\
&\eqx\frac{d(x,y)}{d(x,p)\,d(y,p)}.
\end{align*}
Here we note that the final comparability depends on the quantity $d(q,p)$. We also note that our claim regarding $h\circ g\circ f$ has been verified. Therefore, we conclude that $(3)\Rightarrow(1)$. 

\medskip
Now we suppose that $X$ admits an $M$-quasi-inversion $\sigma_p$. We claim there exists $C\geq 1$ such that any point $q\in\sph_p(X)$ can be mapped to $p$ by an $C$-bi-Lipschitz self-homeomorphism of $\sph_p(X)$. To verify this claim, we first assume that $q\in B(p;1)\subset X$. Based on the assumption that $X$ is $L$-bi-Lipschitz homogeneous, for some $L\geq1$, let $f:X\to X$ denote an $L$-bi-Lipschitz map such that $f(q)=p$. By \rf{L:BL_on_sphere}, we conclude that $f:\sph_p(X)\to\sph_p(X)$ is $K_1$-bi-Lipschitz, where $K_1=K_1(L,d(f(p),p))$. Since $d(f(p),p)=d(f(p),f(q))\leq L$, we have $K_1=K(L)$. Next, we assume that $q\not\in B(p;1)$. Then $q':=\sigma_p(q)$ satisfies $d(p,q')\leq M$. Letting $g:X\to X$ denote an $L$-bi-Lipschitz map such that $g(q')=p$, \rf{L:BL_on_sphere} and \rf{L:BL_inversion_2pt} allow us to conclude that $g\circ \sigma_p:\sph_p(X)\to\sph_p(X)$ is $K_2$-bi-Lipschitz, with $K_2=K_2(L,M)$. It follows that $\sph_p(X)$ is $C^2$-bi-Lipschitz homogeneous, with $C=\max\{K_1,K_2\}$. Thus we prove $(1)\Rightarrow(4)$. 

\medskip
Next, suppose $\sph_p(X)$ is uniformly bi-Lipschitz homogeneous. Therefore, there exists a bi-Lipschitz homeomorphism $f:\sph_p(X)\to\sph_p(X)$ such that $f(\infty)=p$. By \cite[Lemma 3.2]{BHX-inversions}, we conclude that $\inv_\infty(\sph_p(X))$ is bi-Lipschitz homeomorphic to $\inv_p(\sph_p(X))$. By \cite[Proposition 3.4]{BHX-inversions}, we conclude that $\inv_\infty(\sph_p(X))$ is bi-Lipschitz homeomorphic to $X$, and, by \rf{L:iterations}$(1)$, we conclude that $\inv_p(\sph_p(X))$ is bi-Lipschitz homoemorphic to $\inv_p(X)$. Thus $X$ is bi-Lipschitz homeomorphic to $\inv_p(X)$, and we establish $(4)\Rightarrow(2)$. 

\medskip
Lastly, we note that $(2)\Rightarrow(3)$ is almost immediate. Indeed, if $X$ is $L$-bi-Lipschitz homogeneous and $M$-bi-Lipschitz equivalent to $\inv_p(X)$, for some numbers $L,M\geq1$, then $\inv_p(X)$ is $LM^2$-bi-Lipschitz homogeneous. Thus $(2)\Rightarrow(3)$. 
\end{proof}

\begin{remark}
We note that \rf{P:invertible_char} clarifies the relationship between the assumptions of uniform bi-Lipschitz homogeneity and quasi-invertibility with the terminology \textit{inversion invariant bi-Lipschitz homongeneity} as used, for example, in \cite{Freeman-iiblh}.
\end{remark}


\subsection{Additional consequences of bi-Lipschitz homogeneity}

Given a proper, uniformly bi-Lipschitz homogeneous metric space $X$ and a compact subset $K\subset X$, the next lemma demonstrates that one can map a point $x\in K$ to a point $y\in X$ using a bi-Lipschitz map that almost fixes points of $K$, provided that $x$ and $y$ are near enough to each other. 

\begin{lemma}\label{L:q_translate}
Suppose $X$ is a proper and $L$-bi-Lipschitz homogeneous metric space. For every $x\in X$, $\ep>0$, and compact set $K\subset X$ containing $x$, there exists $\delta>0$ such that, for any $y\in B(x;\delta)$ there exists an $L^2$-bi-Lipschitz homeomorphism $h:X\to X$ such that $h(x)=y$ and $\sup_{z\in K}d(h(z),z)<\ep$.
\end{lemma}

\begin{proof}
Let $x\in X$, $\ep>0$, and a compact set $K\subset X$ contaning $x$ be fixed. For a given $n\in\msf{N}$, set $K_n:=\overline{B(x;n)}$, the closure of the ball of radius $n$. Let $(x_m)_{m=1}^{+\infty}$ denote any sequence of points in $X$ such that $d(x_m,x)\to0$. For each $m$, write $x_{1,m}:=x_m$. Suppose there exists a sequence of $L$-bi-Lipschitz homeomorphisms $f_{1,m}:X\to X$ such that $f_{1,m}(x)=x_{1,m}$. Since $X$ is proper, we can assume (up to a subsequence) that $f_{1,m}$ uniformly converges on $K_1$ to an $L$-bi-Lipschitz embedding $f_1:K_1\to X$ such that $f_1(x)=x$. Inductively define sequences of points $(x_{n,m})_{m=1}^{+\infty}$ such that, for $n\geq2$, each $(x_{n,m})_{m=1}^{+\infty}$ is a subsequence of $(x_{n-1,m})_{m=1}^{+\infty}$. Furthermore, define sequences $(f_{n,m})_{m=1}^{+\infty}$ of $L$-bi-Lipschitz self-homeomorphisms of X such that, for $n\geq2$, each $(f_{n,m})_{m=1}^{+\infty}$ is a subsequence of $(f_{n-1,m})_{m=1}^{+\infty}$ such that $f_{n,m}(x)=x_{n,m}$. We can also assume that $(f_{n,m})_{m=1}^{+\infty}$ converges uniformly on $K_n$ to an $L$-bi-Lipschitz embedding $f_n:K_n\to X$ such that $f_n(x)=x$. Note also that $f_n=f_{n-1}$ when restricted to $K_{n-1}$. The sequence $(f_n)_{n=1}^{+\infty}$ locally uniformly converges to an $L$-bi-Lipschitz homeomorphism $f:X\to X$ such that $f(x)=x$. Fix $N\in\msf N$ such that $K\subset K_N$. For $n\geq N$, define $g_n:=f_{n,n}\circ f^{-1}$. Then $g_n(x)=x_{n,n}$, and $g_n$ uniformly converges to the identity map on $K$. 

The above paragraph allows us to conclude that, up to a subsequence, for any sequence of points $x_n\to x$, there exists $N\in\msf N$ such that for any $n\geq N$, there exists an $L^2$-bi-Lipschitz map $g_n:X\to X$ such that $g_n(x)=x_n$ and $\max_{z\in K}d(g_n(z),z)<\ep$. This implies the existence of $\delta>0$ such that, for any $y\in B(x;\delta)$, there exists an $L^2$-bi-Lipschitz map $h:X\to X$ such that $h(x)=y$ and $\max_{z\in K}d(h(z),z)<\ep$.
\end{proof}

Regarding the next lemma, we recall that a point $x\in X$ is called a {\em strong cut point} if $X\setminus \{x\}$ has exactly two connected components.
\begin{lemma}\label{L:cut}
Let  $X$ be a proper and $L$-bi-Lipschitz homogeneous metric space. Assume that $X$ is path connected and locally path connected.
Then any cut point of $X$ is a strong cut point.
\end{lemma}
\begin{proof}
 Suppose $x$ is a cut point of $X$. Suppose, by way of contradiction, that $x$ is not a strong cut point. In other words, suppose there exist three points $z_1$, $z_2$, and $z_3$ in three different connected components $X_1$, $X_2$, and  $X_3$ of $X\setminus\{x\}$, respectively. Since $X$ is path connected, there exist curves $\gamma_i$ joining $z_i$ to $x$, for $i=1,2,3$, and we may further assume that $\gamma_i':=\gamma_i\setminus\{x\} $ is connected. Note that $\gamma_1'$, $\gamma_2'$ and $\gamma_3'$ are contained in different components of $X\setminus\{x\}$, and are therefore pairwise disjoint. Let $U_2$ and $U_3$ denote path connected  neighborhoods of $z_2$ and $z_3$, respectively, that do not contain $x$. Hence, we have  $U_2\subset X_2$ and $U_3\subset X_3$. 

Choose $\varepsilon>0$ such that $B(z_i; \varepsilon)\subset U_i$, for $i=2,3$. 
Apply  \rf{L:q_translate} with $K:=\{x,z_2, z_3\}$. Thus, there exists $\delta>0$ such that, for any $y\in \gamma_1'\cap B(x;\delta)$, there exists a $L^2$-bi-Lipschitz homeomorphism $h:X\to X$ such that $h(x)=y$ and $h(z_i)\in B(z_i;\varepsilon)$, for $i=2,3$. 

By the construction of $U_2$ and $U_3$, there exist curves $\eta_i\subset  U_i$ joining $z_i$ to $h(z_i)$, for $i=2,3$, Therefore, the connected set  $\mu_i:=\eta_i \cup h(\gamma_i)\cup \gamma_1'$  contains both $z_i$ and $z_1$. 
Notice that $\eta_2  \cup \eta_3 \cup \gamma_1'$ does not contain $x$. Therefore $x\in h(\gamma_2) \cap h(\gamma_3)=h(\gamma_2\cap\gamma_3)=h(x)$. However, $h(x)\not=x$. The contradiction ends the proof.
\end{proof}

Our next step is to prove that, given two points $x,y\in X$ along with a compact neighborhood  $K$ containing both $x$ and $y$, one can find a map that is bi-Lipschitz on $K$, fixes $x$, and sends $y$ to any point within a small enough neighborhood of $y$. 

\begin{lemma}\label{L:fixed_point_d}
Suppose $X$ is unbounded, proper, $L$-bi-Lipschitz homogeneous, and $M$-quasi-invertible. Let $x\in X$ and $0<R<\infty$. There exists $C=C(L,M,R)$ such that, for any $y\in B(x;R)\setminus\{x\}$, there exists $\delta>0$ such that, for any point $u\in B(y;\delta)$, there exists a homeomorphism $f:\sph_x(X)\to\sph_x(X)$ such that, for any $a,b\in B(x;R)$, we have $d(f(a),f(b))\eqx_C d(a,b)$. Moreover, $f(x)=x$ and $f(y)=u$. 
\end{lemma}

\begin{proof}
Fix distinct points $x,y\in X$ and $R>0$. We claim there exist constants $C=C(L,M)<+\infty$ and $\delta>0$ such that, for any $u\in B(y;\delta)$, there exists a $C$-bi-Lipschitz homeomorphism of $f:\sph_x(X)\to\sph_x(X)$ such that $f(x)=x$, $f(y)=u$, and $\hat{d}_x(f(\infty),\infty)<(2C(1+R))^{-1}$. 

To verify this claim, choose $\varepsilon\in(0,1)$ (whose value is to be determined below) and $N\in\msf N$ such that $x\in K:=\overline{B(v;N)}$, where $v:=\sigma_x(y)$. By \rf{L:BL_inversion_2pt}, the map $\sigma_x^{-1}:\sph_x(X)\to\sph_x(X)$ is $C_1$-bi-Lipschitz, with $C_1=C_1(M)$. Therefore, for any $a,b\in\hat{X}$, if $\hat{d}_x(a,\sigma_x(b))<\varepsilon/C_1$, then $\hat{d}_x(\sigma_x^{-1}(a),b)<\varepsilon$. 

By \rf{L:q_translate}, there exists $\delta_1>0$ such that, for any $u\in X$ satisfying $\sigma_x(u)\in B(v;\delta_1)$, there exists an $L^2$-bi-Lipschitz homeomorphism $h_u:X\to X$ such that $h_u(v)=\sigma_x(u)$ and 
\[\max_{a\in K}d(h_u(a),a)<\varepsilon/C_1.\]
In particular, $d(h_u(x),x)<1$. By \rf{L:BL_on_sphere}, we conclude that $h_u:\sph_x(X)\to\sph_x(X)$ is $C_2$-bi-Lipschitz, with $C_2=C_2(L)$. 

For each $u$ such that $\sigma_x(u)\in B(v;\delta_1)$, define $g_u:=\sigma_x^{-1}\circ h_u\circ \sigma_x$. Choose $\delta_2>0$ small enough to ensure that $\sigma_x(B(y;\delta_2))\subset B(v;\delta_1)$. By the two preceding paragraphs, $\{g_u\,|\,u\in B(y;\delta_2)\}$ is a collection of uniformly $C_3$-bi-Lipschitz self-homeomorphisms of $\sph_x(X)$, where $C_3=C_3(L, M)$. Here we homeomorphically extend $h_u$ such that $h_u(\infty)=\infty$. Thus we have $g_u(x)=x$ and $g_u(y)=u$, and we note that $\hat{d}_x(h_u(x),x)\leq d(h_u(x),x)<\varepsilon/C_1$. Therefore, $\hat{d}_x(g_u(\infty),\infty)<\varepsilon$, and, if we choose $\varepsilon=(2C_3(1+R))^{-1}<1$, then our claim is verified. 

To conclude the proof of the lemma, choose $x\in X$ and $R>0$. Then choose $y\in B(x;R)\setminus\{x\}$. By the above claim, there exist constants $C=C(L,M)$ and $\delta>0$ such that, for any $u\in B(y;\delta)$, there exists a $C$-bi-Lipschitz homeomorphism $g_u:\sph_x(X)\to \sph_x(X)$ such that $g_u(x)=x$, $g_u(y)=u$, and $\hat{d}_x(g_u(\infty),\infty)<(2C(1+R))^{-1}$. Here we may assume that $\delta$ is small enough to ensure that $B(y;\delta)\subset B(x;R)$. For any $a\in B(x;R)$, it follows from the triangle inequality and the properties of $g_u$ that
\begin{equation}\label{E:distance}
d(g_u(a),x)=\frac{1}{\hat{d}_x(g_u(a),\infty)}-1\leq2C(1+R)-1\leq 2C(1+R).
\end{equation}
Set $C_4:=2C(1+R)$. Via \rf{E:distance}, for any $a,b\in B(x;R)$, we have
\begin{align*}
d(g_u(a),g_u(b))&\leq 4\hat{d}_x(g_u(a),g_u(b))(1+d(g_u(a),x))(1+d(g_u(b),x))\\
&\leq 4C\hat{d}_x(a,b)(1+C_4)^2\leq 4C(1+C_4)^2d(a,b).
\end{align*}
On the other hand, we have
\begin{align*}
d(g_u(a),g_u(b))&\geq\hat{d}_x(g_u(a),g_u(b))\geq\frac{\hat{d}_x(a,b)}{C}\\
&\geq\frac{d(a,b)}{4C(1+d(a,x))(1+d(b,x))}\geq \frac{d(a,b)}{4C(1+R)^2}\geq \frac{d(a,b)}{4C(1+C_4)^2}.
\end{align*}
Defining $C_5:=4C(1+C_4)^2$, for any $a,b\in B(x;R)$ we have $d(g_u(a),g_u(b))\eqx_{C_5}d(a,b)$.
\end{proof}

\subsection{Proof of Theorem~\ref{T:main:connected}}

In this section we prepare for and present the proof of Theorem~\ref{T:main:connected}. We begin by establishing a few technical results. The first of these lemmas is of a general nature and does not rely on the assumption of bi-Lipschitz homogeneity.

\begin{lemma}\label{L:porous}
Let $X$ denote a proper metric space. Fix constants $C,R<+\infty$ and a point $x\in X$. If each open ball in $X$ has infinite Hausdorff $1$-measure, then there exists $\delta>0$ such that, for any rectifiable curve $\gamma\subset X$ such that $\textrm{Length}(\gamma)\leq C$, there exists $y\in B(x;R)$ such that $B(y;\delta)\cap \gamma=\emptyset$.  
\end{lemma}

\begin{proof}
By way of contradiction, suppose that there exists a sequence of positive numbers $\delta_n\to0$ and a sequence of rectifiable curves $\gamma_n\subset X$ such that, for every $q\in B(x;R)$, we have $B(q;\delta_n)\cap \gamma_n\not=\emptyset$. Furthermore, for every $n\in\msf{N}$, we have $\textrm{Length}(\gamma_n)=C_n\leq C$. For each $n\in\msf{N}$, we write $\alpha_n:[0,C]\to\gamma_n$ to denote a parametrization such that $\alpha_n|_{[0,C_n]}$ is an arclength parameterization of $\gamma_n$ and $\alpha_n$ is constant on $[C_n,C]$. Thus each $\alpha_n$ is $1$-Lipschitz. Since $X$ is proper and, for every $n\in\msf{N}$, we have $\gamma_n\cap B(x;R)\not=\emptyset$, by Arzela-Ascoli we can assume that (up to a subsequence) the maps $\alpha_n$ are uniformly convergent to a $1$-Lipschitz map $\alpha_\infty:[0,C]\to X$. Write $\gamma_\infty=\alpha_\infty([0,C])$. Thus we have $d_H(\gamma_n,\gamma_\infty)\to0$, where $d_H$ denotes Hausdorff distance. Let $z\in B(x;R)$. For each $n\in\msf{N}$, we have $B(z;\delta_n)\cap \gamma_n\not=\emptyset$. Since $\delta_n\to0$, it follows that $z\in \gamma_\infty$. Therefore, $B(x;R)\subset \gamma_\infty$. Since $\alpha_\infty:[0,C]\to X$ is $1$-Lipschitz, we conclude that $\mathcal{H}^1(B(x;R))\leq \mathcal{H}^1(\gamma_\infty)\leq C<+\infty$. This contradiction implies the lemma.
\end{proof}

\begin{proposition}\label{P:aqcx}
Suppose $X$ is unbounded, proper, $L$-bi-Lipschitz homogeneous, and $M$-quasi-invertible. If $X$ contains a non-degenerate rectifiable curve, then $X$ is either bi-Lipschitz homeomorphic to $\msf R$ or $X$ is annularly quasiconvex.
\end{proposition}

\begin{proof}
The proof will proceed by a bootstrapping argument. In Part 1, we prove that $X$ is rectifiably connected. In Part 2, we prove that $X$ is quasiconvex. Finally, in Part 3, we prove the conclusion of the proposition.

\medskip
\textit{Part 1.} 
For every $x\in X$, let $E(x)$  be the set of all points in $X $ that can be joined to $x$ by a rectifiable curve in $X$.
Fix any $x\in X$. By assumption, there is a rectifiable curve in $X$ joining two distinct points; by uniform bi-Lipschitz homogeneity, we may assume that such a curve, denoted by  $\gamma$,   joins   $x$ with some other point $y\neq x$.
 Since $\gamma$ is compact, there exists $R<+\infty$ such that $\gamma\subset B(x;R)$. By \rf{L:fixed_point_d}, there exists $C_1=C_1(L,M,R)$ and $\delta_1>0$ such that, for any point $u\in B(y;\delta_1)$, there exists a $C_1$-bi-Lipschitz embedding $f:B(x;R)\to X$ such that $f(x)=x$ and $f(y)=u$. In particular, the curve $f(\gamma)$ is   rectifiable  and joins $f(x)=x$ to $f(y)=u$. Consequently, the set $E(x) \setminus \{x\} $ is open. By symmetry, the point $x$ is in the interior of $E(y)$. In other words, starting from $y$ we can get to an arbitrary point in some neighbourhood of $x$ by a rectifiable curve.  Concatenating  the curve $\gamma$ (and its reverse parametrization) with these curves, we conclude that $x$ is in the interior of $E(x)$. 
 That is,  there exists $\delta_2>0$ such that $B(x;\delta_2) \subset E(x)$.
Since $X$ is unbounded, \rf{L:quasi_dilation} implies the existence of $(C_3,R_n)$-quasi-dilations $f_n:X\to X$ fixing $x$. Here $R_n\to+\infty$ and $C_3=C_3(L,M)$. 
 Given any $z\in X$, there exists $n\in \msf{N}$ such that $\delta_2R_n/C_3>d(x,z)$, and thus  $z\in f_n(B(x;\delta_2))\subset  E(x) $. Since $z\in X$ was arbitrary, we conclude that $E(x)=X$, and $X$ is rectifiably connected.

\medskip
\textit{Part 2.} Since $X$ is rectifiably connected, it is connected. Since $X$ is connected and unbounded, there exist points $x,y\in X$ such that $d(x,y)=1$. Let $\gamma_y$ denote a rectifiable curve joining two such points $x$ and $y$. Choose $R_y>0$ large enough to ensure that $\gamma_y\subset B(x;R_y)$. By \rf{L:fixed_point_d}, there exists $C_1=C(L,M,y)<\infty$ and $\delta_y>0$ such that, for any $u\in B(y;\delta_y)$, there exists a $C_1$-bi-Lipschitz embedding $f:B(x;R_y)\to X$ such that $f(x)=x$ and $f(y)=u$. Therefore, each point in $B(y;\delta_y)$ is connected to $x$ by a rectifiable curve whose length is at most $C_1\textrm{Length}(\gamma_y)$. 

By \rf{L:quasi_dilation}, the metric space $X$ is $C_2$-uniformly quasi-dilation invariant, for $C_2=C_2(L,M)$. Since $X$ is proper, the closure of the annulus $A:=A(x;C_2^{-1},C_2)$ is compact. Therefore, the collection of open balls $\{B(y;\delta_y)\,|\,y\in \overline{A}\}$ contains a finite sub-collection whose union covers $\overline{A}$. It follows that there exists $1\leq C_3<\infty$ such that, for every $v\in A$, there exists a rectifiable curve $\gamma_v$ joining $x$ to $v$ satisfying $\textrm{Length}(\gamma_v)\leq C_3d(x,v)$. 

Fix $w\in X\setminus\{x\}$. Since $X$ is $C_2$-uniformly quasi-dilation invariant, there exists a $(C_2,1/d(x,w))$-quasi-dilation $f:X\to X$ fixing $x$ such that $C_2^{-1}\leq d(f(w),x)\leq C_2$. By the previous paragraph, there exists a rectifiable curve $\gamma_{f(w)}$ joining $x$ to $f(w)$ such that $\textrm{Length}(\gamma_{f(w)})\leq C_3d(x,f(w))$. Then $\gamma_w=f^{-1}(\gamma_{f(w)})$ is a rectifiable curve joining $x$ to $w$ such that 
\[\textrm{Length}(\gamma_w)\leq C_2d(x,w)\textrm{Length}(\gamma_{f(w)})\leq C_2C_3d(x,w)d(x,f(w))\leq C_2^2C_3d(x,w).\]
Therefore, for any $w\in X\setminus\{x\}$, there exists a rectifiable curve $\gamma_w$ joining $x$ to $w$ such that $\textrm{Length}(\gamma_w)\leq C_4d(x,w)$, for $C_4=C_1^2C_3$. Since $X$ is $L$-bi-Lipschitz homogeneous, it follows that $X$ is $C_5$-quasiconvex, with $C_5=L^2C_4$.  

\medskip
\textit{Part 3.} Fix $p\in X$. Assume that, for any $r>0$ and $z\in X$, we have $\mathcal{H}^1(B(z;r))<+\infty$. Since $\sph_p(X)$ is compact, $X$ is locally uniformly bi-Lipschitz equivalent to $\sph_p(X)\setminus\{\infty\}$, and by \rf{P:invertible_char} we know that $\sph_p(X)$ is uniformly bi-Lipschitz homogeneous, it follows that $\mathcal{H}^1(\sph_p(X))<+\infty$. Since $\sph_p(X)$ is a connected metric space, we conclude that $1\leq\dim_T(\sph_p(X))\leq\dim_H(\sph_p(X))\leq 1$. Here $\dim_H$ denotes Hausdorff dimension and $\dim_T$ denotes topological dimension. It follows that the Hausdorff and topological dimensions of $\sph_p(X)$ agree. By \cite[Theorem 1.3]{Freeman-invertible}, we conclude that $X$ is bi-Lipschitz homeomorphic to $\msf{R}$.

Hereafter, we assume that, for any $r>0$ and $z\in X$, we have $\mathcal{H}^1(B(z;r))=+\infty$. Choose $x,y\in A(p;1/(LC_1),2LC_1)$. Here $C_1=C_1(L,M)$ is the quasi-dilation invariance constant for $X$ provided by \rf{L:quasi_dilation}. By Part 2 of the current proof, there exists $C_2<+\infty$ such that $X$ is $C_2$-quasiconvex. Let $\gamma$ denote a rectifiable curve joining $x$ to $y$ in $X$ satisfying $\textrm{Length}(\gamma)\leq C_2d(x,y)$. If $d(x,y)<1/(4LC_1C_2)$, then $\gamma$ also satisfies 
\begin{equation}\label{E:close_points}
\gamma\subset A(p;1/(2LC_1),3LC_1). 
\end{equation}

We assume in the sequel that $d(x,y)\geq 1/(4LC_1C_2)$. For any $a\in \gamma$, we have
\[d(p,a)\leq d(p,x)+d(x,a)\leq 2LC_1+\textrm{Length}(\gamma)\leq 2LC_1+C_2d(x,y)< 8LC_1(1+C_2)=:C_3.\]
Therefore, 
\begin{equation}\label{E:init_outside}
\gamma\subset B(p;C_3).
\end{equation}
By \rf{L:q_translate}, there exists $0<\delta_1<1/(4C_2)$ such that, for any $q\in B(p;\delta_1)$, there exists an $L^2$-bi-Lipschitz homeomorphism $h_q:X\to X$ such that $h_q(q)=p$ and $\sup_{z\in \overline{B(p;C_3)}}d(h_q(z),z)<1/(4LC_1C_2^2)$. Since $\textrm{Length}(\gamma)\leq 4LC_1C_2$, by \rf{L:porous}, there exists $0<\delta_2<L/(4C_2)$ and a point $q\in B(p;\delta_1)$ such that $B(q;\delta_2)\cap \gamma=\emptyset$. Write $\gamma_1:=h_q(\gamma)$. By \rf{E:init_outside}, we have
\begin{equation*}\label{E:final_outside}
\gamma_1\subset B(p;C_3+1/(4LC_1C_2^2))\subset B(p;2C_3).
\end{equation*}
Moreover, we note that
\begin{equation}\label{E:annulus_main}
\gamma_1\subset A(p;\delta_2/L^2,2C_3).
\end{equation}
Let $\gamma_2$ and $\gamma_3$ denote rectifiable curves joining $x$ to $h_q(x)$ and $y$ to $h_q(y)$, respectively, such that $\textrm{Length}(\gamma_2)\leq C_2d(x,h(x))$ and $\textrm{Length}(\gamma_3)\leq C_2d(y,h(y))$. Let $a$ denote any point in $\gamma_2$. Then we observe that
\[d(p,a)<d(p,x)+d(x,a)<2LC_1+\textrm{Length}(\gamma_2)\leq 2LC_1+1/4<3LC_1.\]
On the other hand,
\[d(p,a)\geq d(p,x)-d(x,a)> 1/(LC_1)-\textrm{Length}(\gamma_2)\geq 1/(LC_1)-1/(4LC_1)>1/(3LC_1)\]
The same argument can be applied to points in $\gamma_3$, and thus, for $i=2,3$, we have
\begin{equation}\label{E:annulus_ends}
\gamma_i\subset A(p;1/(3LC_1),3LC_1).
\end{equation}
Concatenating the curves $\gamma_2$, $\gamma_1$, and $\gamma_3$, we obtain a rectifiable curve $\gamma_4$ joining $x$ to $y$ such that 
\[\textrm{Length}(\gamma_4)\leq 1/(2LC_1C_2)+L^2C_2d(x,y)\leq (2+L^2C_2)d(x,y)=C_4d(x,y).\]
Here $C_4=2+L^2C_2$, and we use the assumption that $d(x,y)\geq 1/(4LC_1C_2)$. Furthermore, by \rf{E:annulus_main} and \rf{E:annulus_ends} we observe that, for $C_5=\max\{2C_3,3LC_1,L^2/\delta_2\}$, we have
\begin{equation}\label{E:annular_curve}
\gamma_4\subset A(p;1/C_5,C_5).
\end{equation}

We summarize our work in Part 3 thus far in order to clarify the roles of various constants. Again writing $C_1$ to denote the quasi-dilation invariance constant for $X$ provided by \rf{L:quasi_dilation}, we have shown that there exists a constant $C_0<+\infty$ such that, for any $x,y\in \overline{A}$, there exists a constant $C_{x,y}\in(4LC_1,+\infty)$, and a $C_0$-quasi-convex curve $\gamma_{x,y}\subset A(p;1/C_{x,y},C_{x,y})$ joining $x$ to $y$. Here we write $\overline{A}$ to denote the closure of $A=A(p;1/(LC_1),2LC_1)$.

We note that $\overline{A}\times \overline{A}\subset X\times X$ is compact. Furthermore, we note that, for any $x,y\in \overline A$, the product $B(x;c)\times B(y;c)$ is open in $X\times X$. Here $c>0$ is such that any points of $\overline A$ within distance $2c$ of one another can be joined by a $C_0$-quasi-convex curve contained in $A(p;1/(4LC_1),4LC_1)$; see the discussion immediately preceding \rf{E:close_points}. It follows that any pair $(u,v)\in (B(x;c)\cap\overline{A})\times (B(y;c)\cap\overline{A})$ can be joined by a $(3C_0)$-quasi-convex curve $\gamma_{u,v}$ such that $\gamma_{u,v}\subset A(p;1/C_{x,y},C_{x,y})$. Since $\overline{A}\times \overline{A}$ is compact, there exists a finite collection of open sets of the form $(B(x;c)\cap \overline{A})\times(B(y;c)\cap\overline{A})$ whose union covers $\overline{A}\times\overline{A}$. It follows that there exists a constant $K<+\infty$ such that, for any points $x,y\in A=A(p;1/{LC_1},2LC_1)$, there exists a $K$-quasi-convex curve $\gamma$ joining $x$ to $y$ such that $\gamma\subset A(p;1/K,K)$.

To conclude Part 3 and the proof as a whole, choose any $z\in X$, $r>0$, and $a,b\in A(z;r,2r)$. Let $f:X\to X$ denote a $(C_1,1/r)$-quasi-dilation fixing $z$. Let $g:X\to X$ denote an $L$-bi-Lipschitz homeomorphism such that $g(z)=p$. Then $g\circ f(A(z;r,2r))\subset A(p;1/(LC_1),2LC_1)$. By the preceding paragraph, there exists a $K$-quasi-convex curve $\gamma$ joining $g(f(a))$ to $g(f(b))$ such that $\gamma\subset A(p;1/K,K)$. Writing $\gamma':=f^{-1}(g^{-1}(\gamma))$, we observe that $\gamma'$ is a $(LC_1K)^2$-quasi-convex curve joining $a$ to $b$ such that 
\[\gamma'\subset A\left(z;\frac{r}{(LC_1K)^2},2(LC_1K)^2r\right).\]
Therefore, $X$ is $(LC_1K)^2$-annularly quasi-convex.
\end{proof}

We conclude this subsection with the following result connecting Laakso's line-fitting property with the existence of rectifiable curves. Following \cite{TW-snowflakes}, we say that a space is \textit{line-fitting} provided that, for each $n\in\msf{N}$, there is a distance $d_n$ on the disjoint union $X\sqcup[0,1]$ such that $d_n$ is the standard Euclidean distance on $[0,1]$, $d_n$ is a constant multiple of $d$ on $X$, and $[0,1]$ is contained in the $1/n$-neighborhood of $X$. 

\begin{lemma}\label{L:line_fitting}
Suppose $X$ is uniformly $L$-bi-Lipschitz homogeneous and admits an $M$-quasi-inversion. If $X$ is line-fitting, then $X$ contains a non-degenerate rectifiable curve. 
\end{lemma}

\begin{proof}
For each $n\in\msf N$, let $d_n$ denote the distance on $X\sqcup [0,1]$ given by the assumption that $X$ is line-fitting. For each $n\in\msf N$, let $\{x^{(n)}_k\}_{k=0}^{2^n}$ denote a sequence of points in $X$ such that, for each $0\leq k\leq 2^n$, we have $d_n(x^{(n)}_k,k/2^n)<1/2^n$. Here $k/2^n\in[0,1]\subset X\sqcup[0,1]$. For each $n\in\msf N$, let $c_n>0$ denote the constant such that $d_n=c_nd$ on $X$, and let $f_n:X\to X$ denote a $(K,c_n)$-quasi-dilation at $x^{(n)}_0$, where $K$ is independent of $n$ (here we use \rfs{L:quasi_dilation} and~\ref{L:uniformly_perfect}). Define $\{y^{(n)}_k\}_{k=0}^{2^n}:=\{f_n(x^{(n)}_k)\}_{k=0}^{2^n}$, and fix a point $p\in X$. For each $n\in\msf N$, there exists an $L$-bi-Lipschitz homeomorphism $g_n:X\to X$ such that $g_n(p)=x^{(n)}_0=y^{(n)}_0$. Define $\{z^{(n)}_k\}_{k=0}^{2^n}:=\{g_n^{-1}(y^{(n)}_k)\}_{k=0}^{2^n}$, and note that, for each $n\in\msf N$, we have $p=z^{(n)}_0$. Given any $n\in\msf N$ and $0\leq k\leq 2^n$, we observe that
\begin{align*}
d(p,z^{(n)}_k)&=d(g_n^{-1}(f_n(x_0^{(n)})),g_n^{-1}(f_n(x_k^{(n)})))\leq LKc_nd(x_0^{(n)},x_k^{(n)})\\
&=LKd_n(x_0^{(n)},x_k^{(n)})\leq LK(2^{-n}+k2^{-n}+2^{-n})\leq 2LK.
\end{align*}
Since $X$ is assumed to be proper, and the sequences $\{z^{(n)}_k\}_{k=0}^{2^n}$ are all within a bounded distance of $p$, by Blaschke's Theorem there exists a compact set $E\subset X$ to which the sets $\{z^{(n)}_k\}_{n\in\msf{N}}$ converge with respect to Hausdorff distance (up to a subsequence). 

We claim that $E$ is a non-degenerate rectifiable curve. We first note that the points $z^{(n)}_{2^n}$ converge (up to a subsequence) to a point $z\in E$ such that $z\not=x$. Indeed, for every $n\in\msf N$, we have 
\begin{align*}
d(p,z^{(n)}_{2^n})&=d(g_n^{-1}(f_n(x^{(n)}_0)),g_n^{-1}(f_n(x^{(n)}_{2^n})))\geq (LK)^{-1}c_nd(x^{(n)}_0,x^{(n)}_{2^n})\\
&=(LK)^{-1}d_n(x_0^{(n)},x^{(n)}_{2^n})>(LK)^{-1}(1-2^{-n+1}).
\end{align*}
Therefore, for every $n\geq2$, we have $d(p,z^{(n)}_{2^n})\geq 1/(2LK)>0$, and so $z\not=x$. This demonstrates that $E$ is non-degenerate. 

To see that $E$ is a curve, for each $n\in\msf N$, define the map $h_n:\{k/2^n\}_{k=0}^{2^n}\to X$ as $h_n(k/2^n)=z^{(n)}_k$. We note that this sequence of maps $(h_n)_{n=1}^{+\infty}$ is both \textit{locally uniformly bounded} and \textit{locally equicontinuous} in the sense of \cite[Section 5.2]{Herron-Gromov}. Therefore, via \cite[Proposition 5.1]{Herron-Gromov} we conclude that the sequence $(h_n)_{n=1}^{+\infty}$ converges locally uniformly (in the sense of \cite[Section 5.2]{Herron-Gromov}) to a continuous map $h:[0,1]\to X$. It is straightforward to verify that $h([0,1])=E$. 

Finally, to see that $E$ is rectifiable, we note that each map $h_n$ is $(3LK)$-Lipschitz. By the remarks immediately following the proof of \cite[Proposition 5.1]{Herron-Gromov}, we conclude that $h$ is also Lipschitz. Therefore, $E$ is rectifiable. 
\end{proof}

With the lemmas estrablished we are ready to finish the proof of  Theorem~\ref{T:main:connected}.
\begin{proof}[{Proof of Theorem~\ref{T:main:connected}}]
We begin by confirming (1). Using the argument from Part 1 of the proof of \rf{P:aqcx}, the existence of a non-degenerate arc in $X$ allows us to conclude that $X$ is path connected. In particular, $X$ is connected.  

Since $X$ is locally compact, given any point $x\in X$, there exists an open neighborhood $U$ of $X$ contained in a compact subset $E\subset X$. In particular, $\overline{U}$ is compact. Given any $r>0$, via \rfs{L:quasi_dilation} and \ref{L:uniformly_perfect}, there exists a $K$-quasi-dilation $f:X\to X$ at $x$ of factor $s>0$ such that $f(B(x;r))\subset U$. Therefore, $\overline{f(B(x;r))}$ is compact. Since $f$ is a homeomorphism, $\overline{B(x;r)}$ is compact. 
Since $r>0$ and $x\in X$ were arbitrary, we have demonstrated the properness of $X$. 

To see that $X$ is Ahlfors $Q$-regular, fix $x\in X$. Since $X$ is proper, the ball $B(x;1)$ can be covered by finitely many balls of radius $1/2$. Using the uniform bi-Lipschitz homogeneity and quasi-dilation invariance of $X$, one can then verify that $X$ is doubling. Via \rf{P:invertible_char}, we now satisfy the assumptions of \cite[Theorem 1.1]{Freeman-iiblh}, and so $X$ is Ahlfors $Q$-regular, for some $Q\geq1$. 

Via the argument from Part 2 of the proof of \rf{P:aqcx}, the path connectedness of $X$ implies that $X$ is $\LLC_1$ with respect to curves. That is, there exists a constant $C\geq1$ such that, given $x,y\in X$, there exists a curve $\gamma$ joining $x$ and $y$ such that $\diam(\gamma)\leq C\,d(x,y)$. In particular, $X$ is locally path connected, and thus locally connected. 

\medskip
We now prove (2). By \rf{L:cut}, the cut point of $X$, given by the assumption, is a strong cut point. 
Via bi-Lipschitz homogeneity, every point of $X$ is a strong cut point. Since $X$ is proper, it is separable. Therefore, $X$ is a separable, locally connected, locally compact, and Hausdorff space in which each point is a strong cut point. By Ward's theorem, see \cite{FK70}, there exists a homeomorphism $\phi:\msf{R}\to X$. 

We construct a useful parameterization $g:\msf{R}\to X$ following the method of \cite[Lemma 2.1]{GH-ar}. For $t\geq 0$, define
\[m(t):=\begin{cases} -\mathcal{H}^Q(\phi([t,0])) & \text{ if } t\leq 0 \\
				\mathcal{H}^Q(\phi([0,t])) & \text{ if } t\geq 0
				\end{cases}\]
Here we recall that $X$ is Ahlfors $Q$-regular. Due to basic properties of the measure $\mathcal{H}^Q$, the map $m$ is a self-homeomorphism of $\msf{R}$. Then, for any interval $I\subset \msf{R}$, it is straightforward to verify that the homeomorphism $g(t):=\phi(m^{-1}(t))$ from $\msf{R}$ to $X$ satisfies $\mathcal{H}^Q(g(I))=\mathcal{H}^1(I)$. 

Given any $x,y\in X$, write $a=g^{-1}(x)$ and $b=g^{-1}(y)$. Suppose $a<b$. Then we observe that
\[|b-a|=\mathcal{H}^1([a,b])=\mathcal{H}^Q(g([a,b])).\]
We claim that $\mathcal{H}^Q(g([a,b]))\eqx d(x,y)^Q$, up to a constant independent of the points $x$ and $y$. Indeed, via \cite[Theorem E]{HM-blh} the space $X$ satisfies a \textit{generalized chordarc condition}. Since $X$ is Ahlfors $Q$-regular, this generalized chordarc condition is in fact a $Q$-dimensional chordarc condition in the sense of \cite[Section 4]{GH-ar}. This $Q$-dimensional chordarc condition is precisely the desired comparability. Therefore, for any points $x,y\in X$, we have
\[d(x,y)\eqx |g^{-1}(x)-g^{-1}(y)|^{1/Q}.\]
In particular, $X$ is bi-Lipschitz homeomorphic to the snowflake $(\msf{R},|\cdot|^{1/Q})$, where $1/Q\in(0,1]$. 

\medskip
Next, we prove (3). Suppose $X$ contains no cut points. We have already demonstrated in the proof of (1) that $X$ is $\LLC_1$. To see that $X$ is also $\LLC_2$, and thus linearly locally connected, we cite \cite[Theorem 1.2]{Freeman-iiblh} and \rf{P:invertible_char}. If $X$ contains a non-degenerate rectifiable curve, then \rf{P:aqcx} implies that $X$ is either bi-Lipschitz homeomorphic to $\msf R$ or annularly quasi-convex. Since $X$ contains no cut point, $X$ is annularly quasi-convex. If $X$ does not contain a non-degenerate rectifiable curve, then \rf{L:line_fitting} enables us to conclude that $X$ is not line-fitting. Therefore, by \cite[Theorem 7.2]{TW-snowflakes}, the  space $X$ is bi-Lipschitz homeomorphic to a (non-trivial) snowflake.
\end{proof}

\section{Disconnected Spaces}\label{S:disconnected}

In this final section we prove our results pertaining to disconnected metric spaces. Before proceeding with these proofs we introduce additional of terminology.

Following \cite[Definition 15.1]{DS-fractals}, given $\alpha\in(0,1]$, we say that a metric space $X$ is \textit{$\alpha$-uniformly disconnected} if for every $x\in X$ and $r>0$ there exists a closed subset $A\subset X$ such that $B(x;\alpha r)\subset A\subset B(x;r)$,  and $\dist(A,X\setminus A)\geq \alpha r$. For example, an ultrametric space is $1$-uniformly disconnected (see \cite[pg.\,161]{DS-fractals}). We remark that, for $\alpha\in(0,1)$, this definition is equivalent to the definition of uniform disconnectedness based on the non-existence of so-called \textit{$\alpha$-chains} (see \cite{Heer-QM}, \cite{MT-cfl}). A sequence of points $\{x_0,\dots,x_n\}\subset X$ is an $\alpha$-chain if, for $k=1,\dots,n$, we have $d(x_{k-1},x_k)\leq \alpha\, d(x_0,x_n)$. We say that a space $X$ is \textit{uniformly disconnected with respect to $\alpha$-chains} if there exist no $\alpha$-chains in $X$. 

\begin{lemma}\label{L:uniformly_disconnected}
Suppose $X$ is an unbounded, locally compact, uniformly bi-Lipschitz homogeneous, and quasi-invertible metric space. If $X$ is disconnected, then $X$ is uniformly disconnected.
\end{lemma}

\begin{proof}
Our first goal is to show that $X$ is totally disconnected, and then we will proceed to show that $X$ is uniformly disconnected. For use later in the proof, we being by observing that $X$ satisfies the assumptions of \rfs{L:quasi_dilation} and~\ref{L:uniformly_perfect}, and so $X$ is uniformly quasi-dilation invariant. Using this property along with uniform bi-Lipschitz homogeneity it is not hard to confirm that $X$ is proper. 

To see that $X$ is totally disconnected, we assume that it is not and proceed by way of contradiction through the following three steps: We first show that each connected component of $X$ is unbounded. Next, we show that each connected component of $X$ is a cut point space in the sense of \cite{HB99}. Finally, in order to obtain the desired contradiction, we show that each connected component of $X$ is not a cut point space.  

\medskip
Step 1: To see that each connected component of $X$ is unbounded, let $X(p)$ denote the connected component of $X$ containing a point $p\in X$. Since we are assuming that $X$ is not totally disconnected, there exists a connected component of $X$ consisting of more than one point. Since $X$ is bi-Lipschitz homogeneous, every connected component of $X$ consists of more than one point. In particular, the cardinality of $X(p)$ is greater than one. It follows from \rfs{L:quasi_dilation} and~\ref{L:uniformly_perfect} that $X$ is uniformly quasi-dilation invariant. Therefore, $X(p)$ is unbounded. By uniform bi-Lipschitz homogeneity, every connected component is unbounded. 

\medskip
Step 2: To see $X(p)$ is a cut point space (and thus every connected component is a cut point space), we refer to our assumption that $X$ is not connected to ensure the existence of a connected component $E$ of $X$ such that $E\not=X(p)$. Since $E$ is unbounded, $p$ is an accumulation point of $\sigma_p(E)\subset X$, and so $p\in\overline{\sigma_p(E)}\subset X$. Since $\overline{\sigma_p(E)}$ is connected in $X$ and shares a point with the connected set $X(p)$, the union $\overline{\sigma_p(E)}\cup X(p)$ is also connected in $X$. Since $X(p)$ is a maximal connected subset of $X$, we have $\overline{\sigma_p(E)}\subset X(p)$ and thus $\sigma_p(E)\subset X(p)$. We also note that $p$ is not an accumulation point of the closed set $E$, and thus $\sigma_p(E)$ is bounded in $X$. 

We claim that $p$ is a cut point of the connected set $X(p)$. In other words, $X(p)\setminus\{p\}$ is disconnected. By way of contradiction, we assume that $X(p)\setminus\{p\}$ is connected. First, it is straightforward to verify that, because $p\not\in E$ (the connected component of $X$ described above), the set $E$ is also a connected component of the space $X_p=X\setminus\{p\}$. This implies that $\sigma_p(E)$ is also a connected component of $X_p$. Next, we note that since $\sigma_p(E)\subset X(p)\setminus\{p\}$ and $X(p)\setminus\{p\}$ is assumed to be connected, we have $\sigma_p(E)=X(p)\setminus\{p\}$ (else $\sigma_p(E)$ is not maximal). Since $\sigma_p(E)$ is bounded, while $X(p)\setminus\{p\}$ is unbounded, we reach a contradiction. This contradiction confirms that $p$ is a cut point of $X(p)$. 

Since bi-Lipschitz self-homeomorphisms permute connected components of $X$, the assumptions on $X$ imply that $X(p)$ is itself $L$-bi-Lipschitz homogeneous. By way of this homogeneity, we conclude that every point of $X(p)$ is a cut point for $X(p)$. In other words, $X(p)$ is a cut-point space. Indeed, every connected component of $X$ is a cut-point space. 

\medskip
Step 3: We now show that $X(p)$ is \textit{not} a cut-point space. Given the connected component $E\not=X(p)$ as above, it is easy to see that $K:=\sigma_p(E)\cup\{p\}$ is closed and bounded in $X(p)$.  Since $X(p)$ is a proper metric space, this implies that $K$ is compact. Furthermore, since $K=\overline{\sigma_p(E)}$, and $\sigma_p(E)$ is connected, we conclude that $K$ is also connected. Since $K$ contains more than one point, by \cite[Theorem 3.9]{HB99}, the set $K$ contains at least two points that are not cut points of $K$. This implies that some point $x\in \sigma_p(E)$ is not a cut point for $K$. Since $X(p)$ is a cut-point space, let $U_1$ and $U_2$ denote disjoint open sets in $X$ such that $X(p)\setminus\{x\}\subset U_1\cup U_2$. Without loss of generality, $p\in U_1$, and thus $K\cap U_1\not=\emptyset$. If $K\cap U_2\not=\emptyset$, then $K\setminus\{x\}$ is separated by $U_1\cap K$ and $U_2\cap K$, which contradicts the fact that $x$ is not a cut point for $K$. Therefore, $K\cap U_2=\emptyset$, and so $\sigma_p(E)\setminus\{x\}\subset U_1$. 

Let $E'$ denote any connected component of $X(p)\setminus\{p\}$ such that $E'\not=\sigma_p(E)$. Note that such a component must exist due to the fact that $\sigma_p(E)$ is bounded while $X(p)\setminus\{p\}$ is unbounded. Since $U_1$ is open in $X$, $p\in U_1$, and $p$ is an accumulation point of $E'$, it follows that $U_1\cap E'\not=\emptyset$. Since $E'$ is connected and $x\not\in E'$, we must have $E'\cap U_2=\emptyset$. Otherwise, $U_1\cap E'$ and $U_2\cap E'$ would form a separation of $E'$. This argument indicates that every connected component of $X(p)\setminus\{p\}$ other than $\sigma_p(E)$ is contained in $U_1$. 

The previous two paragraphs imply that $X(p)\setminus\{x,p\}\subset U_1$. Since $p\in U_1$, we conclude that $X(p)\setminus\{x\}\subset U_1$. This implies that $U_2=\emptyset$, and it follows that $X(p)\setminus\{x\}$ is connected. Therefore, $X(p)$ is not a cut-point space.

\medskip
Combining the conclusions of Steps 2 and 3 above, we reach the desired contradiction to our assumption that $X$ is not totally disconnected. Therefore, $X$ is totally disconnected. 

\medskip
Having demonstrated that $X$ is totally disconnected, we finish the proof by demonstrating that $X$ is uniformly disconnected. By way of contradiction, suppose $\theta_k\to0$ is a sequence of positive numbers such that, for each $k\in\msf{N}$, there exists a $\theta_k$-chain $(x_i^{(k)})_{i=0}^{n_k}$ in $X$. By uniform bi-Lipschitz homogeneity (and a quantitatively controlled change the numbers $\theta_k$), we may assume that, for each $k\in\msf{N}$, we have $x_0^{(k)}=p$. Furthermore, \rfs{L:quasi_dilation} and~\ref{L:uniformly_perfect} yield a constant $M\geq1$ such that, for each $k\in\msf{N}$, we have $E_k:=\{x_i^{(k)}\}_{i=0}^{n_k}\subset B(p;M)$. Furthermore, we may assume there exists $j_k\in\{1,\dots,n_k\}$ such that $M^{-1}\leq d(p,x_{j_k}^{(k)})\leq M$. Again using the properness of $X$, we may assume, up to a subsequence, that the sets $E_k$ converge to a non-degenerate compact set $E\subset X$ with respect to Hausdorff distance. 

We claim that $E$ is connected. Indeed, suppose (by way of contradiction) $E'$ and $E''$ are distinct connected components of $E$. Both $E'$ and $E''$ are closed (in $E$) and bounded. Since $E$ is compact, each of $E'$ and $E''$ is compact. Let $\varepsilon>0$ be such that $\dist(E',E'')=3\varepsilon$. Write $U_1$ and $U_2$ to denote $\varepsilon$-neighborhoods of $E'$ and $E''$, respectively. Since $E$ is compact and $U_1\cup U_2$ is open, there exists $N_1\in\msf{N}$ such that, for all $k\geq N_1$, we have $E_k\subset U_1\cup U_2$. Furthermore, there exists $N_2$ such that, for all $k\geq N_2$, we have $\theta_k d(p,x_{n_k}^{(k)})\leq M\theta_k<\varepsilon$. The definition of a $\theta_k$-chain implies that $\varepsilon\leq\dist(U_1,U_2)<\varepsilon$. This contradiction proves that $E$ is connected. 

We have shown that if $X$ is not uniformly disconnected, then $X$ contains a non-degenerate continuum. This contradicts the fact that $X$ is totally disconnected. We conclude that $X$ is uniformly disconnected. 
\end{proof}

\subsection{Examples of disconnected spaces}\label{Sec:ex}
\begin{example}\label{X:basic}
We present the basic example of a disconnected, isometrically homogeneous, and invertible metric space. In contrast to the brief description provided in Section~\ref{s:results}, we here provide a more detailed construction.
 We fix $N\in \msf{N}$ with $N\geq 2$ and $s>1$.
Define the metric space $(  C_{N}, \rho_s )$ by considering the set 
\[C_{N}:= \{  \xi= (\xi_i)_{i\in\msf{Z}}\,|\,\forall\,i, \xi_i \in \{1,\ldots, N\} \text{ and } \exists  \,m\in \msf{Z} \text{ such that } \forall\,i\leq m, \xi_i=1\}\]
equipped with the distance
\begin{equation}\label{def:rho}
\rho_s (\xi,\zeta):=s^{-m(\xi,\zeta)},\qquad \text{ where } \quad m (\xi,\zeta):=\sup\{m\in\msf{Z}\,|\,\forall i\leq m, \xi_i=\zeta_i\}.
\end{equation}

The metric space $(  \hat C_{N}, \rho_s )$, which represents a sphericalization of the metric space $(C_{N}, \rho_s )$, is defined by  the set 
$$\hat C_{N}:= \{  \xi= (\xi_i)_{i\in\msf{N}}\,|\,\xi_1 =1 , \xi_2 \in \{1,\ldots, N+1\} \text{, and } \forall\,i\geq3, \xi_i \in \{1,\ldots, N\}\}$$
and $\rho_s$ is defined by \eqref{def:rho}. Note that for points $\xi,\zeta\in \hat{C}_N$ we have $m(\xi,\zeta)\geq1$. To see that $\hat{C}_N$ is bi-Lipschitz homeomorphic to a sphericalization of $C_N$, we argue as follows. Write $\mathbf{1}\in C_N$ to denote the constant sequence whose every entry is equal to $1$. Given $\xi\in C_N$, define $\hat{\xi}\in \hat{C}_N$ according to the following cases. If $m(\xi,\mathbf{1})\geq 0$, then $\hat{\xi}_i=\xi_{i-1}$ for all $i\geq1$. If $m(\xi,\mathbf{1})=-1$,
\[\hat{\xi}_i=\begin{cases} 1			& \text{ if } i=1\\
							N+1 		& \text{ if } i=2\\
							\xi_{i-3}	& \text{ if } i\geq 3.
							\end{cases}\]
If $m(\xi,\mathbf{1})\leq-2$, 
\[\hat{\xi}_i=\begin{cases} 1			& \text{ if } i=1\\
							N+1 		& \text{ if } i=2\\
							1			& \text{ if } 3\leq i\leq 1-m(\xi,\mathbf{1})\\
							\xi_{i+2m(\xi,\mathbf{1})-1}& \text{ if } i\geq 2-m(\xi,\mathbf{1}).
							\end{cases}\]
This establishes a bijection between points in $\Sph_\mathbf{1}(C_N)$ and $\hat{C}_N$. Here we note that the point at infinity is identified with the point $(1,N+1,1,\dots)\in \hat{C}_N$, where the ellipsis indicates a constant sequence of terms equal to $1$. Via a tedious but straightforward case analysis, one can verify that, for any $\xi,\zeta\in C_N$, 
\[\frac{\rho_s(\xi,\zeta)}{(1+\rho_s(\xi,\mathbf{1}))(1+\rho_s(\zeta,\mathbf{1}))}\eqx \rho_s(\hat{\xi},\hat{\zeta}).\]
Thus we see that $\hat{C}_N$ is indeed bi-Lipschitz homeomorphic to the sphericalized space $\Sph_\mathbf{1}(C_N)$. We note that when $N=2$ and $s=2$, $\hat{C}_N$ is the symbolic Cantor set studied in \cite[Section 2.3]{DS-fractals}.

The function $\rho_s$ is an ultrametric both on $ C_{N}$ and in $\hat C_{N}$. The space $(C_{N}, \rho_s )$ is proper, unbounded, two-point isometrically homogeneous, and invertible. We shall prove these properties in \rf{X:counter}, where we construct a slightly more general collection of spaces. 

\end{example}
\begin{example}\label{X:counter}
In order to illustrate the sharpness of \rf{T:sharp:nonconnected}, we provide the following generalization of the construction from \rf{X:basic}. Using the terminology of the previous example, for any $N,M\in\msf{N}$ such that $N\geq M$, we consider the subset $C_{N|M}\subset C_{N}$ defined by
$$C_{N|M}:= \{  \xi\in C_{N} \;:\; \xi_i \in \{1, \ldots, M\},\,   \forall i \text{ even}   \}.$$
Note that $C_{N|N}=C_N$. We consider $C_{N|M}$ with the metric $\rho_s$ given by \eqref{def:rho}.
The space $(C_{N|M}, \rho_s )$ is proper. Indeed, every point has a neighborhood that is topologically a Cantor set.

We claim that $(C_{N|M}, \rho_s )$ is $2$-point isometrically homogeneous. To verify this claim, we first demonstrate that $(C_{N|M},\rho_s)$ is $1$-point isometrically homogeneous. Fix $\xi,\zeta\in C_{N|M}$. For each $i\in \msf{Z}$  chose a  permutation $\iota_i$ of $\{1, \ldots, M\}$, if $i$ is even, and of $\{1, \ldots, N\}$, if $i$ is odd, such that $\iota_i(\xi_i)=\zeta_i$. We then define $f:C_{N|M}\to C_{N|M}$ such that, for any $\omega\in C_{N|M}$, we have
\begin{equation}\label{E:homogeneity}
f(\omega)=\theta\in C_{N|M} \text{ such that } \begin{cases}
											\theta_i=\omega_i & \text{ if } \xi_i=\zeta_i \\
											\theta_i=\iota_i(\omega_i) & \text{ if } \xi_i\not=\zeta_i.
										\end{cases}
										\end{equation}
We note that $f$ is an isometry of $(C_{N|M}, \rho_s )$ such that $f(\xi)=\zeta$. Therefore, $(C_{N|M}, \rho_s )$ is $1$-point isometrically homogeneous. 

In light of $1$-point isometric homogeneity, it suffices to show that any metric sphere $S(\mathbf{1};s^k)=\{\omega\in C_{N|M}\,|\,\rho_s(\mathbf{1},\omega)=s^{-k}\}$ is homogeneous with respect to isometries of $(C_{N|M},\rho_s)$ fixing $\mathbf{1}$. To see this, we  modify the construction given in \rf{E:homogeneity}. We define the map $f_\mathbf{1}$ to be the identity away from $S(\mathbf{1};s^{-k})$. Given $\xi$ and $\zeta$ in $S(\mathbf{1};s^{-k})$, we define $f_\mathbf{1}$ on $S(\mathbf{1};s^{-k})$ as in \rf{E:homogeneity} under the additional requirement that $\iota_{k+1}(1)=1$. This is additional requirement is possible because neither $\xi_{k+1}$ nor $\zeta_{k+1}$ is equal to $1$. Furthermore, this requirement ensures that $f_\mathbf{1}$ is a self-bijection of $S(\mathbf{1};s^{-k})$. It is then straightforward to see that $f_\mathbf{1}$ is an isometry of $(C_{N|M},\rho_s)$ fixing $\mathbf{1}$ and sending $\xi$ to $\omega$. It follows that $C_{N|M}$ is $2$-point isometrically homogeneous.

Next, we claim that $(C_{N|M}, \rho_s )$ is invertible. Indeed, we define an involutive inversion $\tau$ as follows. Denote by $T $   the  shift operator   $T(\xi)_i:=\xi_{i-1}$, for every $i\in \msf{Z}$. We define an involution $\tau$ as
$$ \tau:C_{N|M}\setminus\{{\bf1} \}\to C_{N|M}\setminus\{{\bf1} \}$$
$$\xi\mapsto  \tau(\xi):=T^{2m}(\xi), \quad \text{ where }m:=m(\xi,{\bf1}),$$
where $m$ is the function in \eqref{def:rho}. To see that $\tau$ is indeed an inversion, fix $\xi$ and $\zeta$ in $C_{N|M}\setminus\{{\bf1} \}$. We consider two cases.

Case 1: $m(\xi,{\bf1})=m(\zeta,{\bf1})=m$. In this case, we have 
\[m(\tau(\xi),\tau(\zeta))=m(\xi,\zeta)-2m.\]
Thus, 
\[\rho_s(\tau(\xi),\tau(\zeta))=\frac{\rho_s(\xi,\zeta)}{\rho_s(\xi,{\bf1})\rho_s(\zeta,{\bf1})}.\]

Case 2: $m_1:=m(\xi,{\bf1})<m(\zeta,{\bf1})=:m_2$. Hence we have , $\tau(\xi)=T^{2m_1}(\xi)$ and $\tau(\zeta)=T^{2m_2}(\zeta)$. Since $-m_2<-m_1$, then for any $i\leq -m_2$ we have $1=\tau(\xi)_i=\tau(\zeta)_i$. However, since $-m_2+1\leq -m_1$, we have $\tau(\zeta)_{-m_2+1}\not=1=\tau(\xi)_{-m_2+1}.$ Consequently, we have $m(\tau(\xi),\tau(\zeta))=-m_2$ and $m(\xi,\zeta)=m_1$. Hence, we observe that
\[m(\tau(\xi),\tau(\zeta))=-m_2=m_1-m_1-m_2=m(\xi,\zeta) -m(\xi,{\bf1})-m(\zeta,{\bf1}),\]
and so
\[\rho_s(\tau(\xi),\tau(\zeta))=\frac{\rho_s(\xi,\zeta)}{\rho_s(\xi,{\bf1})\rho_s(\zeta,{\bf1})}.\]

In both of the above cases we obtain the desired metric behavior for $\tau$. Furthermore, it is straightforward to verify that $\tau:C_{N|M}\setminus\{{\bf1} \}\to C_{N|M}\setminus\{{\bf1} \}$ is a homeomorphism. Therefore, $\tau$ satisfies the definition of an inversion at ${\bf1}$.

Finally, we point out that $(C_{N|M},\rho_s)$ is isometric to $(C_{N'|M'},\rho_{s'})$ if and only if $N'=N$, $M'=M$ and $s'=s$. In particular, when $N>M$ then  $(C_{N|M},\rho_s)$ is not isometric to any $(C_{N'},\rho_{s'})$, for $N'\in N$ and $s'>1$. To see this, we first observe that   the set of distances in $(C_{N|M},\rho_s)$ is equal to $\{s^k\,|\,k\in \msf{Z}\}$. Hence we  only need to consider the case $s'=s$.
Second, we observe that the metric components of the metric spheres $S(\mathbf{1};s^{-k})\subset(C_{N|M},\rho_s)$ characterize $N$ and $M$. We require a bit of terminology: A subset $E\subset X$ is a \textit{$\delta$-component} if it is a maximal subset with the property that every pair of points from $E$ can be joined with by a sequence of points in $E$ whose consecutive distances are less than $\delta$. Using this terminology, we note that, for each $\delta\in(1/s , 1)$ the number of  $\delta$-components in
$S({\bf1};1)$ is exactly $N$, while for   $\delta\in(1  , s)$  the number of  $\delta$-components in $S({\bf1};s)$ is exactly $M$. In conclusion, the values of $N$ and $M$ are metric invariants for $(C_{N|M},\rho_s)$.
\end{example}

\begin{remark}
In light of \rf{T:sharp:nonconnected:2pt} (proved in the sequel), we note that the sphericalized spaces $\Sph_p(C_{N|M})$ are not three-point M\"obius homogeneous if $N\not=M$, despite the fact that they are 2-point isometrically homogeneous and invertible. This can be seen in the fact that, via \rf{L:3pt}, the $3$-point M\"obius homogeneity of $\Sph_p(X)$ implies that $X$ admits dilations of all factors $\lambda\in\{d(x,y)\,|\,x,y\in X\}$, while, if $N\not=M$, the space $C_{N|M}$ only admits dilations of factors $\lambda^2$ for $\lambda\in \{\rho_s(x,y)\,|\,x,y\in C_{N|M}\}$ (see also \rf{P:equivalence1a}). 
\end{remark}

\subsection{Proofs of Theorems~\ref{T:sharp:nonconnected}, \ref{T:sharp:nonconnected:2pt}, and  \ref{T:main:connected:NOT}}\label{s:disconnected_proofs}

In order to present the proof of \rf{T:sharp:nonconnected} we require the following definitions. Given $\delta>0$, a sequence of points $\{x_0,\dots,x_n\}\subset X$ is a \textit{$\delta$-sequence} if, for $k=1,\dots,n$, we have $d(x_{k-1},x_k)<\delta$. A subset $E\subset X$ is \textit{$\delta$-connected} provided that any two points $x,y\in E$ can be joined by a $\delta$-sequence such that $x_0=x$ and $x_n=y$. A \textit{$\delta$-component} of $X$ is a maximal $\delta$-connected subset of $X$.

\begin{proof}[Proof of Theorem~\ref{T:sharp:nonconnected}]
Suppose that $X$ is a disconnected, unbounded, locally compact, isometrically homogeneous metric space that admits an inversion $\sigma_p$ at some point $p\in X$. By \rf{L:uniformly_disconnected}, there exists $\alpha\in(0,1]$ such that $X$ is $\alpha$-uniformly disconnected. Fix $x\in X$. Let $A'\subset X$ denote a closed set such that $B(x;\alpha)\subset A'\subset B(x;1)$ and $\dist(A',X\setminus A')\geq\alpha$. Let $A$ denote the $\alpha$-component of $X$ containing $x$. Note that $A\subset A'\subset B(x;1)$. Since $\Isom(X)$ acts transitively on $X$, the collection $\mathcal{X}_0=\{f(A)\,|\,f\in \Isom(X)\}$ covers $X$. We also claim that $\mathcal{X}_0$ consists of pairwise disjoint sets in the sense that, for $f,g\in \Isom(X)$, either $f(A)=g(A)$ or $f(A)\cap g(A)=\emptyset$. Indeed, suppose that $f,g\in \Isom(X)$ and there exists a point $z\in f(A)\cap g(A)$. By concatenating the $\alpha$-sequences between $f(x)$ and $z$ and between $z$ and $g(x)$ we obtain a $\alpha$-sequence joining $f(x)$ to $g(x)$. Therefore, given any point $w\in g(A)$, there exists a $\alpha$-sequence joining $f(x)$ to $w$. Since $w$ was an arbitrary point of $g(A)$, it follows that $f^{-1}g(A)\subset A$. By symmetry, $g^{-1}f(A)\subset A$. Therefore, $f(A)=g(A)$.  

Choose $s>1$ such that there exists $y\in X$ satisfying $d(x,y)=\sqrt{s}$. By \rf{P:equivalence1a}, there exists an $s^{-1}$-dilation $h:X\to X$ at $x$. For each $i\in\msf{Z}$, define the set of sets
\[\mathcal{X}_{i}=h^{i}(\mathcal{X}_0)=\{h^{i}\circ f(A)\,|\,f\in \Isom(X)\}.\]
Here $h^i$ denotes the $i$-fold composition of $h$ with itself. Thus, for any $E\in \mathcal{X}_i$, we have $h(E)\in\mathcal{X}_{i+1}$. We note that the same set $E$ in $\mathcal{X}_i$ may correspond to two different isometries $f,g\in \Isom(X)$, but this will not hinder our use of $\mathcal{X}_i$ in the sequal. 

For later use, we write $\mathcal{X}$ to denote the disjoint union $\sqcup_{i\in\msf{Z}} \mathcal{X}_i$. Let $N\in\msf{N}$ denote the number of distinct sets from $\mathcal{X}_{1}$ contained in $A$. Since $s^{-1}<1$, we have $N\geq 2$. Since $\Isom(X)$ permutes elements of $\mathcal{X}_0$, $N$ also represents the number of distinct sets from $\mathcal{X}_{1}$ contained in every element of $\mathcal{X}_0$. Similarly, given any $i\in\msf{Z}$, the number $N$ represents the number of distinct sets from $\mathcal{X}_i$ contained in every element of $\mathcal{X}_{i-1}$. 

We label each of the $N$ distinct sets from $\mathcal{X}_{1}$ contained in $A$ using the labels $\{1,2,\dots,N\}$.  
We do this such that $h^{1}(A)\subset A$ receives the label $1$. For each $i\in\msf{Z}$, we use isometries and dilations to  transfer this labelling to the $N$ distinct sets from $\mathcal{X}_{i}$ contained in each element of $\mathcal{X}_{i-1}$. While this labelling is certainly not uniquely determined, we emphasize that, for all $i\in \msf{Z}$, we may assume that the set $h^i(A)$ receives the label $1$. 

We can obtain a bijection between points of $X$ and certain sequences in $\mathcal{X}$ as follows. For each $i\in\msf{Z}$, we denote the collection of distinct (and thus pairwise disjoint) sets in $\mathcal{X}_i$ as $\{E_{i,k}\}_{k\in\msf{N}}$. Given any point $z\in X$, there exists a unique sequence $(E_{i,k_i})_{i\in\msf{Z}}$ such that, for each $i\in\msf{Z}$, we have $z\in E_{i,k_i}\in \mathcal{X}_i$, and $E_{i+1,k_{i+1}}\subset E_{i,k_{i}}$. Since $B(x;\alpha)\subset A$, there exists $M=M(z)\in\msf{Z}$ such that, for any $i\leq M(z)$, we have $z\in h^i(A)\in \mathcal{X}_i$. In other words, for $i\leq M(z)$, we have $E_{i,k_i}=h^i(A)$.  

Conversely, given any sequence $(E_{i,k_i})_{i\in\msf{Z}}$ consisting of elements from $\mathcal{X}$ such that, for each $i\in\msf{Z}$, we have $E_{i,k_i}\in\mathcal{X}_i$ and $E_{i+1,k_{i+1}}\subset E_{i,k_{i}}$, there exists a unique point $z\in X$ such that $\cap_{i=0}^{+\infty}E_{i,k_i}=\{z\}$ (this is because $X$ is proper, each set $E_{i,k_i}$ is closed, and $\diam(E_{i,k_i})\to0$ as $i\to+\infty$). As in the preceding paragraph, there exists $M=M(z)\in\msf{Z}$ such that, for any $i\leq M(z)$, we have $z\in h^i(A)$ and thus $E_{i,k_i}=h^i(A)$. 

Via the preceding two paragraphs, the  labelling of $\mathcal{X}$ constructed above yields a bijection between $X$ and $C_{N}$, as defined in \rf{X:basic}.   We denote this bijection by $\phi:X\to C_{N}$. 

To see that $\phi$ is bi-Lipschitz when $C_{N}$ is equipped with the distance $\rho_s$, we proceed as follows. Choose $\xi=\phi(u)$ and $\zeta=\phi(v)$ in  $C_{N}$, and write $m\in\msf{Z}$ to denote $m (\xi,\zeta)$, where $m(\xi,\zeta)$ is defined as in \eqref{def:rho}. By the construction of $\phi$, there exists $E=h^m(f(A))\in\mathcal{X}_m$ such that $u,v\in E$ but $u$ and $v$ are contained in disjoint elements of $\mathcal{X}_{m+1}$. Note that $f(x)\in E\subset B(f(x);s^{-m})$. Since $\Isom(X)$ acts transitively on $X$ and permutes elements of $\mathcal{X}_m$, we conclude that $\Isom(X)$ acts transitively on $E$. Therefore, $E\subset B(u;s^{-m})$, and so $d(u,v)< s^{-m}$. On the other hand, distinct sets from $\mathcal{X}_{m+1}$ are separated by a distance of at least $\alpha s^{-m-1}$. Therefore, $d(u,v)\geq \alpha s^{-m-1}$. It follows that
\begin{equation}\label{E:final}
d(u,v)<\rho_s(\phi(u),\phi(v))\leq \frac{s}{\alpha}d(u,v).
\end{equation}
Thus $\phi:X\to C_{N}$ is $(s/\alpha)$-bi-Lipschitz. 
\end{proof}

We shall make use of the following result in the proof of \rf{T:sharp:nonconnected:2pt}. We include the proof for the sake of completeness, noting its similarity to the proof of \rf{P:equivalence1b}. 

\begin{lemma}\label{L:3pt}
Suppose $X$ is unbounded. The space $\hat{X}$ is $3$-point M\"obius homogeneous if and only if the following two statements are true: 
\begin{enumerate}
	\item{For any two pairs of distinct points $x,y$ and $u,v$ in $X$, there exists a $\lambda$-similarity $f:X\to X$ such that $f(x)=u$, $f(y)=v$, and $\lambda=d(u,v)/d(x,y)$.}
	\item{$X$ is invertible.}
\end{enumerate}
\end{lemma}

\begin{proof}
We first assume that $\Sph_p(X)$ is $3$-point M\"obius homogeneous. Given any two pairs of distinct points $x,y$ and $u,v$ in $X$, let $f:\Sph_p{X}\to\Sph_p{X}$ denote a M\"obius map fixing $\infty$ such that $f(x)=u$ and $f(y)=v$. By \rf{R:similarity}, $f$ is a $\lambda$-similarity of $X$. Furthermore, 
\[d(u,v)=d(f(x),f(y))=\lambda\,d(x,y),\]
and so $\lambda=d(u,v)/d(x,y)$. Thus we confirm $(1)$. 

To verify $(2)$, fix any point $a\in X\setminus\{p\}$. let $f\Sph_p(X)\to\Sph_p(X)$ denote a M\"obius homeomorphism such that $f(p)=\infty$, $f(\infty)=p$, and $f(a)=a$. For any point $x\in X$, we find that
\[d(p,f(x))=r\cdot d(p,x)^{-1},\]
where $r=d(p,a)^2$. Here we follow the calculations utilized in the proof of \rf{P:equivalence1b}. Continuing these calculations, we find that, for any $x,y\in X\setminus\{p\}$, we have
\[d(f(x),f(y))=\frac{r\,d(x,y)}{d(x,p)\,d(y,p)}.\]
By the proof of $(1)$ above, the space $X$ admits an $r$-dilation $h$ at $p$ of factor $r$. Therefore, $f\circ h:X_p\to X_p$ is an inversion of $X$ at $p$.

Conversely, if $X$ satisfies $(1)$ and $(2)$, then fix a triple $a,b,\infty$ of distinct points from $\hat{X}$. Let $x,y,z$ denote a second triple of distinct points from $\hat{X}$. If $z=\infty$, then, via $(1)$, there exists a M\"obius map $f:\Sph_p(X)\to\Sph_p(X)$ fixing $\infty$ such that $f(x)=a$ and $f(y)=b$. If $z\not=\infty$, then we first map $z$ to $p$ via an isometry of $X$ and then, via $(2)$, send $p$ to $\infty$ via the inversion of $X$. Thus we are back in the case that $z=\infty$, and we confirm that $\Sph_p(X)$ is $3$-point M\"obius homogeneous.
\end{proof}

\begin{proof}[Proof of Theorem \ref{T:sharp:nonconnected:2pt}]
Via \rf{L:3pt}, we see that $X$ admits dilations of arbitrarily large factors. Therefore, since $X$ is locally compact, it is straightforward to verify that $X$ is proper. Furthermore, $X$ satisfies the assumptions of \rf{L:uniformly_disconnected}, and so $X$ is uniformly disconnected. We claim that $X$ is an ultrametric space. By way of contradiction, suppose there exist points $x,y,z$ in $X$ such that $d(x,y)>\max\{d(x,z),d(z,y)\}$. In particular, there exists $c\in (0,1)$ such that $c\,d(x,y)\geq \max\{d(x,z),d(y,z)\}$. In order to obtain the desired contradiction, we construct a non-degenerate connected subset of $X$ using a construction from the proof of \cite[Lemma 3.5]{KL16}. We write $x_0^{(1)}=x$, $x_1^{(1)}=z$, and $x_2^{(1)}=y$. Given a sequence 
\[x=x_0^{(k)},x_1^{(k)},\dots,x_{2^k}^{(k)}=y,\]
for which pairs of consecutive points are distinct, we form a new sequence
\[x=x_0^{(k+1)},x_1^{(k+1)},\dots,x_{2^{k+1}}^{(k+1)}=y\]
by defining
\[x_i^{(k+1)}=\begin{cases} x_{i/2}^{(k)} & \text{ if } i \text{ is even }\\
					f_i^{(k+1)}(z)	& \text{ if } i \text{ is odd}.
					\end{cases}\]
Here $f_i^{(k+1)}$ is a $\lambda_i^{(k+1)}$-similarity of $X$ such that 
\[f_i^{(k+1)}(x)=x_{i-1}^{(k+1)}, \quad f^{(k+1)}(y)=x_{i+1}^{(k+1)}, \quad \text{and} \quad \lambda_i^{(k+1)}=d(x_{i-1}^{(k+1)},x_{i+1}^{(k+1)})/d(x,y).\]
For use in the sequel, we also define $\lambda^{(1)}_1:=1$. By construction, we note that pairs of consecutive points in the newly created chain are distinct.

We claim that $\Lambda_k:=\max\{\lambda_i^{(k)}\,|\,i=1,\dots,2^k\}$ converges to $0$ as $k\to+\infty$. To see this, for any $k\geq1$ and odd integer $1\leq i <2^{k+1}$, we have 
\[\lambda_i^{(k+1)}=\frac{d(x_{(i-1)/2}^{(k)},x_{(i+1)/2}^{(k)})}{d(x,y)}=\frac{d(f_j^{(k)}(z),f_j^{(k)}(w))}{d(x,y)}.\]
Here $j\in\{(i-1)/2,(i+1)/2\}$ is odd. If $j=(i-1)/2$, then $w=y$. If $j=(i+1)/2$, then $w=x$. Continuing, we find that
\[\frac{d(f_j^{(k)}(z),f_j^{(k)}(w))}{d(x,y)}\leq\Lambda_k\frac{d(z,w)}{d(x,y)}\leq c\Lambda_k.\]
Therefore, $\Lambda_{k+1}\leq c\Lambda_k$. By way of induction, $\Lambda_k\leq c^{k-1}$. Since $c\in(0,1)$, we conclude that $\Lambda_k\to0$ as $k\to+\infty$. 

For $k\in \msf{N}$, write $\mathbf{x}_k$ to denote the collection of points $\{x_0^{(k)},\dots,x_{2^k}^{(k)}\}$ constructed as above. Write $E$ to denote the closure of the union $\cup_{k\in\msf{N}}\mathbf{x}_k\subset X$. In order to reach a contradiction and conclude that $X$ is an ultrametric space, we demonstrate that $E$ is a non-degenerate connected set. Indeed, $x,y\in E$ and $x\not=y$, so $E$ is non-degenerate. To see that $E$ is connected, suppose that $U_1\cup U_2$ is a non-trivial separation of $E$ by disjoint open sets. It is not difficult to verify that $E$ is bounded, and so $E$ is compact. Therefore, we may assume that $\overline{U}_1\cap\overline{U}_2=\emptyset$ and $\dist(U_1,U_2)>\varepsilon$ for some $\varepsilon>0$. However, since $\Lambda_k\to0$, there exists $K\in\msf{N}$ such that, for any $k\geq K$, we have $\Lambda_k<\varepsilon d(x,y)$. Since consecutive points of each $\mathbf{x}_k$ are within distance of $\Lambda_kd(x,y)$ of each other, it follows that $\dist(U_1,U_2)<\varepsilon$. This contradiction demonstrates that $E$ is connected, which in turn contradicts the fact that $X$ is uniformly disconnected. Therefore, $X$ is an ultrametric space.

Given $r>0$ and $x\in X$, since $X$ is an ultrametric space, the ball $B(x;r)$ is closed. Therefore, 
\[s(r)=\max\{d(x,y)\,|\,y\in B(x;r)\}<r.\]
If there exist $\lambda$-dilations of $X$ at $x$ with $\lambda>1$ arbitrarily close to $1$, we contradict the definition of $s(r)$. Therefore, 
\[\lambda_0:=\inf\{\lambda>1\,|\,\exists\,\lambda\text{-dilations of } X\}=\min\{\lambda\,|\,\exists\,\lambda\text{-dilations of } X\}>1.\]
Here the properness of $X$ allows us to replace the infimum by a minimum in the definition of $\lambda_0$. By \rf{L:3pt}, $(1,\lambda_0)\cap\Delta(X)=\emptyset$ and $(\lambda_0^{-1},1)\cap\Delta(X)=\emptyset$. Here 
\[\Delta(X):=\{r\geq0\,|\,\exists\, x,y\in X \text{ such that } d(x,y)=r\}.\]
Since there exists a $\lambda_0$-dilation of $X$, it follows that, for every $k\in \msf{Z}$, we have $(\lambda_0^k,\lambda_0^{k+1})\cap\Delta(X)=\emptyset$. Since $\Delta(X)\not=\emptyset$ and $X$ contains more than one point, we conclude that $\Delta(X)=\{\lambda_0^k\,|\,n\in\msf{Z}\}$.

To conclude, we appeal to the proof of \rf{T:sharp:nonconnected}. Using the methods of this proof, we construct sets
\[\mathcal{X}_i:=h^i(\mathcal{X}_0):=\{h^i\circ f(B(x;1))\,|f\in \Isom(X)\}.\]
Here $h$ is a $\lambda_0^{-1}$-dilation of $X$ at $x$. We then proceed to construct the bijection $\phi:X\to C_N$, where $N$ is the number of pairwise distinct balls of radius $\lambda_0^{-1}$ contained in $B(x;1)$. As in \rf{E:final}, for points $u,v\in X$, we have
\[d(u,v)< \rho_{\lambda_0}(\phi(u),\phi(v))\leq \lambda_0\,d(u,v).\]
Here we use the fact that $X$ is $\alpha$-uniformly disconnected with $\alpha=1$. Since both $d(u,v)$ and $\rho_{\lambda_0}(\phi(u),\phi(v))$ are integer powers of $\lambda_0$, we conclude that $\rho_{\lambda_0}(\phi(u),\phi(v))=\lambda_0d(u,v)$. We conclude that $\phi\circ h:X\to C_N$ is an isometry.
\end{proof}

\begin{proof}[{Proof of Theorem~\ref{T:main:connected:NOT}}]
From \rfs{L:quasi_dilation},~\ref{L:uniformly_perfect}, and~\ref{L:uniformly_disconnected}, we conclude that $X$ is uniformly perfect, uniformly disconnected, proper, and doubling. Here we say that a metric space $X$ is \textit{doubling} provided that there exists a finite constant $D\geq1$ such that any ball of radius $r>0$ in $X$ can be covered by at most $D$ balls of radius $2r$. Given $p\in X$, via \cite[Theorem 1.2]{Heer-QM} we conclude that $\sph_p(X)$ is uniformly disconnected. Via \cite[Theorem 7.1]{Meyer-thesis} we conclude that $\sph_p(X)$ is uniformly perfect. Via \cite[Theorem 1.1]{Heer-QM} (see also \cite[Proposition 3.2.2]{LS15}) we conclude that $\sph_p(X)$ is doubling. In these assertions we are using the facts that the identity map between $X$ and $\sph_p(X)\setminus\{\infty\}$ is strongly quasi-M\"obius and that quasi-sphericalization can be viewed as a special case of quasi-inversion (see \cite[pg.\,847]{BHX-inversions}).

Since $\sph_p(X)$ is compact, doubling, uniformly perfect, and uniformly disconnected, by \cite[Proposition 15.11]{DS-fractals} we conclude that $\sph_p(X)$ is quasi-symmetrically homeomorphic to $(\hat{C}_2,\rho_2)$ (see \cite[Section 2.3]{DS-fractals} and \rf{X:basic}). Since $\sph_p(X)\setminus\{\infty\}$ is quasi-M\"obius homeomorphic to $X$, $\hat{C}_2\setminus\{(1,3,1,\dots)\}$ is quasi-M\"obius equivalent to $C_2$ (see \rf{X:basic}), and all of these spaces are uniformly bi-Lipschitz homogeneous (via \rf{P:invertible_char}), it follows that $X$ is quasi-M\"obius homeomorphic to $(C_2,\rho_2)$. In fact, $X$ is quasi-symmetrically homeomorphic to $(C_2,\rho_2)$. 
\end{proof}

\bibliographystyle{amsalpha}  
\bibliography{bib}            

\newcommand{\etalchar}[1]{$^{#1}$}
\providecommand{\bysame}{\leavevmode\hbox to3em{\hrulefill}\thinspace}
\providecommand{\MR}{\relax\ifhmode\unskip\space\fi MR }
\providecommand{\MRhref}[2]{%
  \href{http://www.ams.org/mathscinet-getitem?mr=#1}{#2}
}
\providecommand{\href}[2]{#2}
\begin{thebibliography}{CKLD{\etalchar{+}}17}

\bibitem[BB05]{BB05}
Zolt\'an~M. Balogh and Stephen~M. Buckley, \emph{Sphericalization and
  flattening}, Conform. Geom. Dyn. \textbf{9} (2005), 76--101. \MR{2179368}

\bibitem[BHX08]{BHX-inversions}
Stephen~M. Buckley, David~A. Herron, and Xiangdong Xie, \emph{Metric space
  inversions, quasihyperbolic distance, and uniform spaces}, Indiana Univ.
  Math. J. \textbf{57} (2008), no.~2, 837--890.

\bibitem[BK02]{BK-spheres}
Mario Bonk and Bruce Kleiner, \emph{Quasisymmetric parametrizations of
  two-dimensional metric spheres}, Invent. Math. \textbf{150} (2002), no.~1,
  127--183.

\bibitem[Bou95]{Bourdon95}
Marc Bourdon, \emph{Structure conforme au bord et flot g\'eod\'esique d'un
  {${\rm CAT}(-1)$}-espace}, Enseign. Math. (2) \textbf{41} (1995), no.~1-2,
  63--102. \MR{1341941}

\bibitem[Bou96]{Bourdon}
\bysame, \emph{Sur le birapport au bord des {${\rm CAT}(-1)$}-espaces}, Inst.
  Hautes \'Etudes Sci. Publ. Math. (1996), no.~83, 95--104.

\bibitem[BS07]{MR2327160}
Sergei Buyalo and Viktor Schroeder, \emph{Elements of asymptotic geometry}, EMS
  Monographs in Mathematics, European Mathematical Society (EMS), Z\"{u}rich,
  2007. \MR{2327160}

\bibitem[BS14]{Buyalo_Schroeder14}
\bysame, \emph{M\"obius characterization of the boundary at infinity of rank
  one symmetric spaces}, Geom. Dedicata \textbf{172} (2014), 1--45.
  \MR{3253769}

\bibitem[BS15]{Buyalo_Schroeder15}
\bysame, \emph{Incidence axioms for the boundary at infinity of complex
  hyperbolic spaces}, Anal. Geom. Metr. Spaces \textbf{3} (2015), 244--277.
  \MR{3399374}

\bibitem[BS17]{BS17-trees}
Jonas Beyrer and Victor Schroeder, \emph{Trees and ultrametric {M}\"{o}bius
  structures}, p-Adic Numbers Ultrametric Anal. Appl. \textbf{9} (2017), no.~4,
  247--256.

\bibitem[CDKR91]{CDKR-Heisenberg}
Michael Cowling, Anthony~H. Dooley, Adam Kor{\'a}nyi, and Fulvio Ricci,
  \emph{{$H$}-type groups and {I}wasawa decompositions}, Adv. Math. \textbf{87}
  (1991), no.~1, 1--41.

\bibitem[CDKR98]{CDKR98}
\bysame, \emph{An approach to symmetric spaces of rank one via groups of
  {H}eisenberg type}, J. Geom. Anal. \textbf{8} (1998), no.~2, 199--237.

\bibitem[CDPT07]{CDPT07}
Luca Capogna, Donatella Danielli, Scott~D. Pauls, and Jeremy~T. Tyson, \emph{An
  introduction to the {H}eisenberg group and the sub-{R}iemannian isoperimetric
  problem}, Progress in Mathematics, vol. 259, Birkh\"auser Verlag, Basel,
  2007.

\bibitem[CKLD{\etalchar{+}}17]{CKLDGO17}
Michael~G. Cowling, Ville Kivioja, Enrico Le~Donne, Sebastiano~Nicolussi Golo,
  and Alessandro Ottazzi, \emph{From homogeneous metric spaces to lie groups},
  arXiv preprint (2017).

\bibitem[dC18]{Cornulier18}
Yves de~Cornulier, \emph{On the quasi-isometric classification of locally
  compact groups}, New directions in locally compact groups, London
  {M}athematical {S}ociety {L}ecture {N}ote {S}eries, vol. 447, Cambridge
  {U}niversity {P}ress, 2018, pp.~275--342.

\bibitem[DCL17]{DCL17}
Estibalitz Durand-Cartagena and Xining Li, \emph{Preservation of bounded
  geometry under sphericalization and flattening: quasiconvexity and
  {$\infty$}-{P}oincar\'e inequality}, Ann. Acad. Sci. Fenn. Math. \textbf{42}
  (2017), no.~1, 303--324.

\bibitem[DS97]{DS-fractals}
Guy David and Stephen Semmes, \emph{Fractured fractals and broken dreams},
  Oxford Lecture Series in Mathematics and its Applications, vol.~7, The
  Clarendon Press, Oxford University Press, New York, 1997, Self-similar
  geometry through metric and measure.

\bibitem[FH10]{FH-blh}
David~M. Freeman and David~A. Herron, \emph{Bilipschitz homogeneity and inner
  diameter distance}, J. Anal. Math. \textbf{111} (2010), 1--46.

\bibitem[FK71]{FK70}
S.~P. Franklin and G.~V. Krishnarao, \emph{``{O}n the topological
  characterization of the real line'': {A}n addendum}, J. London Math. Soc. (2)
  \textbf{3} (1971), 392.

\bibitem[FLS07a]{Foertsch_Lytchak_Schroeder_ERR}
Thomas Foertsch, Alexander Lytchak, and Viktor Schroeder, \emph{Erratum to:
  ``{N}onpositive curvature and the {P}tolemy inequality'' [{I}nt. {M}ath.
  {R}es. {N}ot. {IMRN} {\bf 2007}, no. 22, {A}rt. {ID} rnm100, 15 pp.;
  mr2376212]}, Int. Math. Res. Not. IMRN (2007), no.~24, Art. ID rnm160, 1.
  \MR{2377019}

\bibitem[FLS07b]{FLS07}
\bysame, \emph{Nonpositive curvature and the {P}tolemy inequality}, Int. Math.
  Res. Not. IMRN (2007), no.~22, Art. ID rnm100, 15.

\bibitem[Fre10]{Freeman-blh}
David~M. Freeman, \emph{Unbounded bilipschitz homogeneous {J}ordan curves},
  Ann. Acad. Sci. Fenn. Math. \textbf{36} (2010), no.~1, 81--99.

\bibitem[Fre12]{Freeman-iiblh}
\bysame, \emph{Inversion invariant bilipschitz homogeneity}, Michigan Math. J.
  \textbf{61} (2012), no.~2, 415--430.

\bibitem[Fre14]{Freeman-invertible}
\bysame, \emph{Invertible {C}arnot groups}, Anal. Geom. Metr. Spaces \textbf{2}
  (2014), 248--257.

\bibitem[FS11]{Foertsch-Schroeder}
Thomas Foertsch and Viktor Schroeder, \emph{Hyperbolicity, {${\rm
  CAT}(-1)$}-spaces and the {P}tolemy inequality}, Math. Ann. \textbf{350}
  (2011), no.~2, 339--356.

\bibitem[FS12]{Foertsch_Schroeder12}
\bysame, \emph{M\"obius characterization of hemispheres}, Arch. Math. (Basel)
  \textbf{99} (2012), no.~1, 81--89. \MR{2948664}

\bibitem[GH98]{GH-ar}
Manouchehr Ghamsari and David~A. Herron, \emph{Higher dimensional {A}hlfors
  regular sets and chordarc curves in {$\bold R\sp n$}}, Rocky Mountain J.
  Math. \textbf{28} (1998), no.~1, 191--222.

\bibitem[Ham89]{MR1016663}
Ursula Hamenst\"{a}dt, \emph{A new description of the {B}owen-{M}argulis
  measure}, Ergodic Theory Dynam. Systems \textbf{9} (1989), no.~3, 455--464.
  \MR{1016663}

\bibitem[Ham91]{MR1114460}
\bysame, \emph{A geometric characterization of negatively curved locally
  symmetric spaces}, J. Differential Geom. \textbf{34} (1991), no.~1, 193--221.
  \MR{1114460}

\bibitem[HB99]{HB99}
B.~Honari and Y.~Bahrampour, \emph{Cut-point spaces}, Proc. Amer. Math. Soc.
  \textbf{127} (1999), no.~9, 2797--2803.

\bibitem[Hee17]{Heer-QM}
Loreno Heer, \emph{Some invariant properties of quasi-{M}\"obius maps}, Anal.
  Geom. Metr. Spaces \textbf{5} (2017), 69--77.

\bibitem[Hel01]{Helgason}
Sigurdur Helgason, \emph{Differential geometry, {L}ie groups, and symmetric
  spaces}, Graduate Studies in Mathematics, vol.~34, American Mathematical
  Society, Providence, RI, 2001, Corrected reprint of the 1978 original.

\bibitem[Her16]{Herron-Gromov}
David~A. Herron, \emph{Gromov--{H}ausdorff distance for pointed metric spaces},
  J. Anal. \textbf{24} (2016), no.~1, 1--38.

\bibitem[HM99]{HM-blh}
David~A. Herron and Volker Mayer, \emph{Bi-{L}ipschitz group actions and
  homogeneous {J}ordan curves}, Illinois J. Math. \textbf{43} (1999), no.~4,
  770--792.

\bibitem[HS97]{HS-questions}
Juha Heinonen and Stephen Semmes, \emph{Thirty-three yes or no questions about
  mappings, measures, and metrics}, Conform. Geom. Dyn. \textbf{1} (1997),
  1--12 (electronic).

\bibitem[Kin15]{Kinneberg-fractal}
Kyle Kinneberg, \emph{Rigidity for quasi-{M}\"obius actions on fractal metric
  spaces}, J. Differential Geom. \textbf{100} (2015), no.~2, 349--388.

\bibitem[KLD16]{KL16}
Kyle Kinneberg and Enrico Le~Donne, \emph{A metric characterization of
  snowflakes of {E}uclidean spaces}, Ann. Sc. Norm. Super. Pisa Cl. Sci. (5)
  \textbf{16} (2016), no.~2, 469--480.

\bibitem[Kra03]{Kramer-transitive}
Linus Kramer, \emph{Two-transitive {L}ie groups}, J. Reine Angew. Math.
  \textbf{563} (2003), 83--113.

\bibitem[LD10]{LeDonne-blh}
Enrico Le~Donne, \emph{Doubling property for bilipschitz homogeneous geodesic
  surfaces}, Journal of Geometric Analysis (2010), 1--24,
  10.1007/s12220-010-9167-7.

\bibitem[LD11]{LeDonne-geodesic}
\bysame, \emph{Geodesic manifolds with a transitive subset of smooth
  bi{L}ipschitz maps}, Groups Geom. Dyn. \textbf{5} (2011), no.~3, 567--602.

\bibitem[LS15]{LS15}
Xining Li and Nageswari Shanmugalingam, \emph{Preservation of bounded geometry
  under sphericalization and flattening}, Indiana Univ. Math. J. \textbf{64}
  (2015), no.~5, 1303--1341.

\bibitem[Mey09]{Meyer-thesis}
Johannes Bj\o rn~Thomas Meyer, \emph{Ph.{D}. thesis}, University of Z\"urich,
  Faculty of Science (2009).

\bibitem[MT10]{MT-cfl}
John~M. Mackay and Jeremy~T. Tyson, \emph{Conformal dimension}, University
  Lecture Series, vol.~54, American Mathematical Society, Providence, RI, 2010,
  Theory and application.

\bibitem[MZ74]{mz}
Deane Montgomery and Leo Zippin, \emph{Topological transformation groups},
  Robert E. Krieger Publishing Co., Huntington, N.Y., 1974, Reprint of the 1955
  original.

\bibitem[Pla13]{Platis-ptolemy}
Ioannis~D. Platis, \emph{Cross-ratios and the {P}tolemaean inequality in
  boundaries of symmetric spaces of rank 1}, Geom. Dedicata (2013).

\bibitem[PS17]{PS17}
I.~D. Platis and V.~Schroeder, \emph{M\"obius rigidity of invariant metrics in
  boundaries of symmetric spaces of rank-1}, Monatsh. Math. \textbf{183}
  (2017), no.~2, 357--373.

\bibitem[Sie86]{Siebert86}
Eberhard Siebert, \emph{Contractive automorphisms on locally compact groups},
  Math. Z. \textbf{191} (1986), no.~1, 73--90.

\bibitem[TW05]{TW-snowflakes}
Jeremy~T. Tyson and Jang-Mei Wu, \emph{Characterizations of snowflake metric
  spaces}, Ann. Acad. Sci. Fenn. Math. \textbf{30} (2005), no.~2, 313--336.

\bibitem[V{\"a}i85]{Vaisala-qm}
Jussi V{\"a}is{\"a}l{\"a}, \emph{Quasi-{M}\"obius maps}, J. Analyse Math.
  \textbf{44} (1984/85), 218--234.

\bibitem[Xie13]{Xie-non-rigid}
Xiangdong Xie, \emph{Quasiconformal maps on non-rigid {C}arnot groups},
  preprint (2013), arXiv:1308.3031.

\end{thebibliography}

\end{document}